\documentclass[10pt]{amsart}
\usepackage[T1]{fontenc}
\usepackage[utf8]{inputenc}
\usepackage[french]{babel}
\usepackage{mathrsfs}
\usepackage[all]{xy}
\usepackage{tikz}
\usepackage{amsmath, amsthm, amssymb, amscd}
\usepackage{txfonts}
\usepackage{hyperref}
\hypersetup{colorlinks=true, linkcolor=black, citecolor=blue}

\baselineskip 18pt \textwidth 16cm  \theoremstyle{plain}

\newtheorem*{thm*}{Théorème}
\newtheorem*{rem*}{Remarque}
\newtheorem*{prop*}{Proposition}
\newtheorem{lemme}{Lemme}[section]
\newtheorem{prop}[lemme]{Proposition}
\newtheorem{rem}[lemme]{Remarque}

\newtheorem{thm}[lemme]{Théorème}
\newtheorem{defn}[lemme]{Définition}
\newtheorem{cor}[lemme]{Corollaire}

\newtheorem*{conjecture*}{Conjecture}

\newtheorem*{theorem*}{Theorem}

\newcommand{\End}{\operatorname{End}}
\newcommand{\Hom}{\operatorname{Hom}}

\newcommand{\Ker}{\operatorname{Ker}}

\newcommand{\Irr}{\operatorname{Irr}}

\newcommand{\Ind}{\operatorname{Ind}}
\newcommand{\Rep}{\operatorname{Rep}}
\newcommand{\Res}{\operatorname{Res}}

\newcommand{\C}{\mathbb C}

\newcommand{\Z}{\mathbb Z}

\newcommand{\Gal}{\operatorname{Gal}}

\newcommand{\Cent}{\operatorname{Cent}}

\newcommand{\id}{\operatorname{Id}}
\newcommand{\nnn}{\operatorname{N}}

\newcommand{\stab}{\operatorname{Stab}}
\newcommand{\Stab}{\operatorname{Stab}}
\newcommand{\supp}{\operatorname{supp}}
\newcommand{\tr}{\operatorname{Tr}}

\begin{document}
\title{Correspondances de Howe pour les groupes des similitudes}

\date{\today}
\author{chun-hui wang }
\address{chun-hui wang. Université Paris-Sud, Bâtiment 425 91405 Orsay Cedex, France }
\email{chun-hui.wang@math.u-psud.fr}
\maketitle

\setcounter{tocdepth}{1}

\tableofcontents
\section*{Introduction}
Cet article  est pour but d'examiner comment on peut généraliser la correspondance de Howe (qui est déjà prouvée  pour un corps $p$-adique $F$ sauf si sa caractéristique résiduelle  vaut $2$ ) aux groupes des similitudes.\\

Certains aspects de la correspondance de Howe, pour des  groupes des similitudes, ont déjà été étudiés. 
Nous rappelons quelques résultats pour un corps global. En 1972, Shimizu.H a étudié les groupes $(GL_2, D^{\times})$ [cf. \cite{Shim}]
 et son résultat permet de réaliser la correspondance de Jacquet-Langlands en utilisant la correspondance de Howe.
 Puis, Waldspurger  a déterminé complètement la correspondance de Howe pour les groupes $(\widehat{SL_2}, PGL_2)$[cf. \cite{Wald2}]. 
Pour des résultats du m\^eme genre, nous renvoyons  le lecteur  à l'article de I.I.Piateski-Shapiro [cf. \cite{PS} ] 
qui traite   les groupes $(\widehat{SL_2}, PGSp_4)$, à l'article de M.Harris et S.Kudla [cf. \cite{HK}] 
pour les groupes $(GSp_2, GO_4)$, ou bien aux autres articles de la bibliographie.

Cependant, notre article sera consacré  à une étude analogue sur un corps $p$-adique. 
Dans la Thèse de L.Barthel [cf. \cite{Bar2}], elle a construit le groupe métaplectique $\widetilde{GSp}$,
 et a défini la représentation de Weil pour $\widetilde{GSp}$. 
En particulier, elle aussi mis en évidence  les difficultés à généraliser la correspondance de Howe aux groupes des similitudes. 
Une difficulté  est que pour une paire réductive duale $(H_1,H_2)$ de $GSp$, 
leurs images réciproques $(\widetilde{H_1}, \widetilde{H_2})$ dans $\widetilde{GSp}$ ne commutent pas en général.
 Ensuite dans \cite{Rob1}, B.Roberts a généralisé définitivement la correspondance de Howe pour les groupes $(GSp, GO)$. 
Ses résultats doivent se demander  que les représentations lisses irréductibles de $GSp$ (resp. $GO$) se restreignent à $Sp$ (resp. $O$) 
sans multiplicité. Ensuite W.T.Gan et S.Tantono ont fait leurs formes intérieures dans \cite{GT}, et ont aussi vérifié que les conditions 
de multiplicité $1$ mentionnées ci-dessus ne  étaient pas nécessaires. Gr\^ace à  une analyse attentive de leurs travaux, 
nous généralisons ces résultats à des paires plus nombreuses dans cet article.\\

Cet article consiste en deux parties. Dans les sections §\ref{leplusquotient} à §\ref{lesrepresentationdebigrapheforte}, 
nous donnons les théorèmes principaux pour les représentations  de graphe forte et de bigraphe forte. 
Dans les sections §\ref{howeduale} à §\ref{appendice2}, en utilisant les résultats principaux énoncés dans la première partie, 
nous généralisons la correspondance de Howe aux groupes des similitudes.

Dans cet article, nous utilisons librement les notions et notations de la théorie des représentations lisses de groupes 
localement compacts totalement discontinus [cf. \cite{BernZ}].  Nous aussi citons librement les résultats de la représntation de Weil 
sur un corps $p$-adique [cf. \cite{MVW} et \cite{Kud2}].\\

\smallskip
\noindent \textbf{Remerciements.}
Ce travail s'est accompli sous la direction de Guy Henniart.
Je tiens le remercier de m'avoir proposé ce sujet, et pour ses remarques très nécessaires. 
Je remercie également Brooks Boberts et Colette Moeglin pour  leurs encouragements.  Je voudrais remercier aussi Atsushi Ichino pour ses remarques. 

\section{Le plus grand quotient du groupe }\label{leplusquotient}
 Soient $G$ un groupe localement compact totalement discontinu, $(\rho, V)$ une représentation lisse de $G$.
 Si $(\pi, W)$  est une représentation lisse irréductible de $G$, on note $V_{\pi}$ le plus grand quotient $\pi$-isotypique de $V$. \\

Posant
$$V[\pi]=\cap_f \ker(f) \textrm{ où } f \textrm{ parcourt } \Hom_G(V, W),$$

on a
$$V_{\pi}= V/{V[\pi]}.$$

Il satisfait à la propriété universelle suivante: \\

l'application de quotient $V \longrightarrow V_{\pi}$ induit une isomorphisme
$$ \Hom_G(V, W) \simeq \Hom_G(V_{\pi}, W)$$
et
$$\Hom_G(V, W)= 0 \textrm{ si et seulement si } V_{\pi}= 0.$$

Dans le cas particulier où $\pi=1_G$, $V_{\pi}$ n'est pas autre que l'espace $V_G$ des coinvariants de $G$ dans $V$, c'est-à-dire que le quotient de $V$ par le sous-espace $V[G]=V[1_G]$ engendré par les éléments $\rho(g)v-v$ pour $v$ parcourant $V$ et $g$ parcourant $G$.

Soient $G$ un groupe localement compact totalement discontinu, $(\rho, V)$ une représentation lisse de $G$. On notera
$$\mathcal{R}_G(\rho)=\{ \pi \in \Irr(G)| \Hom_G(\rho, \pi) \neq 0\}$$
\begin{prop}\label{quotientzero}
Supposons que $(\rho, V)$ est une représentation de type fini. Alors $(\rho, V)=0$ si et seulement si $\mathcal{R}_G(\rho)=0$.
\end{prop}
\begin{proof}
Ceci découle de [\cite{BernZ}, Page 16, Lemma].
\end{proof}
\begin{prop}\label{typefini}
Soient $G$ un groupe localement compact totalement discontinu, $H$ un sous-groupe fermé de $G$.
\begin{itemize}
\item[(1)] Si $H$ est aussi un sous-groupe ouvert  de $G$ et que $\pi$ est une représentation lisse de type fini de $H$, alors $c\!-\!\Ind_H^G\pi$ est une représentation lisse de type fini de $G$.
\item[(2)] Si $G/H$ est compact et que  $(\pi, V)$ est une représentation lisse de  type fini de $G$, alors $\Res_H^G \pi$ est une représentation lisse de type fini de $H$.
\end{itemize}
\end{prop}
\begin{proof}
(1) Comme $H$ est un sous-groupe ouvert de $G$, l'induite compacte $c-\Ind_H^G \pi$ s'indentifie à l'induite ordinaire $\C[G] \otimes_{\C[H]} \pi$; le résultat est alors clair. \\
(2) Soient $\{v_1, \cdots, v_n\}$ un ensemble de vecteurs engendrant $V$ comme représentation de $G$ et $K$ un sous-groupe ouvert de $G$, tels que
$$ \epsilon_K\star v_i=v_i \textrm{ pour chaque }i.$$
L'image de $K$ dans $H\setminus G$ est ouverte; Comme $H\setminus G$ est compact, il existe des éléments en nombre fini, $g_1, \cdots, g_m \in G$, tels que $G$ soit l'union des $Hg_i K$ pour $i=1, \cdots, m$.
 Il en résulte que $\Res_H^G\pi$ est engendré par les éléments $\pi(g_i)v_j$ pour $ i=1, \cdots, m, j=1, \cdots, n$ comme $H$-module.
\end{proof}

Soit $\rho$ une représentation lisse du groupe localement compact totalement discontinu $G$, on définit une application $m_G(\rho,-)$   sur  $\Irr(G)$ de la façon suivante:
$$m_G(\rho, \pi)= \textrm{ le cardinal de la dimension de } \Hom_G(\rho, \pi) \textrm{ pour } \pi \in \Irr(G).$$
Ainsi $\mathcal{R}_G(\rho)$ est le support de cette application.

\begin{defn}\label{quotient}
\begin{itemize}
\item[(1)] Si $m_G(\rho, \pi)$ est  fini pour tout $\pi \in \Irr(G)$, on dit que $\rho$ est \textbf{une représentation de quotient admissible}.
\item[(2)] Si $m_G(\rho, \pi)$ est égal à $0$ ou $1$ pour tout $\pi \in \Irr(G)$, on dit que $\rho$ est \textbf{une représentation de quotient sans multiplicité}.
\item[(3)] Si $\rho$ est une représentation de quotient admissible et que le support de l'application $m_G(\rho, -)$ est un seul élément  $\pi$, on dit que $\rho$ est \textbf{une représentation de quotient de Langlands} et que $\pi$ est \textbf{le quotient de Langlands de $\rho$}.
\end{itemize}
\end{defn}
\begin{prop}\label{typeimpliquequotientadmissible}
Soit $(\rho, V)$ une représentation lisse de type fini du groupe  localement compact totalement discontinu $G$. Supposons que toute représentation irréductible de $G$ est admissible.  Alors $\rho$ est une représentation de quotient admissible.
\end{prop}
\begin{proof}
Supposons que $V$ est engendré par les éléments $v_1, \cdots, v_n$ comme $G$-module. Prenons un élément $(\pi, W)$ de $\Irr(G)$. Soit $f$ un élément de $\Hom_G(V,W)$, on a
$$f(v)=f(\sum_{i=1}^n \sum_{j=1}^m \lambda_{ij}\pi(g_j)v_i)=\sum_{i=1}^n \sum_{j=1}^m  \lambda_{ij}\pi(g_j) f(v_i).$$
C'est-à-dire que l'application sera déterminée par ses valeurs en points $v_1, \cdots, v_n$.
Soit $K$ un groupe ouvert fixant tous éléments $v_1, \dots, v_n$;  alors $f(v_i)$ est un élément de $W^K$ pour tout $i$ et par hypothèse, la dimension de $W^K$ est finie.  Donc

$$\dim \Hom_G(V, W) \leq n \dim W^K< \infty.$$ 
\end{proof}
\section{La représentation de quotient admissible}\label{quotientadmissible}
Dans cette sous-section, on suppose que  $G$ est un groupe réductif $p$-adique et que $P$ est
un sous-groupe parabolique de $G$, et que $P=MN$ est une décomposition de Levi.\\

On rappel deux functeurs  non normalisés:\\
$$\Ind_P^G: \Rep(M) \longrightarrow \Rep(G) \textrm{ le foncteur d'induction parabolique},$$
$$J_N: \Rep(G) \longrightarrow \Rep(M) \textrm{ le foncteur de Jacquet}.$$
\begin{thm}\label{lesfoncteurstypefini}
Les foncteurs $(\Ind_P^G, J_N)$ préservent la classe des représentations lisses de type fini.
\end{thm}
\begin{proof}
(1) Soit $(\pi, V)$ une représentation lisse de type fini du groupe $G$, comme $G/P$ est compact, Par la Proposition \ref{typefini} (2), $\Res_P^G \pi$ est aussi de type fini. Soit l'espace $V$ qui est engendré par les éléments $v_1, \cdots, v_n$, fixés par un sous-groupe ouvert compact $K$ de $P$. L'espace de $J_N(\pi)$ est l'espace des coinvariants de $N$ dans $V$; Il est de type fini comme représentation de $P$, mais $N$ agit trivialement donc il est de type fini comme représentation de $M$.\\
(2) C'est un résultat difficile de Bernstein (voir [\cite{Rena}, Page 215, Théorème]).
\end{proof}

\begin{lemme}\label{longueurfiniequotient}
Soient $(\pi, V)$ une représentation lisse de quotient admissible de $G$ et $(\rho, W)$ une représentation lisse de longueur finie de $G$.  Alors
$$\dim_{\C} \Hom_G(\pi, \rho) \textrm{ est finie }.$$
\end{lemme}
\begin{proof}
Si
$$ 0=W_0 \leqq W_1 \leqq \cdots \leqq W_s=W$$
est une filtration de $W$ par  des sous  $G$-modules telle que $W_i/ {W_{i-1}}$ est une représentation irréductible de $G$ pour $i=1, \cdots, s$. On a une suite exacte $$1 \longrightarrow W_{s-1} \longrightarrow W \stackrel{p}{\longrightarrow} W/{W_{s-1}} \longrightarrow 1.$$
Comme le foncteur $\Hom_G(V, -)$ est exact à gauche, on obtient
$$1 \longrightarrow \Hom_G(V, W_{s-1}) \longrightarrow \Hom_G(V,W) \longrightarrow \Hom_G(V, W/{W_{s-1}}).$$
Cela implique que
$$\dim_{\C}\Hom_G(V, W) \leq \dim_{\C} \Hom_G(V, W_{s-1})+ \dim_{\C}\Hom_G(V, W/{W_{s-1}}).$$
Par récurrence sur $s$, on trouve
$$\dim_{\C} \Hom_G(V,W) \leq \sum_{i=1}^s \dim_{\C}\Hom_G(V, W_i/{W_{i-1}})< \infty \textrm{ par hypothèse}.$$
\end{proof}
\begin{lemme}\label{VNetVN}
Soit $(\pi, V)$ une représentation lisse du groupe $N$; alors l'application canonique
$$V^N \stackrel{p_N}{\longrightarrow} V_N$$ est injective.
\end{lemme}
\begin{proof}
Soit $p_N(v)=0$ pour quelque $v$ de $V^N$. Par définition, on a $v\in V[N]$. Comme $N$ est une union des sous-groupes ouverts  compacts de lui-m\^eme,  on a
$$ V[N]=\cup_K V[K] \textrm{ où } K \textrm{ parcourt les sous-groupes ouverts  compacts de } N.$$
Il existe un groupe ouvert compact $K_v$ de $N$ tel que
 $$v\in V[K_v] \textrm{ i.e } v=\sum_{i=1}^n \pi(k_i)v_i -v_i \textrm{ pour quelques } v_i \in V, k_i \in K_v.$$
 Choisissons un sous-groupe ouvert  compact $K$ de $G$ qui contient chacun des $k_i$. Alors $v=\pi(\epsilon_K)v= \pi(\epsilon_K)(\sum_{i=1}^n \pi(k_i)v_i -v_i) =0$.
\end{proof}
Rappelons le théorème de Howe et celui de Jacquet et Harish-Chandra [cf. \cite{BernZ}. p.37 et \cite{Bern3} Theorem ].
\begin{thm}[Howe]\label{longueurfiniedeHowe}
Soit $(\pi, V)$ une représentation lisse de $G$; alors $(\pi, V)$ est de longueur finie ssi elle est de type fini et admissible.
\end{thm}

\begin{thm}[Harish-Chandra, Jacquet]\label{Gadmissible}
Toute représentation lisse irréductible du groupe $G$ est admissible.
\end{thm}

\begin{cor}\label{longueurfiniedeJacquet}
Le foncteur de Jacquet préserve les représentations de longueur finie.
\end{cor}
\begin{proof}
C'est une conséquence du théorème de Howe, du Théorème \ref{lesfoncteurstypefini} et du Théorème \ref{Gadmissible}.
\end{proof}

\begin{lemme}\label{MPquotientadmissible}
Soient $(\pi, V)$ une représentation lisse de quotient admissible du groupe $M$, $(\rho, W)$ une représentation irréductible du groupe $G$. Alors
$$\dim_{\C} \Hom_P(V, W) < \infty.$$
\end{lemme}
\begin{proof}
Supposons $T\in \Hom_P(V,W)$. On a $T(v)=T(nv)=nT(v)$ pour $n\in N, v\in V$. Composant par l'application de $W^N$ dans $W_N$, qui est injective ( Lemme \ref{VNetVN}), on trouve
$$\Hom_P(V, W) \hookrightarrow \Hom_P(V, W_N) \simeq \Hom_M(V, W_N).$$
Comme $W_N$ est une représentation du groupe $M$ de longueur finie ( Corollaire \ref{longueurfiniedeJacquet}), en utilisant le Lemme \ref{longueurfiniequotient},  on obtient le résultat.
\end{proof}

\begin{prop}\label{quotientIndJn}
Les foncteurs $(\Ind_P^G, J_N)$ préservent la classe des représentations lisses de quotient admissible.
\end{prop}
\begin{proof}
(i) Pour $\Ind_P^G$: soient $(\pi, V)$ une représentation lisse de quotient admissible du groupe $M$ et $\rho \in \Irr(G)$. On note $\Delta_P$ le fonction unimodulaire de $P$[cf. \cite{BernZ}, Page 10]. Donc on  a
$\Hom_G(\Ind_P^G \pi, \rho) \simeq \Hom_P(\Delta_P^{-1}\pi, \Res_P^G \rho),$ qui sont de dimension finie par le Lemme \ref{MPquotientadmissible}.\\
(ii) Soient $(\pi, V)$ une représentation de quotient admissible du groupe $G$
et $(\rho, W)\in \Irr(M)$. Par la réciprocité de Frobenius, on a
$$\Hom_M(J_N(\pi), \rho) \simeq \Hom_G( \pi, \Ind_P^G\rho).$$
Comme $G/P$ est compact, on sait que $\Ind_P^G \rho$ est une représentation lisse du groupe $G$ de longueur finie.  Comme $\pi$ est une représentation de quotient admissible, d'après le Lemme \ref{longueurfiniequotient}, on obtient
$$\dim_{\C}\Hom_G(\pi, \Ind_P^G\rho) < \infty,$$
ceci montre l'assertion!
\end{proof}
\section{Les quotients  pour le produit de deux groupes}\label{lesquotientsdedeuxgroupes}
Soient $G_1, G_2$ deux groupes localement compacts totalement discontinus, $(\pi, S)$ une représentation lisse de $G_1\times G_2$. 
Rappelons les définitions des ensembles $\mathcal{R}_{G_1\times G_2}(S), \mathcal{R}_{G_1}(S), \mathcal{R}_{G_2}(S)$ dans §\ref{quotientadmissible}.
\begin{lemme}[Waldspurger]\label{waldspurger1}
Soient $(\pi_1, V_1)$ une représentation admissible irréductible de $G_1$,
$(\pi_2, V_2)$ une représentation lisse de $G_2$, $V$ un sous-espace $G_1\times G_2$-invariant de $V_1\otimes V_2$. Alors il existe un sous-espace $V_2'$ de $V_2$, invariant par $G_2$, tel que $V=V_1\otimes V_2'$.
\end{lemme}
\begin{proof}
Nous suivons la démonstration de [\cite{MVW}, Pages 45-46]. On  note
$$V_2'= \{v_2'\in V_2| \textrm{ il existe } 0 \neq v_1\in V_1 \textrm{ tel que } v_1\otimes v_2' \in V\},$$
qui est un $\C$-espace $G_2$-invariant. En effet, prenons $v_1\otimes v_2', u_1 \otimes u_2' \in V$ avec $v_1, u_1 \neq 0$. 
Comme $V_1$ est irréductible, il existe $f\in \mathcal{H}(G_1)$ tel que $u_1=\pi_1(f) v_1$. 
Soit $K_2'$ un sous-groupe ouvert compact de $G_2$ fixant $v_2'$ et $u_2'$,
on a, pour $\alpha, \beta \in \C$,
 $$ \alpha \pi_1(f) \otimes \pi_2(\epsilon_{K_2'})(v_1 \otimes v_2') + \beta u_1 \otimes u_2'= u_1\otimes (\alpha v_2' + \beta u_2') \in V.$$
 Ainsi $\alpha v_2' + \beta u_2' \in V_2'$. \'Evidement, $V_1\otimes V_2'$ est un sous-espace de $V$, 
et on veut montrer que c'est tout. Si $V=0$, ceci  termine la démonstration. 
Sinon, soit $0 \neq v=\sum_{i=1}^n u_i \otimes v_i \in V$ avec $0 \neq u_i \in V_1, 0 \neq v_i \in V_2$. 
On suppose $u_1, \cdots,u_n \in V_1^{K_1}$ pour quelque  sous-groupe ouvert  compact $K_1$ de $G_1$. 
Comme $V_1$ est admissible irréductible et on sait que $V_1^{K_1}$ est un $\mathcal{H}(G,K_1)$-module irréductible de dimension finie. 
\underline{Donc} $\mathcal{H}(G,K_1) \longrightarrow \End_{\C}(V_1^{K_1})$ est surjectif,
 et on peut trouver des éléments $\epsilon_1,  \cdots, \epsilon_n$ de $\mathcal{H}(G,K_1)$, tels que $\pi_1(\epsilon_i) u_j= \delta_{ij} u_j$.
 Soit $K_2$ un sous-groupe ouvert  compact de $G_2$, qui satisfait à $\pi_2(\epsilon_{K_2}) v_i=v_i$;
alors $\pi_1(\epsilon_i) \otimes \pi_2(\epsilon_{K_2}) v= u_i \otimes v_i \in V$,
 il en résulte que $v_i \in V_2'$ pour $i=1, \cdots, n$, enfin on trouve $v\in V_1\otimes V_2'$.
\end{proof}
\begin{lemme}[Waldspurger]\label{waldspurger2}
Soient  $(\pi_1, V_1)$ une représentation admissible irréductible de $G_1$, $(\pi, V)$ une représentation lisse de $G_1\times G_2$. Supposons que
 $$\cap \ker(f)=\{0\} \textrm{ où } f \textrm{ parcourt } \Hom_{G_1}(V, V_1).$$
Alors il existe une représentation lisse $(\pi_2', V_2')$ de $G_2$, unique à isomorphisme près, telle que $\pi$ soit isomorphe au produit tensoriel externe $\pi_1 \otimes \pi_2'$.
\end{lemme}
\begin{proof}
Nous suivons la démonstration de [\cite{MVW}, Pages 45-46].\\
(1) L'unicité. Soient $V\simeq V_1\otimes V_2 \simeq V_1\otimes V_2'$ pour deux représentations lisses $(\pi_2, V_2)$, $(\pi_2', V_2')$ du groupe $G_2$, on a $$ V_2 \simeq (\check{V_1} \otimes V_1)_{G_1} \otimes V_2 \simeq (\check{V_1} \otimes (V_1 \otimes V_2))_{G_1} \simeq (\check{V_1} \otimes (V_1\otimes V_2'))_{G_1} \simeq V_2'$$ comme représentation de $G_2$.\\
(2) L'existence.\\
(i) D'abord, nous avons l'application bilinéaire
$$V\times \Hom_{G_1}(V, V_1) \longrightarrow V_1;$$
$$(v, f) \longmapsto  f(v).$$
Comme $\cap_f \Ker(f)=0$ où $f$ parcourt $\Hom_{G_1}(V, V_1)$, on obtient un morphisme injectif:
$$ V \stackrel{\star}{\hookrightarrow } \Hom_{\C}(\Hom_{G_1}(V, V_1), V_1).   $$
Pour utiliser le résultat du Lemme \ref{waldspurger1} ci-dessus, on va montrer que son image est dans $\Hom_{\C}(\Hom_{G_1}(V, V_1), \C) \otimes_{\C}V_1$.\\
(ii) Vérifions cela. Prenons un élément $0\neq v\in V$ fixé par un sous-groupe ouvert  compact $K_1\times K_2$ de $G_1\times G_2$. Supposons que $K_1$ est suffisamment petit  pour que $V_1^{K_1} \neq 0$,  qui est alors un $\mathcal{H}(G_1, K_1)$-module irréductible de dimension finie. Choisissons
 une base $v_1, \cdots, v_n$ de $V_1^{K_1}$. Pour $f\in \Hom_{G_1}(V, V_1)$, on a
$$v^{\star}(f)=f(v)=f(\pi|_{G_1}(\epsilon_{K_1})v)=\epsilon_{K_1}\star f(v)=\sum_{i=1}^n f_i(v)v_i$$ pour quelques $f_i(v)\in  \C$. On définit ainsi une application $$c_{i,v}^{\star}: \Hom_{G_1}(V, V_1) \longrightarrow \C;$$
$$f \longmapsto  f_i(v).$$
On  vérifie que $c_{i,v}^{\star}$ est $\C$-linéaire et on a alors $v^{\star}=\sum_{i=1}^n c_{i,v}^{\star} \otimes v_i\in \Hom_{\C}(\Hom_{G_1}(V, V_1), \C)\otimes V_1$.\\
(iii) Par hypothèse, $V_1$ est admissible, donc  $\Hom_{G_1}(V,V_1) \simeq \Hom_{G_1}(V\otimes \check{V_1}, \C) \simeq \Hom_{G_1}((V\otimes \check{V_1})_{G_1}, \C) \simeq \Hom_{\C}((V\otimes \check{V_1})_{G_1}, \C)$. On a $$V\stackrel{\star}{\hookrightarrow } \Hom_{\C}(\Hom_{\C}((V\otimes \check{V_1})_{G_1}, \C), \C)\otimes V_1 \simeq \big((V\otimes \check{V_1})_{G_1}\big)^{\star\star}\otimes V_1.$$ Il reste à montrer qu'en fait, l'image est $(V\otimes \check{V_1})_{G_1} \otimes V_1$. Prenons un élément $0\neq v\in V$ comme en (i); on  définit de m\^eme les applications $c_{i, v}^{\star}, \cdots, c_{n,v}^{\star}$ de $((V\otimes \check{V_1})_{G_1})^{\star\star}$. Soit  $\{ v_1^{\star}, \cdots, v_n^{\star}\}$ la base duale de $\{v_1, \cdots, v_n\}$ dans $(\check{V}_1)^{K_1}$. On a
$$c_{i,v}^{\star}=\langle f(v), v_i^{\star}\rangle = \langle \overline{v\otimes v_i^{\star}}, f\rangle$$
où $\overline{v\otimes v_i^{\star}} $ est l'image de $v\otimes v_i^{\star}$ par l'application $(V\otimes \check{V_1}) \longrightarrow ((V\otimes \check{V_1})_{G_1})^{\star\star}$. Donc nous avons montré que l'image de l'application $\star$ est un sous-espace de $(V\otimes \check{V_1})_{G_1} \otimes V_1$. En utilisant le Lemme \ref{waldspurger1}, on a $V\simeq V_2 \otimes V_1$ pour une représentation lisse $V_2$ de $G_2$; par le raisonnement pour l'unicité, nous obtenons $V_2 \simeq (V\otimes \check{V_1})_{G_1}$.
\end{proof}

\begin{rem}\label{waldspurger3}
\begin{itemize}
\item[(1)]
Dans le lemme \ref{waldspurger1} ci-dessus, supposons que la représentation $(\pi_2, V_2)$ de $G_2$ est  admissible. Alors la sous-représentation $(\pi'_2, V'_2)$ l'est aussi.
\item[(2)] Dans le lemme \ref{waldspurger2} ci-dessus, si $(\pi, V)$ est une représentation lisse admissible de $G_1\times G_2$, alors la représentation $(\pi_2', V_2')$ l'est aussi.
\end{itemize}
\end{rem}
\begin{proof}
de (2): nous suivons les démonstrations de la Proposition de [\cite{BernZ}, Page 20 ]. Soient $K_i$ un sous-groupe ouvert compact de $G_i$ pour $i=1,2$. On sait que $V^{K_1 \times K_2} \simeq (V_1 \otimes V_2')^{K_1 \times K_2}= V_1^{K_1} \otimes V_2'^{K_2}$ qui sont de dimension finie. Cela implique que  l'espace $V_2'^{K_2}$  l'est aussi ( car on peut choisir $K_1$ pour que $V_1^{K_1} \neq 0$).
\end{proof}

Soient $(\pi_1, V_1)$ une représentation admissible irréductible de $G_1$,  et $S_{\pi_1}= S/{S[\pi_1]}$  le plus grand quotient $\pi_1$-isotypique.  Par le Lemme \ref{waldspurger1}, à isomorphisme près, il existe une unique représentation lisse $(\pi_2', V_2')$ de $G_2$ telle que
$$S_{\pi_1} \simeq \pi_1 \otimes \pi_2'.$$
De plus
$$\pi_2' \simeq  \big( \check{V_1} \otimes S_{\pi_1}\big)_{G_1}$$
le plus grand quotient de $\check{V_1} \otimes S_{\pi_1}$, sur lequel $G$ agisse trivialement.\\
Par passage au dual, d'après le résultat du Lemme \ref{waldspurger1}, on trouve
$$\pi_2^{'\star} \simeq \Hom_{G_1}(\check{V_1}\otimes S_{\pi_1}, \C) \simeq \Hom_{G_1}(S_{\pi_1}, V_1) \simeq \Hom_{G_1}(S, V_1)\simeq \Hom_{G_1}(\check{V_1}\otimes S, \C).$$
L'espace $\Hom_{G_1}(S, V_1)$ est naturellement muni d'une action de $G_2$, et les isomorphismes précédents sont $G_2$-équivariants. Par passage à la partie lisse, on obtient $$\check{\pi_2'} \simeq \Hom_{G_1}(S, V_1)^{\infty} \simeq \Hom_{G_1}(\check{V_1}\otimes S, \C)^{\infty}.$$
Nous donnons de plus un lemme pour expliquer pourquoi nous intéressons au plus grand quotient $\pi_1$-isotypique.
\begin{lemme}\label{quotientdedeuxgroupes}
Soient $G_1, G_2$ deux groupes localement compacts totalement discontinus, $(\pi, S)$ une représentation lisse de $G_1 \times G_2$, $(\pi_1, V_1)$(resp. $(\pi_2, V_2)$) une représentation admissible irréductible de $G_1$ ( resp. $G_2$). On note par $S_{\pi_1}$ le plus grand quotient $\pi_1$-isotypique de $\pi$ et soit $S_{\pi_1} \simeq \pi_1 \otimes \pi_2'$, alors
\begin{itemize}
\item[(1)]  $\Hom_{G_1 \times G_2}( S, V_1 \otimes V_2) \simeq \Hom_{G_1\times G_2}(S_{\pi_1}, V_1 \otimes V_2)$.
\item[(2)]   $\Hom_{G_2}( \pi_2', \pi_2) \simeq \Hom_{G_1 \times G_2}( \pi_1 \otimes \pi_2', \pi_1 \otimes \pi_2)$.
\end{itemize}
\end{lemme}
\begin{proof}
(1) D'abord, si  $\Hom_{G_1\times G_2}(S, V_1 \otimes V_2)=0$, comme
 $S \longrightarrow S_{\pi_1}$ est surjective,  on a $\Hom_{G_1\times G_2}(S_{\pi}, V_1\otimes V_2)=0$.
Soit $f\in \Hom_{G_1 \times G_2}(S, V_1 \otimes V_2)$ qui n'est pas triviale,
 comme $(\pi_1 \otimes \pi_2, G_1\times G_2, V_1 \otimes V_2)$ est une représentation irréductible,  $f: S \longrightarrow V_1 \otimes V_2$ est surjective.
Prenons un élément non trivial $e_2$ de $V_2$.
On définit un morphisme canonique $V_1 \otimes V_2 \stackrel{p}{\longrightarrow} V_1\otimes e_2$ qui est $G_1$-équivalent.
Le composé de $f$ avec $p$ détermine un morphisme non trivial dans $\Hom_{G_1}(S, V_1)$, cela implique que $f$ se factorise par $S_{\pi_1} \longrightarrow V_1 \otimes V_2$.\\
(2) L'isomorphisme est défini par $\varphi \longrightarrow 1 \otimes \varphi$. Ce morphisme est bien défini et injectif. Il suffit de démonter qu'il est aussi surjectif. Soit $\varphi' \in \Hom_{G_1 \times G_2}(V_1 \otimes V_2', V_1 \otimes V_2)$ un élément non trivial.  Choisissons
une base $\{u_i\}_{i\in I}$ de $V_2$, nous notons par $V_{2,i}$ la droite engendré par l'élément $u_i$ pour $i\in I$. On a  $$V_2 = \oplus_{i\in I} V_{2,i}$$ qui peut plonger dans $ \prod_{i\in I} V_{2,i}$ comme un sous-espace.  Donc $V_1 \otimes V_2 \simeq \oplus_{i\in I} V_1 \otimes V_{2,i}$ se voit aussi comme un sous-espace vectoriel de $\prod_{i\in I} V_1 \otimes V_{2,i}$. Nous notons sa projection canonique $\prod_{i\in I} V_1 \otimes V_{2,i} \longrightarrow V_1 \otimes V_{2,i}$ par $p_i$. Chaque composante $V_1 \otimes V_{2,i}$ est isomorphe à $V_1$. Prenons un élément non trivial $e_2' \in V_2'$, et considérons le morphisme $\varphi'|_{V_1 \otimes e_2'}: V_1 \otimes e_2' \longrightarrow V_1 \otimes V_2$. Composons le avec $$V_1 \otimes V_2 \longrightarrow  \prod_{i\in I} V_1 \otimes V_{2,i} \stackrel{p_i}{\longrightarrow} V_1 \otimes V_{2,i}.$$ Nous obtenons un morphisme
$$\varphi_i': V_1 \otimes e_2' \longrightarrow  V_1 \otimes V_{2,i},$$
 qui est  $G_1$-équivalent.

Puisque $\pi_1$ est admissible, on sait que le lemme de Schur est vrai en ce cas, c'est-à-dire que $\End_{G_1}(V_1) \simeq \C$. Cela implique que $\varphi'_i$ est défini par
 $$ V_1 \otimes e_2' \longrightarrow V_1 \otimes V_{2,i};$$
 $$\sum_k v_k \otimes e_2' \longmapsto \sum_k v_k \otimes c_iu_i,$$ pour quelque $c_i \in \C$.

 De plus le morphisme $\prod_{i\in I} \varphi_i': V_1 \otimes e_2' \longrightarrow \prod_{i\in I} V_1 \otimes V_{2,i}$ se factorise à $V_1 \otimes e_2' \longrightarrow V_1 \otimes V_2$, ce qui implique que sauf pour un nombre fini d'indices $i$, $\varphi'_i=0$.

 Ensuite, nous définissons une application $\varphi_{e_2'}: \C e_2' \longrightarrow V_2$ par $\varphi_{e_2'}(e_2')= \sum_{i\in I} c_i u_i$. On a
 $$\varphi'|_{V_1 \otimes e_2'}= 1 \otimes \varphi_{e_2'}.$$

 De cette manière, pour chaque élément $v_2' \in V_2'$, on construit une application $\varphi_{v_2'}: \C v_2' \longrightarrow V_2$ satisfaisant à $\varphi'|_{V_1 \otimes v_2'}= 1 \otimes \varphi_{v_2'}$, et cette application est unique. On a donc

$\varphi_{\alpha v_2' + \beta v_2'}=\varphi_{\alpha v_2'} + \varphi_{\beta v_2''}=\alpha \varphi_{v_2'}+ \beta \varphi_{v_2''}$ pour $\alpha, \beta \in \C$, $v_2', v_2'' \in V_2'$.

Nous avons une application $\C$-linéaire: $ \varphi: V_2' \longrightarrow V_2$  donnée par $\varphi(\sum_i v_{2,i}'):= \sum_i\varphi_{v_{2,i}'}(v_{2,i}')$. Alors $\varphi'=1 \otimes \varphi$, et  $\varphi$ est forcément $G_2$-équivariant, i.e. $\varphi \in \Hom_{G_2}(V_2', V_2)$.
\end{proof}

En vue des applications, nous démontons un lemme de [\cite{MVW}, Page 59]. Soient $\mathcal {S}=\mathcal {S}(G)$ l'espace des fonctions de Schwartz-Bruhat sur $G$, muni de la représentation naturelle $\rho$ de $G\times G$:
$$\rho(g_1, g_2) f(g):= f(g_1^{-1} g g_2).$$
\begin{lemme}\label{exemple}
Pour tout $\pi \in \Irr(G)$, le plus grand  quotient $\pi$-isotopique de $\rho_{G\times 1}$ est isomorphe à $\pi \otimes \check{\pi}$ comme  $G\times G$-module.
\end{lemme}
\begin{proof}
Soit $\rho_{\pi}$ le plus grand quotient $\pi$-isotypique de $\rho_{G\times 1}$, alors $\rho_{\pi} \simeq \pi\otimes \sigma$. Par la discussions ci-dessus, on sait que $\check{\sigma} \simeq \Hom_{G_1}(\rho, \pi)^{\infty}$, et le résultat de  [\cite{Bern},  Page 74], a montré que $\Hom_{G_1}(S, \pi)^{\infty} \simeq \pi$.
\end{proof}

 \section{Des définitions}
 Dans cette sous-section, on supposera que toute représentation lisse irréductible de $G_1\times G_2$ est admissible. Par [\cite{BernZ}, Page 20, Proposition],  on sait qu'une représentation lisse irréductible de $G_1\times G_2$ est de la forme $\pi_1 \otimes \pi_2$ où $\pi_1$ et $\pi_2$ sont uniques à isomorphisme près.

\begin{prop}\label{lissedetypefini}
Si $(\pi,  S)$ est  une   représentation lisse  de type fini du groupe $G_1\times G_2$.  Alors\\
(1) $\pi$ est une représentation de quotient admissible.\\
(2) $\mathcal{R}_{G_1\times G_2}(S)=\varnothing$ si et seulement si $(\pi, S)=0$.\\
(3) Soient $\pi_1 \in \Irr(G_1)$, $S_{\pi_1}$ le plus grand quotient $\pi_1$-isotypique de $\pi$ et $S_{\pi_1} \simeq \pi_1\otimes \pi_2'$; alors $\pi_2'$ est  une représentation de type fini de $G_2$.
\end{prop}
\begin{proof}
Nous avons vu (1) [cf. Proposition \ref{typeimpliquequotientadmissible}] et (2) [cf. Proposition \ref{quotientzero}]. Pour (3), on a
$$S_{\pi_1} \simeq S/S[\pi_1] \simeq \pi_1 \otimes \pi_2'$$
qui sont aussi des représentations du groupe $G_1\times G_2$ de type fini. On choisit l'ensemble $\{v_1^{(1)}\otimes v_2^{'(1)}, \cdots, v_1^{(n)}\otimes v_2^{'(n)}\}$ de $n$ vecteurs  engendrant $S_{\pi_1}$. Comme $(\pi_1,V_1)$ est irréductible admissible, par le Lemme \ref{waldspurger2}, on a vu  que $\pi_2'$ est engendré par les éléments $v_2^{'(1)}, \cdots, v_2^{'(n)}$.
\end{proof}
\begin{prop}\label{grapheprojectif}
Soit $(\pi, S)$ une représentation lisse de type fini du groupe $G_1 \times G_2$.
\begin{itemize}
\item[(1)]  Si $\pi_1 \otimes \pi_2 \in \mathcal{R}_{G_1 \times G_2}(\pi)$, alors $\pi_1 \in \mathcal{R}_{G_1}(\pi)$.
\item[(2)]  Si $\pi_1 \in \mathcal{R}_{G_1}(\pi)$, alors il existe $\pi_2 \in \mathcal{R}_{G_2}(\pi)$ tel que $\pi_1 \otimes \pi_2 \in \mathcal{R}_{G_1 \times G_2}(\pi)$.
\end{itemize}
\end{prop}
\begin{proof}
(1) Soit $(\pi_1\otimes \pi_2, G_1\times G_2, V_1\otimes V_2)$ un élément de $\mathcal{R}_{G_1\times G_2}(\pi)$, on a une application non triviale
$$V\stackrel{f}{\longrightarrow} V_1\otimes V_2$$ qui est surjective. Prenons un élément $0\neq e_2 \in V_2$ et une  application canonique
$$V_2 \stackrel{p_{e_2}}{\longrightarrow} \mathbb{C}e_2.$$
Composons $f$ avec $1\otimes p_{e_2}$; nous obtenons une application non triviale de $V$ à $V_1$ i.e. $\pi_1 \in \mathcal{R}_{G_1}(\pi)$.\\
(2)  Supposons $(\pi_1, V_1) \in \mathcal{R}_{G_1}(\pi)$.
On peut trouver le plus grand quotient $\pi_1$-isotypique $S_{\pi} \simeq \pi_1 \otimes \pi_2'$, qui n'est pas trivial, cela implique que $\pi'_2$ ne l'est pas aussi. Comme $\pi_2'$ est une représentation lisse de type fini, il en résulte que $$\mathcal{R}_{G_2}(\pi_2')\neq 0.$$
D'après le lemme \ref{quotientdedeuxgroupes}, on a une bijection entre $\mathcal{R}_{G_1\times G_2}( \pi)$ et $\mathcal{R}_{G_2}(\pi_2')$. Ceci montre qu'il existe une représentation $(\pi_2,V_2)$ de $G_2$ telle que $\pi_1\otimes \pi_2 \in \mathcal{R}_{G_1\times G_2}(\pi)$.
\end{proof}

Si $(\pi, S)$ est une représentation lisse du groupe $G_1 \times G_2$, on a vu que le résultat dans la  proposition \ref{grapheprojectif} (1) sera aussi vrai, donc  on définit deux applications projectives canoniques
$$\mathcal{R}_{G_1 \times G_2}(\pi) \stackrel{p_i}{\longrightarrow} \mathcal{R}_{G_i}(\pi); \pi_1 \otimes \pi_2 \longmapsto \pi_i \textrm{ pour } i=1,2.$$

On notera leurs images par $\mathcal{R}^0_{G_i}(\pi)$.
\begin{cor}
Si $(\pi, S)$ est une représentation lisse de  type fini du groupe $G_1 \times G_2$, alors $p_i$ est surjective pour $i=1,2$.
\end{cor}

Si $p_1$(resp.  $p_2$) est injectif, alors pour chaque $\pi_1 \in \mathcal{R}_{G_1}^0(\pi)$ (resp. $\pi_2 \in \mathcal{R}_{G_2}^0(\pi)$), il existe une unique représentation irréductible $\pi_2^{(1)} \in \mathcal{R}_{G_2}(\pi)$ (resp.  $\pi_1^{(2)} \in \mathcal{R}_{G_1}(\pi)$) telle que $\pi_1 \otimes \pi_2^{(1)} \in \mathcal{R}_{G_1 \times G_2}(\pi)$ (resp.  $\pi_1^{(2)} \otimes \pi_2 \in \mathcal{R}_{G_1 \times G_2}(\pi)$).

\begin{defn}\label{representationdegraphe}
Si $p_1$(resp.  $p_2$) est injectif, on définit une application $\theta_1: \mathcal{R}_{G_1}^0(\pi)\longrightarrow \mathcal{R}_{ G_2}(\pi); \pi_1 \longmapsto \pi_2^{(1)}$, (resp.  $\theta_2: \mathcal{R}_{G_2}^0(\pi) \longrightarrow \mathcal{R}_{G_1}(\pi); \pi_2 \longmapsto \pi_1^{(2)} $), on dit que $(\mathcal{R}_{G_1\times G_2}(\pi), p_i)$ est le graphe de \textbf{l'application de Howe} $\mathcal{R}_{G_1}^0(\pi) \stackrel{\theta_1}{\longrightarrow} \mathcal{R}_{G_2}(\pi)$(resp.  $\mathcal{R}_{G_2}^0(\pi) \stackrel{\theta_2}{\longrightarrow} \mathcal{R}_{G_1}(\pi)$) et que $(\pi, G_1\times G_2)$ est \textbf{une représentation de graphe à gauche} (resp.  \textbf{ à droite}). De plus, si  $\pi$ est aussi une représentation de quotient sans multiplicité, i.e.
$$m_{G_1\times G_2}(\pi, \pi')\leq 1 \textrm{ pour tout } \pi' \in \Irr(G_1\times G_2),$$
on dit que $\pi$ est \textbf{une représentation de graphe forte à gauche} (resp. \textbf{ à droite}), et   que $\pi$ satisfait à \textbf{la propriété  de graphe forte à gauche }(resp. \textbf{ à droite}).
\end{defn}

\begin{defn}\label{representationdebigraphe}
Si $p_1,p_2$ sont injectifs, on dit que $(\mathcal{R}_{G_1\times G_2}(\pi), p_i)$ est un \textbf{bigraphe} entre $\mathcal{R}_{G_1}^0(\pi)$ et $\mathcal{R}_{G_2}^0(\pi)$, et que $(\pi, G_1\times G_2)$ est  une \textbf{représentation de Howe} (ou de bigraphe). Il a déterminé une correspondance bijective entre $\mathcal{R}^0_{G_1}(\pi)$ et $\mathcal{R}^0_{G_2}(\pi)$ qui sera appelée la \textbf{ correspondant de Howe }(ou de bigraphe), on aussi dit que $(\pi, G_1\times G_2)$ satisfait à la \textbf{propriété de Howe}(ou de bigraphe).  De plus, si $\pi$ est aussi une représentation de quotient sans multiplicité, i.e. $$m_{G_1\times G_2}(\pi, \pi')\leq 1 \textrm{ pour tout } \pi' \in \Irr(G_1\times G_2),$$
on dit que $\pi$ est une \textbf{ représentation de Howe forte } (ou  de bigraphe forte), et  que $\pi$ satisfait à la  \textbf{propriété de Howe forte }( ou de bigraphe forte ). Ce qui détermine \textbf{une correspondance de Howe forte  } ( ou de bigraphe forte ) entre $\mathcal{R}_{G_1}^0(\pi)$ et $\mathcal{R}_{G_2}^0(\pi)$.
\end{defn}
\section{La théorie de Clifford}
Soient $G$ un groupe localement compact totalement discontinu, $H$ un sous-groupe fermé de $G$. $G/H$ est un groupe abélien. On supposera que la représentation irréductible du groupe $G$ ou $H$ est admissible.
\begin{thm}\label{cliffordadmissible}
Soit $(\pi, V) \in \Irr(G)$ telle que $\pi|_H$ est admissible. Alors:
\begin{itemize}
\item[(1)] $\pi|_H$ est semi-simple, et  une représentation de quotient admissible.
\item[(2)] si $\sigma_1, \sigma_2 \in \mathcal{R}_H(\pi)$, alors il existe un élément $g\in G$ tel que $\sigma_2 \simeq \sigma_1^g$ où
$ \sigma_1^g(h):= \sigma_1(g^{-1} hg)$ pour $h\in H$.
\item[(3)]  il existe une constante $m$ telle que
$$\pi|_H=\sum_{\sigma\in \mathcal{R}_H(\pi)} m \sigma.$$
\item[(4)] soit $(\sigma, W)\in \mathcal{R}_H(\pi)$ qui est une sous-représentation de $\pi|_H$.\\
(i)  On note $ G_{\sigma}^0=\{ g\in G| g(W)=W  \}$,
 alors $G_{\sigma}^0$ est un sous-groupe distingué ouvert de $G$, et $G_{\sigma_1}^0=G_{\sigma_2}^0$
 pour  $\sigma_1, \sigma_2 \in \mathcal{R}_H(\pi)$, on le note par $\widetilde{H}^0$.\\
(ii) On note $ G_{\sigma}=\{ g\in G| \sigma^{g} \simeq \sigma \}$,
 alors $G_{\sigma}$ est un sous-groupe distingué ouvert de $G$, et $G_{\sigma_1}=G_{\sigma_2}$
 pour  $\sigma_1, \sigma_2 \in \mathcal{R}_H(\pi)$, on le note par $\widetilde{H}$.\\
(iii) $\widetilde{H}/{\widetilde{H}^0}$ est un groupe abélien d'ordre fini.
\item[(5)] (i) pour chaque $\sigma \in \mathcal{R}_H(\pi_1)$, elle peut se prolonger en une représentation irréductible du groupe $\widetilde {H}^0$, on la note aussi par
$\sigma$.\\
(ii) $\pi|_{\widetilde{H}} = \oplus_{\widetilde{\sigma} \in \mathcal{R}_{\widetilde{H}}(\pi)} \widetilde{\sigma}$, et $\widetilde{\sigma}|_H \simeq m \sigma$. Le groupe $G/{\widetilde{H}}$ permute transitivement l'ensemble $\mathcal{R}_{\widetilde{H}}(\pi)$.
\end{itemize}
\end{thm}
\begin{proof}
(1) On prend un sous-groupe ouvert compact $K_H$ de $H$ tel que $V^{K_H}\neq \varnothing$. D'après le Lemme \ref{admissibleduale}, $\pi|_H$ est aussi admissible, ce qui entraîne que $V^{K_H}$ est un espace vectoriel de dimension finie.   Posant une sous-représentation irréductible $W_1^{K_H}$ de $\mathcal{H}(H,K_H)$ de $V^{K_H}$. On obtient  une sous-représentation $(\sigma_1, W_1)$ du groupe $H$ de $\pi|_H$ où $W_1=\mathcal{H}(H)\otimes_{\mathcal{H}(H,K_H)}W_1^{K_H}$. Par [\cite{BernZ}, Page 17, Proposition], $(\sigma_1, W_1)$ est irréductible. Soit $\Omega=\{ g_i | g_i \in G\}$ un ensemble de représentants de $G/H$. Considérons $V'=\sum_{g_i \in \Omega} \pi(g_i) W_1$ qui est $G$-invariante. Donc $V' = V$, et on montre que $\pi|_H$ est semi-simple. \\
Soient $(\sigma, W) \in \mathcal{R}_H(\pi)$, et $W^K \neq 0$ pour un sous-groupe ouvert compact $K$ de $H$. Comme $\sigma$ est admissible, $W^K$ est de dimension finie. Donc on a
$$m_H(\pi, \sigma) \leq m_H( \pi, \Ind_K^H W^K) = m_K(\pi, W^K)< \infty.$$
(2) Ceci découle de la démonstration de $(1)$.\\
(3) Si $\sigma_1, \sigma_2 \in \mathcal{R}_H(\pi)$. Par (2), on a $\sigma_2 \simeq \sigma_1^g$ pour un élément $g\in G$, donc
$$m_H( \pi, \sigma_1)=m_H(\pi^g, \sigma_1^g) =m_H(\pi, \sigma_2)=m, $$
 d'où l'on tire le résultat annoncé.\\
(4) (i) Soit $W=\pi(H) w_0$ pour un vectoriel $w_0$ de $W$, on note $K_{w_0}=\stab_G(w_0)$.  Si $g\in K_{w_0}$,
on voit asiément que $\pi(g) W=W$, i.e. $G_{\sigma}^0 \supseteq K_{w_0}$. Ceci implique que $G_{\sigma}^0$ est
un groupe ouvert de $G$. Si $(\sigma_1,W_1), \, (\sigma_2, W_2) $ sont deux sous-représentations irréductibles de $\pi|_H$, alors il existe $g_{12}\in G$
tel que $g_{12}W_1=W_2$. On a un morphisme bijectif $G_{\sigma_1}^0\longrightarrow G_{\sigma_2}^0; g \longmapsto g_{12} g g_{12}^{-1}$. Comme $G/H$ est abélien,
on obtient $G_{\sigma_1}^0=G_{\sigma_2}^0$.\\
(ii) Comme $G_{\sigma_1}\supseteq G_{\sigma_1}^0$, cela implique que $G_{\sigma_1}$ est un sous-groupe ouvert de $G$.
Par la m\^eme démonstration comme (i) ci-dessus, on
trouve $G_{\sigma_1}=G_{\sigma_2}$ pour $\sigma_1, \sigma_2 \in \mathcal{R}_H(\pi)$.\\
(iii) On note $\mathcal{W}=\{ (\sigma_i, W_i) | \sigma_i$ est une sous-représentation de $\pi|_H$ telle que  $\sigma_i \simeq \sigma \}$. Le groupe $\widetilde{H}/{\widetilde{H}^0}$ permute
transitivement l'ensemble $\mathcal{W}$. Il reste à montrer que  le cardinal de $\mathcal{W}$ est fini. Soit $K$ un sous-groupe ouvert compact de $\widetilde{H}^0$, et m\^eme de $\widetilde{H}$ et $G$, tel que $W^{K\cap H} \neq \varnothing$. Un élément $(\sigma_i, W_i)$  correspond uniquement à une sous-representation
irréductible du groupe $\mathcal{H}(H, K\cap H)$ de $V^{K\cap H}$. Comme $V^{K\cap H}$ est un  espace vectoriel de dimension finie, on trouve le résultat.\\
(5) (i) Ceci découle de la défintion du groupe $\widetilde{H}^0$. \\
(ii) Comme $\pi|_{\widetilde{H}}$ est admissible, on a vu que  $\pi|_{\widetilde{H}}$ est semi-simple. La composante $\sigma$-isotypique $\widetilde{\sigma}=m\sigma$ est $\widetilde{H}$-invariante. Par définition, $\widetilde{\sigma}$ est $\widetilde{H}$-irréductible. Puisque $\widetilde{\sigma}|_H \simeq m \sigma$, et $\widetilde{\sigma_1} \ncong \widetilde{\sigma_2}$ si $\sigma_1 \ncong \sigma_2$ pour $\sigma_1, \sigma_2 \in \mathcal{R}_H(\pi)$. Si $g\in G$ satisfait à $\widetilde{\sigma}^g \simeq \widetilde{\sigma}$, a fortiori,  $\sigma^g\simeq \sigma$, donc $g\in \widetilde{H}$, d'où le résultat.\\
\end{proof}
\begin{cor}\label{equalitydequotient}
Si $\pi_1, \pi_2 \in \Irr(G)$ telles que  $\pi_1|_H, \pi_2|_H$ sont admissibles, alors:
\begin{itemize}
\item[(1)] $\mathcal{R}_H(\pi_1) \cap \mathcal{R}_H(\pi_2) \neq \varnothing$ si et seulement si $\mathcal{R}_H(\pi_1)=\mathcal{R}_{H}(\pi_2)$.
\item[(2)] Si $\mathcal{R}_H(\pi_1)=\mathcal{R}_H(\pi_2)$, alors $\pi_1 \simeq \pi_2 \otimes \chi_{G/H}$ pour un caractère $\chi_{G/H}$ de $G/H$.
\end{itemize}
\end{cor}
\begin{proof}
(1) Si $\sigma \in \mathcal{R}_H(\pi_1) \cap \mathcal{R}_H(\pi_2) $. D'après le Théorème \ref{cliffordadmissible}(1), pour chauqe $\sigma' \in \mathcal{R}_H(\pi_1)$, il existe $g\in G/H$ tel que $\sigma^g\simeq \sigma'$, donc
$$m_H(\pi_2, \sigma')=m_H(\pi_2, \sigma^g)=m_H(\pi_2^g, \sigma^g)=m_H(\pi_2, \sigma),$$
d'où $\sigma' \in \mathcal{R}_H(\pi_2)$, on trouve le résultat.\\
(2) Soit $(\sigma, W) \in \mathcal{R}_H(\pi_1) \cap \mathcal{R}_H(\pi_2)$. Comme $\pi_1|_H, \pi_2|_H$ sont admissibles. 
On suppose que $(\sigma, W)$ est une sous-représentation de $\pi_1|_H$, et $\pi_2|_H$.
 Par ailleurs, si $\pi_1|_H=W \oplus \pi_1(g_1) W \oplus \cdots \oplus \pi_1(g_m) W \oplus V_1$, et $\pi_2|_H=W\oplus V_2$, 
où $W \oplus \pi_1(g_1) W \oplus \cdots \oplus \pi_1(g_m) W$ est la composante $\sigma$-isotypique de $\pi_1|_H$.  
On définit un élément $0 \neq f\in  \Hom_H(\pi_1, \pi_2)$ de la façon suivante:\\
  Pour $w\in W$, $v_1 \in V_1$, on a 
\begin{itemize}
\item[(1)] $f|_W=\id_W$.
\item[(2)] $f|_{\pi_1(g_i)(W)}(\pi_1(g_i)w)=\pi_2(g_i)w$.
\item[(3)] $f|_{V_1}=0$.
\end{itemize}
Soit $W=\pi_1(H) w_0$ pour un vectoriel $w_0$ non trivial de $W$. On note  $K_{w_0}=\Stab_G(w_0)\cap \Stab_G(f(w_0))$, et $K_{g_i (w_0)}=\Stab_G(\pi_1(g_i)w_0)\cap \Stab_G(\pi_2(g_i) w_0)$ qui sont des groupes ouverts de $G$. On pose $K=K_{w_0} \cap \cap_{i=1}^m
K_{g_i(w_0)}$. Soit $\overline{g} \in G/H$, on définit un espace vectoriel $V_0$ engendré par $f^{\overline{g}}$, 
où $f^{\overline{g}}(v):= \pi_2(g) f\Big(\pi_1(g^{-1})v\Big)$. Si $g_0\in K$, 
on note $\overline{g_0}$ l'image de $g_0$ dans $G/H$ par $G \stackrel{p}{\longrightarrow} G/H$. 
Alors, pour $w=\sum_{i=1}^n c_i \pi_1(h_i) w_0 \in W$, et $v_1\in V_1$,
on a
\begin{itemize}
\item[(1)] $f^{\overline{g_0}}(w)= f^{\overline{g_0}}(\sum_{i=1}^n c_i\pi_1(h_i) w_0)= \sum_{i=1}^n c_i\pi_2(h_i) f^{\overline{g_0}}\Big(  w_0\Big) = \sum_{i=1}^n c_i \pi_2(h_i) f\Big(  w_0\Big)= f(w)$.
\item[(2)] $f^{\overline{g_0}}(\pi_1(g_i)w)= f^{\overline{g_0}}(\sum_{i=1}^n c_i\pi_1(g_i)\pi_1(h_i) w_0)= \sum_{i=1}^n c_i\pi_2(g_ih_ig_i^{-1}) f^{\overline{g_0}}\Big( \pi_1(g_i) w_0\Big)$\\ $ = \sum_{i=1}^n c_i \pi_2(g_ih_ig_i^{-1}) f\Big(\pi_1(g_i)  w_0\Big)= f(\pi_1(g_i)w)$.
\item[(3)] $f^{\overline{g_0}}(v_1)=\pi_2(g_0)f(\pi_1(g_0^{-1})v_1)=0=f(v_1).$
\end{itemize}
 Ceci implique que $\Stab_{G/H}(f)$ contient l'ensemble ouvert $p(K)$. 
De la m\^eme raison, pour $g\in G$, $\Stab_{G/H}(f^{\overline{g}}) \supseteq \overline{g}^{-1} p(K)$ ,
qui est un sous-ensemble ouvert de $G/H$.  
Donc  $V_0$ est une représentation lisse du groupe $G/H$ 
munie de l'action $\overline{g}\cdot F:= F^{\overline{g}}$ pour $\overline{g} \in G/H, F\in V_0$.
 Il existe un caractère $\chi_{G/H}$ et $0\neq F \in V_0$ tels 
que $\overline{g} \cdot F=\chi_{G/H} F$, c'est-à-dire que $ 0 \neq F \in \Hom_G(\pi_1, \pi_2 \otimes \chi_{G/H}^{-1})$, d'où le résultat.
\end{proof}

\begin{prop}\label{dimesionsmall1}
Soient $\pi$ une représentation lisse de quotient admissible de $G$, $\pi_1 \in \mathcal{R}_G(\pi)$.  Supposons que $\pi_1|_H$ est admissible.
\begin{itemize}
\item[(1)] Si $\sigma \in \mathcal{R}_H(\pi_1)$, alors $\sigma \in \mathcal{R}_H(\pi)$.
\item[(2)] $m_H(\pi, \sigma_1)= m_H(\pi, \sigma_2)$ pour $\sigma_1, \sigma_2 \in \mathcal{R}_H(\pi_1)$.
\item[(3)] Si $m_H(\pi, \sigma) \leq 1$ pour chaque  $\sigma \in \mathcal{R}_H(\pi_1)$, alors $m_G( \pi, \pi_1) \leq 1$.
\end{itemize}
\end{prop}
\begin{proof}
(1) On a un homomorphisme surjectif $\pi \longrightarrow \pi_1$. Ceci implique le résultat.\\
(2) D'après le Théorème \ref{cliffordadmissible}, il existe un élément $g\in G$, tel que $\sigma_2 \simeq \sigma_1^g$, donc
$m_H(\pi, \sigma_1)=m_H(\pi^g, \sigma_1^g)=m_H(\pi, \sigma_2)$.\\
(3) Supposons que $\pi_1|_H=\oplus_{\sigma \in \mathcal{R}_H(\pi_1)} m \sigma$. En vertu du Théorème \ref{cliffordadmissible}, on prend le sous-groupe $\widetilde{H}$ de $G$ tel que
\begin{equation}\label{ladecomposition}
\pi_1|_{\widetilde{H}}=\oplus_{\widetilde{\sigma} \in \mathcal{R}_{\widetilde{H}}(\pi_1)}\widetilde{\sigma}, \textrm{ où }  \widetilde{\sigma}|_H=m\sigma  \textrm{ pour }
\sigma\in \mathcal{R}_H(\pi_1).
\end{equation}
Comme $m_H(\pi, \sigma) \leq 1$, on voit que $m_{\widetilde{H}}(\pi, \widetilde{\sigma})\leq 1$.  De plus
$$\Hom_G(\pi, \pi_1) =\Hom_{\widetilde{H}}(\pi, \pi_1)^{\widetilde{H} \setminus G},$$
où on munit $\Hom_{\widetilde{H}}(\pi, \pi_1)$ de l'action évidente de $\widetilde{H}\setminus G$, i.e. $f^{\overline{g}}(v):= \pi_1(g) f\Big( \pi(g^{-1}) v\Big)$ pour
$\overline{g} = \widetilde{H} g \in \widetilde{H} \setminus G$, $f\in \Hom_{\widetilde{H}}(\pi, \pi_1)$. Par  décomposition (\ref{ladecomposition}),
on note l'application projective de $\pi_1|_{\widetilde{H}}$ à $\widetilde{\sigma}$ par $p_{\widetilde{\sigma}}$. Pour  chaque élément $F\in \Hom_{\widetilde{H}}(\pi, \pi_1)$, ce qui est
déterminé par la famille de morphismes $\{ p_{\widetilde{\sigma}}\circ F\}_{\widetilde{\sigma} \in \mathcal{R}_{\widetilde{H}}( \pi_1)}$ d'où  $p_{\widetilde{\sigma}} \circ F \in
\Hom_{\widetilde{H}}(\pi, \widetilde{\sigma})$. Le groupe $\widetilde{H}\setminus G$ agit transitivement sur
 $ \{ \Hom_{\widetilde{H}}(\pi, \widetilde{\sigma})\}_{\widetilde{\sigma}  \in \mathcal{R}_{\widetilde{H}}(\pi_1)}$, et
$\dim_{\C} \Hom_{\widetilde{H}}(\pi, \widetilde{\sigma}) \leq 1$ pour $\widetilde{\sigma} \in \mathcal{R}_{\widetilde{H}}(\pi_1)$, on trouve le  résultat.
\end{proof}
\section{Les représentations de graphe forte}
Dans cette sous-section, pour  présenter simplement les résultats, ``graphe'' signe `` graphe à gauche'' sauf mention du contraire. Et on supposera que toute représentation irréductible du groupe localement compact totalement discontinu est admissible.
\begin{prop}[ Produits]\label{graphesproduits}
Soient $G_i, H_i$ des groupes localement compacts totalement discontinus pour $i=1, 2$, $(\pi, V)$ une représentation lisse de $(G_1\times G_2) \times (H_1\times H_2)$.
\begin{itemize}
\item[(1)] Supposons que  pour chaque indice $i$, $ \pi|_{G_i \times H_i} \textrm{ satisfait à la propriété  de graphe}$. Alors $\pi$ est une représetnation de graphe du groupe $(G_1 \times G_2) \times (H_1 \times H_2)$.
\item[(2)] Supposons que $\pi$ est une représentation de graphe du groupe $G_1\times \cdots \times G_n$ et qu'il existe un couple $(G_i , G_j)$ de l'ensemble $\{ G_1, \cdots, G_n\}$ tel que $\pi|_{G_i\times G_j}$ satisfait à la condition de sans multiplicité. Alors $\pi$ est une représentation de graphe forte du groupe $G_1 \times \cdots \times G_n$.
\end{itemize}
\end{prop}
\begin{proof}
Par définition, $p_1^{(i)}: \mathcal{R}_{G_i \times H_i}(\pi) \longrightarrow \mathcal{R}_{G_i}(\pi)$ est injective pour $i=1, 2$.  Il en résulte que $p_1: \mathcal{R}_{(G_1 \times H_1) \times (G_2 \times H_2)}(\pi) \longrightarrow \mathcal{R}_{G_1 \times G_2}(\pi)$ est aussi injective. Donc $(\pi, (G_1 \times G_2) \times (H_1 \times H_2))$ est une représentation de graphe.  De plus, l'application de Howe de $\pi$ soit définie par $\theta=\theta_1  \times \theta_2: \mathcal{R}^0_{G_1 \times G_2}(\pi) \longrightarrow  \mathcal{R}^0_{H_1\times H_2}(\pi)$  où $\theta_i: \mathcal{R}^0_{G_i}(\pi) \longrightarrow \mathcal{R}^0_{H_i}(\pi)$  est l'application de Howe de $\pi|_{G_i\times H_i}$.\\
(2) Soit $\pi_1\otimes \pi_2 \otimes \cdots \otimes \pi_n \in \mathcal{R}_{G_1\times \cdots \times G_n}(\pi)$, si $m_{G_i\times G_j}(\pi, \pi_i \otimes \pi_j) \leq 1$, on  a alors $m_{G_1\times \cdots \times G_n}(\pi, \pi_1 \otimes \cdots \pi_n) \leq m_{G_i\times G_j}(\pi, \pi_i\otimes \pi_j) \leq 1$. Ceci trouve le résultat.
\end{proof}

\underline{Maintenant, nous présentons  un des résultats principaux de cet article:}\\
Soient $G_1, G_2$ des groupes localement compacts totalement discontinus, $H_1$ (resp. $H_2$) un sous-groupe distingué fermé de $G_1$ (resp. $G_2$). Supposons que  $G_i/H_i$ est soluble pour $i=1,2$. On supposera que toute représentation irréductible du groupe $G_i$ ou $H_i$ est admissible. Si $\pi_i \in \Irr(G_i)$, on supposera que $\pi_i|_{H_i}$ est aussi admissible.

\begin{thm}\label{abeliengraphedefini}
Si $\Gamma$ est un sous-groupe distingué fermé  de $G_1\times G_2$ contenant $H_1\times H_2$  tel que
$$\Gamma/{H_1\times H_2} \textrm{ est le graphe d'un morphisme surjectif } \gamma: G_1/H_1 \longrightarrow G_2/H_2.$$
Supposons que $(\rho, V)$ est une représentation de graphe forte du groupe $H_1 \times H_2$ qui se prolonge en une représentation du groupe $\Gamma$. Alors
$\pi=c-\Ind_{\Gamma}^{G_1\times G_2} \rho$ est aussi une représentation de graphe forte du groupe $G_1\times G_2$.
\end{thm}
On procède par récurrence sur le cardinal du groupe  $\Gamma/{H_1 \times H_2}$, en traitant d'abord  le cas où ce groupe est abélien.

\begin{lemme}\label{principallemmedegraphe}
Dans la situation du théorème \ref{abeliengraphedefini}, supposons de plus que $G_1/H_1$ est abélien.
Si $(\pi_1, V_1) \in \Irr(G_1)$ et que $(\pi_2, V_2) \in \Irr(G_2)$ tels que $\pi_1 \otimes \pi_2\in \mathcal{R}_{G_1 \times G_2}(\pi)$.\\
(1) Pour chaque $\sigma_{\alpha} \in \mathcal{R}_{H_1}(\pi_1)$, il existe un et un seul élément $\delta_{\alpha} \in \mathcal{R}_{H_2}(\pi_2)$ tel que $\sigma_{\alpha} \otimes \delta_{\alpha}\in \mathcal{R}_{H_1 \times H_2}(\rho)$.\\
(2) $\mathcal{R}_{\Gamma}(\rho) \cap \mathcal{R}_{\Gamma}(\pi_1 \otimes \pi_2)= \{ \widetilde{\sigma}\}$, et $m_{\Gamma}(\pi_1 \otimes \pi_2, \widetilde{\sigma})=1.$\\
(3) $m_{G_1 \times G_2}(\pi, \pi_1 \otimes \pi_2)= 1$.
\end{lemme}
\begin{proof}
(1)  Comme $\pi_1 \otimes \pi_2 \in \mathcal{R}_{G_1 \times G_2}(\pi)$, en vertu du  Théorème Réciproque de Frobenius, on a
$$0 \neq m_{G_1 \times G_2}(\pi, \pi_1 \otimes \pi_2)=m_{\Gamma}(\rho, \pi_1 \otimes \pi_2).$$
A fortiori, $m_{H_1 \times H_2}( \rho, \pi_1 \otimes \pi_2) \neq 0$. Quitte à change l'indexation, on trouve que
$m_{H_1 \times H_2}(\rho, \sigma_1 \otimes \delta_1) \neq 0$ pour des éléments $\sigma_1 \in \mathcal{R}_{H_1}(\pi_1)$, $ \delta_1 \in \mathcal{R}_{H_2}(\pi_2)$. En vertu de la théorie de Clifford, il existe un élément $t_{\alpha}H_1$ de $G_1/H_1$ tel que
$\sigma_1^{t_{\alpha}}\simeq\sigma_{\alpha}$ où $\sigma_1^{t_{\alpha}}(h_1):= \sigma_1(t_{\alpha}^{-1} h_1 t_{\alpha})$ pour $h_1 \in H_1$.
Soit $\gamma(t_{\alpha}H_1) =s_{\alpha}H_2 \in G_2/{H_2}$; alors $(t_{\alpha}, s_{\alpha})  H_1 \times H_2$ appartient à $\Gamma/{H_1 \times H_2}$. Supposons que $(t_{\alpha}, s_{\alpha}) \in \Gamma$ et $\delta_{\alpha}= \delta_1^{s_{\alpha}}$. On a
$$ 0 \neq m_{H_1 \times H_2}(\rho, \sigma_1 \otimes \delta_1) = m_{H_1 \times H_2}( \rho^{(t_{\alpha}, s_{\alpha})}, \sigma_1^{t_{\alpha}} \otimes \delta_1^{s_{\alpha}})=m_{H_1 \times H_2}(\rho, \sigma_{\alpha} \otimes \delta_{\alpha}).$$
Il en résulte que $\sigma_{\alpha} \otimes \delta_{\alpha} \in \mathcal{R}_{H_1 \times H_2}(\rho)$. De plus, comme $\rho|_{H_1 \times H_2}$ satisfait à la propriété de graphe, la représentation $\delta_{\alpha}$ est déterminée uniquement par $\sigma_{\alpha}$.\\
(2) Soit $\widetilde{\sigma} \in \mathcal{R}_{\Gamma}(\rho) \cap \mathcal{R}_{\Gamma}(\pi_1 \otimes \pi_2)$.
 Si $\sigma_1 \otimes \delta_1 \in \mathcal{R}_{H_1 \times H_2} (\widetilde{\sigma})$. Par la Proposition \ref{dimesionsmall1}(1),
on voit que $\sigma_1 \otimes \delta_1 \in \mathcal{R}_{H_1 \times H_2}(\rho)$. Pour chaque élément $\sigma_{\alpha} \otimes \delta_{\alpha} \in
\mathcal{R}_{H_1 \times H_2}(\rho)\cap \mathcal{R}_{H_1 \times H_2}( \pi_1 \otimes \pi_2)$, on peut prendre $(t_{\alpha}, s_{\alpha}) \in \Gamma$
tel que $\sigma_1^{t_{\alpha}} \otimes \delta_1^{s_{\alpha}}\simeq \sigma_{\alpha} \otimes \delta_{\alpha}$.
Il en résulte que $\sigma_{\alpha} \otimes \delta_{\alpha} \in \mathcal{R}_{H_1 \times H_2}(\widetilde{\sigma})$, i.e
$\mathcal{R}_{H_1 \times H_2}(\rho)\cap \mathcal{R}_{H_1 \times H_2}( \pi_1 \otimes \pi_2)=\mathcal{R}_{H_1 \times H_2}(\widetilde{\sigma})$.
On va montrer que $\mathcal{R}_{\Gamma}(\rho) \cap \mathcal{R}_{\Gamma}(\pi_1 \otimes \pi_2) = \{ \widetilde{\sigma}\}$.
Si $\widetilde{\sigma}\neq  \widetilde{\sigma}' \in \mathcal{R}_{\Gamma}(\rho) \cap \mathcal{R}_{\Gamma}(\pi_1 \otimes \pi_2)$, la m\^eme démonstration comme ci-dessus, on a
$\mathcal{R}_{H_1 \times H_2}(\widetilde{\sigma})=\mathcal{R}_{H_1 \times H_2}(\widetilde{\sigma}')$. Par la théorie de Clifford, on obtient  $\widetilde{\sigma}' \simeq
\widetilde{\sigma} \otimes \chi_{\Gamma/{H_1 \times H_2}}$ pour un caractère non-trivial $\chi_{\Gamma/{H_1 \times H_2}} \in \Irr(\Gamma/{H_1 \times H_2})$. On prend un
m\^eme  espace vectoriel $\widetilde{W}$ réalizé les représentations $\widetilde{\sigma}$ et $\widetilde{\sigma}'$. On prend un élément $(\sigma_1 \otimes \delta_1, W_1 \otimes
W_2)$ de $\mathcal{R}_{H_1 \times H_2}( \widetilde{\sigma})= \mathcal{R}_{H_1 \times H_2}(\widetilde{\sigma}')$. Si $0 \neq F \in \Hom_{\Gamma}(\rho, \widetilde{\sigma})$,
et $0 \neq G \in \Hom_{\Gamma}(\rho, \widetilde{\sigma}')$, on considère le morphisme  $\rho \oplus \rho\otimes \chi^{-1}_{\Gamma/{H_1 \times H_2}} \stackrel{F \oplus G}{\longrightarrow} \widetilde{\sigma}$.
Par l'application canonique $\widetilde{W}\stackrel{p}{\longrightarrow} W_1 \otimes W_2$, nous obtenons deux éléments non trivials $p\circ F, p\circ G$ dans $\Hom_{H_1  \times H_2}(\rho,
\sigma_1 \otimes \delta_1)$. Donc $p\circ F$ est proportionnel à $p\circ G$, ceci implique $p\circ F =c p\circ G$ pour une constante $c\in \C^{\times}$, i.e. $p\circ \big( F-cG \big)=0$,  d'où $F =cG$. Cela
contradict à l'hypothèse. Finalement, on trouve $\mathcal{R}_{\Gamma}(\rho)\cap \mathcal{R}_{\Gamma}(\pi_1 \otimes \pi_2)= \{\widetilde{\sigma}\}$. Comme $\rho|_{H_1 \times H_2}$ est une représentation de sans multiplicité du groupe $H_1 \times H_2$,
on obtient $m_{\Gamma}(\pi_1 \otimes \pi_2, \widetilde{\sigma})=1$. \\
(3) $$1 \leq m_{G_1 \times G_2}(\pi, \pi_1 \otimes \pi_2)=m_{\Gamma} (\rho, \pi_1 \otimes \pi_2)=m_{\Gamma}
(\rho, \widetilde{\sigma})\stackrel{ \textrm{ Proposition \ref{dimesionsmall1}(3)}}{\leq} 1,$$
d'où le résultat.
\end{proof}
\ \\
\textbf{La preuve du Théorème \ref{abeliengraphedefini}}:\\
\underline{La première étape--- cas abélien: Supposons que $G_i/H_i$ est un groupe abélien pour $i=1, 2$}.\\
La propriété de sans multiplicité vient du Lemme \ref{principallemmedegraphe}(3). Il suffit de montrer que $\pi$ satisfait à la propriété de graphe.  Supposons $\pi_1 \otimes \pi_2, \pi_1 \otimes \pi_2' \in \mathcal{R}_{G_1\times G_2}(\pi)$. D'après le Lemme \ref{principallemmedegraphe}, on voit que
 $$\mathcal{R}_{\Gamma}(\pi_1 \otimes \pi_2) \cap \mathcal{R}_{\Gamma}(\rho) =\{ \widetilde{\sigma}\}, \quad \widetilde{\sigma}|_{H_1 \times H_2}=
\oplus_{\sigma_{\alpha} \otimes \delta_{\sigma} \in \mathcal{R}_{H_1 \times H_2}(\rho)\cap \mathcal{R}_{H_1 \times H_2}(\pi_1 \otimes \pi_2)} m \sigma_{\alpha} \otimes \delta_{\alpha}.$$
 et
 $$\mathcal{R}_{\Gamma}(\pi_1 \otimes \pi_2') \cap \mathcal{R}_{\Gamma}(\rho) =\{ \widetilde{\sigma}'\}, \quad
\widetilde{\sigma}'|_{H_1 \times H_2}= \oplus_{\sigma_{\alpha} \otimes \delta_{\sigma}' \in \mathcal{R}_{H_1 \times H_2}(\rho)\cap \mathcal{R}_{H_1 \times H_2}(\pi_1 \otimes \pi_2')}
m'\sigma_{\alpha} \otimes \delta_{\alpha}'.$$
 Donc par la propriété de graphe de la représnetation $\rho|_{H_1 \times H_2}$,   on obtient $\delta_{\alpha}= \delta_{\alpha}'$ et  $\mathcal{R}_{H_2}(\pi_2)=\mathcal{R}_{H_2}(\pi_2')$. En vertu de la théorie de Clifford, $\pi_2' \simeq \pi_2 \otimes \chi_{G_2/{H_2}}$
pour un caractère $\chi_{G_2/H_2}$ du groupe $G_2/{H_2}$, et $m=m'$.\\
 D'autre part, en suivant la démonstration du Lemme \ref{principallemmedegraphe} (2), on montre que $\widetilde{\sigma}=\widetilde{\sigma}'$.
 Par le Corollaire \ref{equalitydequotient}, on voit que  $\mathcal{R}_{\Gamma}(\pi_1 \otimes \pi_2)=\mathcal{R}_{\Gamma}(\pi_1 \otimes \pi_2')$. Cela implique qu'il existe un caractère $\chi_{G_1 \times G_2 /{\Gamma}}$ tel que
 $$\pi_1 \otimes \pi_2' \simeq \pi_1 \otimes \pi_2 \otimes \chi_{G_1 \times G_2/{\Gamma}}.$$
 Ainsi $1 \otimes \chi_{G_2/{H_2}}$ est trivial sur $\Gamma$. Puisque l'application $\Gamma/{H_1\times H_2} \longrightarrow G_2/H_2$ est surjective, il est clair que $\chi_{G_2/{H_2}}=1$,  donc $\pi_2 \simeq \pi_2'$.\\
\underline{La deuxième étape--- raisonnement par récurrence }:
On vient de traiter le cas où $G_1/H_1$ est abélien.  Pour le cas général, on procède par récurrence sur
le cardinal de $G_1/H_1$. Si $G_1 \neq H_1$, introduisons une sous-groupe $G_1'$ de $G_1$, contenant $H_1$, et tel que $G_1/{G_1'}$ soit abélien. Notons $G_2'$ l'image inverse dans $G_2$
 de $\gamma(G_1'/H_1) \subset G_2/H_2$. Alors $\gamma$ induit un homomorphisme surjectif $\gamma'$ de $G_1'/H_1$ sur $G_2'/H_2$; on note $\Gamma'$ le sous-groupe de $\Gamma$ contenant $H_1 \times H_2$ et tel que $\Gamma'/{H_1 \times H_2}$ soit le graphe de $\gamma'$. Par récurrence, la représentation $\pi'=c-\Ind_{\Gamma'}^{G_1' \times G_2'} \rho$ est une représentation de graphe forte. D'autre part, $\gamma$ induit par passage aux quotients un homomorphisme surjectif de $G_1/{G_1'}$ sur $G_2/{G_2'}$, dont le graphe est $\Gamma(G_1' \times G_2')/{G_1' \times G_2'}$. Considérons la représentation $\rho'=c-\Ind_{\Gamma}^{\Gamma(G_1' \times G_2')} \rho$; sa restriction à $G_1' \times G_2'$ est $\pi'$. Par le cas abélien traité dans la première étape, on obtient que $c-\Ind_{\Gamma(G_1' \times G_2')}^{G_1 \times G_2} \rho'$ est une représentation de graphe forte de $G_1 \times G_2$. Mais cette représentation n'est pas autre que $c-\Ind_{\Gamma}^{G_1 \times G_2} \rho$.

\begin{rem}
Pour le cas général, on peut traiter le sous-ensemble $\Irr^{ad}(G_i)$ de $\Irr(G_i)$ dont les objets sont les $\pi_i$ où chacun  $\pi_i|_{H_i}$ soit admissible.
\end{rem}

\begin{rem}
Tous les résultats dans cette sous-section sont aussi vrais, si on considère `` graphe à droite''.
\end{rem}
\section{Les représentations de bigraphe forte}\label{lesrepresentationdebigrapheforte}
Dans cette sous-section,  on supposera que toute représentation irréductible du groupe localement compact totalement discontinu est admissible.
\begin{prop}[ Produits]\label{bigraphesproduits}
Soient $G_i, H_i$ des groupes localement compacts totalement discontinus pour $i=1, 2$. Soit $(\pi, V)$ une représentation lisse de $(G_1\times G_2) \times (H_1\times H_2)$,
\begin{itemize}
\item[(1)] Supposons que  pour chaque index $i$, $ \pi|_{G_i \times H_i} \textrm{ satisfait à la propriété  de bigraphe}$.  Alors $\pi$ est une représetnation de bigraphe du groupe $(G_1 \times G_2) \times (H_1 \times H_2)$.
\item[(2)] Supposons que $\pi$ est une représentation de graphe du groupe $G_1\times \cdots \times G_n$ et qu'il existe un couple $(G_i , G_j)$ de l'ensemble $\{ G_1, \cdots, G_n\}$ tel que $\pi|_{G_i\times G_j}$ satisfait à la condition de sans multiplicité.  Alors $\pi$ est une représentation de bigraphe forte du groupe $G_1 \times \cdots \times G_n$.
\end{itemize}
\end{prop}
\begin{proof}
C’est une conséquence  de la Proposition \ref{graphesproduits}.
\end{proof}

\begin{prop}\label{troisgroupesdebigrapheforte}
Soient $G_1, G_2, H$ des groupes localement compacts totalement discontinus et $H$ un groupe abélien, $\gamma$ un automorphisme de $H$, soit
$(\pi, V) \textrm{ une représentation lisse de } G_1 \times G_2 \times H$. Gr\^ace au morphisme des groupes
$$(G_1 \times H) \times (G_2 \times H) \longrightarrow G_1\times G_2 \times H$$
$$[(g_1, h_1), (g_2,h_2) \longmapsto (g_1g_2, h_1\gamma(h_2))]$$
on obtient une représentation  $\widetilde{\pi}$ du groupe
$(G_1 \times H) \times (G_2 \times H)$.
\begin{itemize}
\item[(1)] Si $\pi|_{G_1 \times G_2}$ est une représentation de bigraphe, alors $\widetilde{\pi}$ l'est aussi.
\item[(2)] Si $\pi|_{G_1\times G_2}$ est une représentation de bigraphe forte, alors $\widetilde{\pi}$ l'est aussi.
\end{itemize}
\end{prop}
\begin{proof}
(1) Soit $(\pi_1 \otimes \chi_1) \otimes (\pi_2 \otimes \chi_2) \in \mathcal{R}_{(G_1 \times H) \times (G_2 \times H)}(\widetilde{\pi})$.
 Si $$0 \neq F \in \Hom_{(G_1\times H) \times (G_2 \times H)}(\pi, (\pi_1 \otimes \chi_1) \otimes (\pi_2 \otimes \chi_2)).$$
On a
$$F(\pi((g_1\otimes g_2), h_1\gamma(h_2))v)=\pi_1(g_1) \otimes \pi_2(g_2) F(v) \chi_1(h_1) \chi_2(h_2) \textrm{ pour } v\in V, g_i \in G_i, h_i \in H.$$
En particulier, prenons $g_1=g_2=1$, $h_1=h\in H$ et $h_2=\gamma^{-1}(h^{-1}) \in H$, nous obtenons
$$F(v)= F(v) \chi_1(h) \chi_2\Big(\gamma^{-1}(h^{-1})\Big) \textrm{ pour tout } v\in V.$$
Comme $F\neq 0$ et que $\gamma$  est un isomorphisme , on trouve $\chi_2=\chi_1^{\gamma}$  où $\chi_1^{\gamma}(h):= \chi_1\Big(\gamma(h)\Big)$ pour $h\in H$. Si $\theta_{\pi}$ est l'application de Howe de $\pi|_{G_1 \times G_2}$, alors on obtient une application de Howe de $\widetilde{\pi}$:
$$\theta: \mathcal{R}_{G_1 \times H_1}(\widetilde{\pi}) \longrightarrow \mathcal{R}_{G_2 \times H_2}(\widetilde{\pi});$$
$$ \pi_1 \otimes \chi_1 \longmapsto \theta_{\pi}(\pi_1) \otimes \chi_1^{\gamma},$$
donc on a trouvé  que $(\pi, (G_1 \times H_1) \times (G_2 \times H_2))$ est une représentation de bigraphe.\\
(2) Il suffit de montrer que la représentation $\widetilde{\pi}$ satisfait à la propriété de sans multiplicité.
Comme $m_{G_1 \times H_1 \times G_2 \times H_2}(\widetilde{\pi}, (\pi_1 \otimes \chi_1) \otimes \pi_2 \otimes \chi_1^{\gamma}) \leq  m_{G_1\times G_2}(\pi, \pi_1 \otimes \pi_2)$, on obtient le résultat.
\end{proof}

\begin{rem}\label{ouvertmorphisme}
Si le morphisme précédent, $(G_1 \times H) \times (G_2 \times H) \longrightarrow G_1\times G_2 \times H$, se factorise par $(G_1 \times H) \times (G_2 \times H) \longrightarrow G_1H \times G_2H$,  où $G_i \times H \longrightarrow G_i H$ est un morphisme de groupes ouvert surjectif pour $i=1,2$, alors  le résultat dans le Lemme \ref{troisgroupesdebigrapheforte} est aussi vrai pour la représentation $(\pi, G_1H \times G_2H)$.
\end{rem}
\begin{proof}
Comme $p_i$ est un morphisme ouvert, toute représentation lisse du groupe  $G_iH$ est donnée par une représentation lisse du groupe $G_i \times H$ qui est triviale en sous-groupe $\ker(p_i)$. Cela entra\^ine le résultat.
\end{proof}

Soient $G_1, G_2$ des groupes localement compacts totalement discontinus, $H_1$ (resp. $H_2$) un sous-groupe distingué fermé de $G_1$ (resp. $G_2$). Supposons que  $G_i/H_i$ est soluble pour $i=1,2$. On supposera que toute représentation irréductible du groupe $G_i$ ou $H_i$ est admissible. Si $\pi_i \in \Irr(G_i)$, on supposera que $\pi_i|_{H_i}$ est aussi admissible.

\begin{thm}\label{abelienbigraphedefini}
Si $\Gamma$ un sous-groupe distingué fermé  de $G_1\times G_2$ contenant $H_1\times H_2$  tel que
$$\Gamma/{H_1\times H_2} \textrm{ est le graphe d'un morphisme bijectif } \gamma: G_1/H_1 \longrightarrow G_2/H_2.$$
Supposons que $(\rho, V)$ est une représentation de bigraphe forte du groupe $H_1 \times H_2$ qui se prolonge en une représentation du groupe $\Gamma$. Alors
$\pi=c-\Ind_{\Gamma}^{G_1\times G_2} \rho$ est aussi une représentation de bigraphe forte du groupe $G_1\times G_2$.
\end{thm}
\begin{proof}
C'est une consequence du Théorème \ref{abeliengraphedefini}.
\end{proof}
\begin{prop}\label{larelationdebigraphe}
Conservons  les notations et hypothèses du Théorème \ref{abelienbigraphedefini}. Soient $\pi_1 \in \Irr(G_1)$, $\pi_2 \in \Irr(G_2)$ tels que
$$ \pi_1=\oplus_{\sigma \in \mathcal{R}_{H_1}(\pi_1)} m_1 \sigma \textrm{ et } \pi=\oplus_{\delta \in \mathcal{R}_{H_2}(\pi_2)} m_2 \delta.$$
Supposons que $\pi_1\otimes \pi_2 \in \mathcal{R}_{G_1\times G_2}(\pi)$.
Alors\\
(1)  il existe une application bijective
 $$\theta_{\rho}: \mathcal{R}_{H_1}(\pi_1) \longrightarrow \mathcal{R}_{H_2}(\pi_2); \sigma_{\alpha} \longmapsto \delta_{\alpha}$$
telle que $\sigma_{\alpha} \otimes \delta_{\alpha} \in \mathcal{R}_{H_1 \times H_2}(\rho)$ et
$\sigma_{\alpha} \otimes \delta_{\beta} \notin \mathcal{R}_{H_1\times H_2}(\rho)$ pour $\alpha \neq \beta$.\\
(2) $m_1=1$ ssi $m_2=1$.
\end{prop}
\begin{proof}
(i) Comme $\pi_1 \otimes \pi_2 \in \mathcal{R}_{G_1 \times G_2}(\pi)$, par le Théorème Réciproque de Frobenius,
on obtient $\mathcal{R}_{\Gamma}(\pi_1 \otimes \pi_2)\cap \mathcal{R}_{\Gamma}(\rho) \neq \emptyset$, a fortiorie,
$\mathcal{R}_{H_1\times H_2}(\pi_1 \otimes \pi_2) \cap \mathcal{R}_{H_1 \times H_2}(\rho) \neq \emptyset$.
On peut supposer que $\sigma_1 \otimes \delta_1 \in \mathcal{R}_{H_1 \times H_2}(\pi_1 \otimes \pi_2) \cap \mathcal{R}_{H_1 \times H_2}(\rho) $.
D'après la théorie de Clifford, il existe un élément $t_{\alpha} H_1$ de $G_1/{H_1}$ tel que $\sigma_1^{t_{\alpha}} \simeq \sigma_{\alpha}$ où $\sigma_1^{t_{\alpha}}(h_1):= \sigma_1(t_{\alpha}^{-1}h_1 t_{\alpha})$. Soit $\gamma(t_{\alpha} H_1)=s_{\alpha}H_2 \in G_{2}/{H_2}$; alors $(t_{\alpha}, s_{\alpha}) H_1 \times H_2 $ appartient à $\Gamma/{H_1 \times H_2}$. Supposons que $(t_{\alpha}, s_{\alpha})\in \Gamma$, on note $\delta_{\alpha}=\delta_1^{s_{\alpha}}$. Donc
$$\sigma_{\alpha} \otimes \delta_{\alpha} = \sigma_1^{t_{\alpha}} \otimes \delta_1^{s_{\alpha}} \in \mathcal{R}_{H_1 \times H_2}\big( (\pi_1 \otimes \pi_2)^{(t_{\alpha}, s_{\alpha})}\big) \cap \mathcal{R}_{H_1 \times H_2}(\rho^{(t_{\alpha}, s_{\alpha})})= \mathcal{R}_{H_1 \times H_2}(\pi_1 \otimes \pi_2) \cap \mathcal{R}_{H_1\times H_2}(\rho).$$
Comme $\rho|_{H_1 \times H_2}$ satisfait à la propriété de bigraphe, on sait que $\delta_{\alpha} \ncong \delta_{\beta}$ si $\alpha \neq \beta $. D'autre part, si on choisit un élément $\delta_{\beta}$ de $\mathcal{R}_{H_2}(\pi_2)$, il aussi existe un élément $\sigma_{\beta}$ tel que $\sigma_{\beta} \otimes \delta_{\beta} \in \mathcal{R}_{H_1 \times H_2}(\rho)$, donc on obtient une application bijective $\theta_{\rho}$. Comme $\rho|_{H_1\times H_2}$ satisfait à la propriété de bigraphe, on trouve
$\sigma_{\alpha} \otimes \delta_{\beta} \notin \mathcal{R}_{H_1 \times H_2}(\rho)$ pour $\alpha \neq \beta$.\\
(ii) On suppose $m_1=1$. On procède par récurrence sur le cardinal de $G_1/{H_1}$.
 Reprenons les notations dans la démonstrations du Théorème \ref{abeliengraphedefini}(cf. la deuxième étape). D'après le Théorème \ref{cliffordadmissible}(4),  on suppose $\pi_1|_{G_1'}= \oplus_{\sigma_{\alpha}' \in \mathcal{R}_{G_1'}(\pi_1)} \sigma_{\alpha'}$ et $\pi_2|_{G_2'}=\oplus_{\delta_{\beta}' \in \mathcal{R}_{G_2'}(\pi_2)}  \delta_{\beta}'$. De plus, on suppose $\sigma_{\alpha}'\otimes \delta_{\beta}' \in \mathcal{R}_{G_1' \times G_2'}(\rho')$ ssi $\alpha=\beta$. Par récurrence, pour chaque indice fixé $\alpha$, on a $\sigma_{\alpha}'= \oplus_{\sigma_{\alpha i}\in \mathcal{R}_{H_1}(\sigma_{\alpha}')} \sigma_{\alpha i}$ et $\delta_{\beta}'=\oplus_{\delta_{\beta j} \in \mathcal{R}_{H_2}(\delta_{\beta }')} \delta_{\beta j}$. On suppose $\theta_{\rho}(\sigma_{\alpha i})=\delta_{\alpha i}$ où $\theta_{\rho}$ est l'application de Howe de $\rho$. Par hypothèse, $\sigma_{\alpha i} \ncong \sigma_{\beta j}$ si $\alpha \neq \beta$ ou $i\neq j$. Cela en résulte que $\delta_{\alpha i}=\theta_{\rho}(\sigma_{\
alpha 
i}) \ncong \delta_{st}=\theta_{\rho}(\sigma_{st})$ si $s\neq \alpha $ ou $t\neq i$, d'où le résultat.
\end{proof}

\begin{prop}
Conservons  les notations et hypothèses du Théorème \ref{abelienbigraphedefini} sauf que $G_1/H_1$ est soluble. \\
  Si $G_1/H_1$ est un groupe abélien d'ordre fini,  $\chi^k \in \Irr(\Gamma/H_1 \times H_2)$ pour $k=1, \cdots, n$, on note  $\pi^k=c-\Ind_{\Gamma}^{G_1\times G_2} (\rho \otimes \chi^k)$.
\begin{itemize}
\item[(1)] $\pi^k$ est aussi une représentation de bigraphe forte du groupe $G_1\times G_2$.
\item[(2)] Si on pose $\mathcal{R}_{G_1 \times G_2}(\pi^k)|_{H_1 \times H_2}: =\{ \sigma_1 \otimes \sigma_2| $ il existe un élément $\pi_1 \otimes \pi_2 \in \mathcal{R}_{G_1 \times G_2}(\pi^k)$ tel que $\sigma_i\in \mathcal{R}_{H_i}(\pi_i)$ et $\sigma_1 \otimes \sigma_2 \in \mathcal{R}_{H_1 \times H_2} ( \rho) \}$, alors $\mathcal{R}_{H_1 \times H_2} (\rho)=\cup_{k=1}^n \mathcal{R}_{G_1 \times G_2}( \pi^k)|_{H_1 \times H_2}$.
\end{itemize}
\end{prop}
\begin{proof}
(1) Comme $\rho\otimes \chi $ aussi prolonge $\rho\otimes \chi|_{H_1\times H_2}=\rho|_{H_1\times H_2}$, on obtient le résultat.\\
(2) Si $\sigma_1 \otimes \sigma_2 \in \mathcal{R}_{H_1 \times H_2}(\rho)$, on a
$$1= m_{H_1\times H_2}(\rho, \sigma_1\otimes \sigma_2) = m_{\Gamma}\big(\rho, \Ind_{H_1\times H_2}^{\Gamma} (\sigma_1 \otimes \sigma_2)\big).$$
Cela implique que
$$\Ind_{H_1\times H_2}^{\Gamma} (\sigma_1 \otimes \sigma_2) = \oplus_{\chi_i \in \Irr(\Gamma/{H_1\times H_2})} \widetilde{\sigma}^0 \otimes \chi_i,$$
où $\widetilde{\sigma}^0 $ est la représentation irréductible unique du groupe $\Gamma$ telle que
$$m_{\Gamma}(\rho, \widetilde{\sigma}^0 )=1.$$
Pour chaque $\pi_1\otimes \pi_2\in  \Irr(G_1\times G_2)$ tel que $ \pi_1\otimes \pi_2|_{H_1\times H_2}$ continent $\sigma_1\otimes \sigma_2$, il existe $\chi^k \in \Irr(\Gamma/{H_1 \times H_2})$ tel que $\pi_1\otimes \pi_2|_{\Gamma}$ contient $\widetilde{\sigma}^0 \otimes \chi^k$. Donc
$$m_{G_1\times G_2}(\pi^k, \pi_1\otimes \pi_2) = m_{G_1\times G_2}(c-\Ind_{\Gamma}^{G_1\times G_2} \rho\otimes \chi^k, \pi_1\otimes \pi_2)= m_{\Gamma}(\rho \otimes \chi^k, \pi_1\otimes \pi_2)$$
 $$\geq  m_{\Gamma}(\rho\otimes \chi^k, \widetilde{\sigma}^0\otimes \chi^k)=m_{\Gamma}(\rho, \widetilde{\sigma}^0 )=1.$$
 Donc on a montré que
 $$\mathcal{R}_{H_1 \times H_2} (\rho)=\cup_{k=1}^n \mathcal{R}_{G_1 \times G_2}( \pi^k)|_{H_1 \times H_2}.$$
 \end{proof}

\section{Les paires duales de Howe pour $Sp(W)$}\label{howeduale}
Notre reférence dans cette sous-section est \cite{MVW}.
Dans cette sous-section, on désigne par $F$ un corps non archimédien de caractéristique impaire et de  caractéristique résiduelle $p $.
\begin{defn}
Soit $G$ un groupe. Un sous-groupe $H \subset G$ tel que le double commutant de $H$ dans $G$ soit égal à $H$ sera appelé  \textbf{un sous-groupe de Howe}  de $G$. Si $H'=Z_{G}(H)$ est le commutant de $H$ dans $G$, on dira que $(H,H')$ est \textbf{une paire duale} dans $G$.
\end{defn}

Soit $D$ un corps de dimension finie sur son centre $F$(pas nécessairement commutatif), muni d'une involution $\tau$. Soit $W$ un espace vectoriel à droite  sur $D$, de dimension $n$, muni d'un produit $\epsilon$-hermitien, où $\epsilon \tau(\epsilon)=1$.  $W$ est dit de \textbf{type 2}(resp.  de \textbf{type 1})  si le produit hermitien est nul(resp.  non dégénéré).
\begin{defn}
Soit $W$ un $D$-espace à droite $\epsilon$-hermitien de dimension finie de type 1 ou 2, une paire duale $(H, H')$ de $U(W)$ est dite \textbf{réductive} si
\begin{itemize}
\item[(i)] $W$ est $HD$ et $H'D$-semi-simple.
\item[(ii)] $H$ et $H'$ sont réductifs.
\end{itemize}
(Ces deux conditions sont probablement équivalentes). On dit alors que $H$ est \textbf{un sous-groupe de Howe réductif} de $U(W)$.
\end{defn}

\begin{defn}
Une paire duale $(H,H')$ de $U(W)$ est dite irréductible s'il n'existe pas de décomposition orthogonale de $W$ stable par $HH'D$.
\end{defn}

\begin{lemme}
Toute paire réductive duale de $U(W)$ est produit de paires réductives duales irréductibles.
\end{lemme}
\begin{proof}
Ceci découle de [\cite{MVW} Pages, 10-11]
\end{proof}

Ci-dessous, nous citons la classification des paires duales irréductibles de $Sp(2n,F)$ dans  [\cite{MVW}, Page 15].
Soit $t_{D/F}\in \Hom_F(D,F)$ tel que la forme bilinéaire $(d,d') \longmapsto t_{D/F}(dd')$ pour $d,d' \in D$ soit non dégénérée.

\begin{lemme}\label{leproduittensoriel}
Si $(W_1, \langle, \rangle_1)$, $(W_2, \langle, \rangle_2)$ sont deux espaces sur $D$ respectivement à droite et à gauche, $\epsilon_i$-hermitiens de type 1 tels que $-1= \epsilon_1 \epsilon_2$, alors le produit tensoriel $W=W_1\otimes_D W_2$ muni de la forme
$$<< w_1 \otimes w_2, w_1'\otimes w_2' >>= t_{D/F}(< w_1, w_1'>_1 \tau(< w_2, w_2'>_2)), \quad w_i, w_i' \in W_i$$
est symplectique. \underline{Inversement}, toute décomposition d'un espace symplectique $(W, <, >)$ en produit tensoriel hermitien est de ce type.
\end{lemme}
\begin{proof}
Ceci découle de [\cite{MVW}, Pages 14-15, Lemme].
\end{proof}

\begin{rem}
\begin{itemize}
\item[(1)]  Soient $(W, \langle, \rangle)$ un espace anti-hermitien sur $(D, \tau)$, $t_{D/F}$ un élément de $\Hom_F(D,F)$ tels que $(d,d') \longmapsto t_{D/F}(dd')$ est non dégénérée, on obtient une forme symplectique $t_{D/F}(\langle, \rangle)$ sur le $F$-espace $W$, l'espace symplectique $(W, t_{D/F}(\langle, \rangle))$ sera dit
\textbf{déduit} de $(W, \langle, \rangle )$ et $t_{D/F}$ par restriction des scalaires.
\item[(2)] Supposons l'espace $W'$ déduit de $W$ par restriction des scalaires par une extension $F'/F$. Par la classification des sous-groupes de Howe réductifs des groupes classiques dans  [\cite{MVW}, Page 13], on sait qu'une paire duale non triviale de $Sp(W)$ reste une parie  duale non triviale de $Sp(W')$. Comme $\Cent_{Sp(W')}(\{\pm 1\})= Sp(W')$, on voit que la paire duale triviale $(\{\pm1\}, Sp(W))$ ne peut pas  \^etre une paire duale du groupe $Sp(W')$.
\end{itemize}
\end{rem}

\begin{thm}\label{MVWclassificationdepaires}
Voici la classification des paires duales irréductibles de $Sp(W)$ qui ne proviennent pas par restriction des scalaires de $Sp(W')$ satisfaisant à $(\dim_{F'}W') [F':F]=\dim_FW$ pour une extension $F'$ sur $F$.
\begin{itemize}
\item[(1)] Les paires duales de type $2$: $(GL_D(X_1), GL_D(X_2))$.\\
Si $W$ est totalement isotrope et non dégénéré pour toute décomposition d'un lagrangien $X$ de $V$ en produit tensoriel $X= X_1 \otimes_D X_2$.\\
\item[(2)] Les paires duales de type 1:
\begin{itemize}
\item[(i)] $(O(W_1), Sp(W_2))$.\\
Ici $W \simeq W_1 \otimes_F W_2$ avec $W_1$(resp.  $W_2$) un espace orthogonal (resp.  symplectique) sur $F$.
\item[(ii)] $(U(W_1), U(W_2))$.\\
Ici $W\simeq W_1\otimes_D W_2$ où $D/F$ est une extension quadratique ou un corps de quaternions muni de l'involution canonique. $W_1$(resp.  $W_2$) est un espace hermitien (resp.  anti-hermitien) sur $D$ et $\dim_D W_2 \neq 1$ si $D$ est un corps de quaternions.
\end{itemize}
\end{itemize}
\end{thm}
\begin{proof}
Ceci découle de \cite{MVW}.
\end{proof}
Soit $W$ un $F$-espace vectoriel de dimension $2n$, muni d'une forme symplectique $\langle, \rangle$. 
On note $\overline{Sp}(W)$(resp. $\widetilde{Sp}(W)$) le groupe métaplectique associé à l'extension de centre  
$\mu_8=\langle e^{\frac{2\pi i}{8}}\rangle$( resp. $\C^{\times}$). Si $H$ est un sous-groupe fermé de $Sp(W)$,
 nous  notons son image réciproque dans $\overline{Sp}(W)$ (resp. $\widetilde{Sp}(W)$) par $\overline{H}$( resp. $\widetilde{H}$).\\

Supposons $W= W_1\oplus W_2$ une somme orthogonale. On considère l'application canonique
$$Sp(W_1) \times Sp(W_2) \stackrel{p}{\longrightarrow } Sp(W).$$
On note $\overline{Sp(W_1)\times Sp(W_2)}$ l'image réciproque de $p(Sp(W_1) \times Sp(W_2))$ dans $\overline{Sp}(W)$.
\begin{lemme}\label{morphismedegroupes}
On a un morphisme de groupes
$$\overline{Sp}(W_1) \times \overline{Sp}(W_2) \stackrel{\overline{p}}{\longrightarrow} \overline{Sp(W_1) \times Sp(W_2)},$$
$$ [(g_1, \epsilon_1), (g_2, \epsilon_2)]  \longmapsto [(g_1,g_2), \epsilon_1\epsilon_2 c_{Rao}((g_1, 1), (1, g_2))],$$
où $c_{Rao}$ est le $2$-cocycle de Rao[cf. \cite{MVW}], et où on écrit $(g_1,g_2)$ pour l'image $p(g_1, g_2)$ dans $Sp(W)$. En particulier, en considèrant $\overline{p}|_{\overline{Sp}(W_1)}$ et $\overline{p}|_{\overline{Sp}(W_2)}$, on obtient
$$c_{Rao}(g_1,g_1')=c_{Rao}((g_1, 1), (g_1', 1))$$ et
$$c_{Rao}(g_2,g_2')=c_{Rao}((1,g_2), (1, g_2'))$$
pour $g_1, g_1'\in Sp(W_1)$, $g_2, g_2'\in Sp(W_2)$.
\end{lemme}
\begin{proof}
Ceci découle de [\cite{HanM}, Pages 245-246].
\end{proof}

\begin{thm}\label{scindagesurlepaireduale}
Soit $(H_1, H_2)$ une paire duale irréductible de $Sp(W)$, soit $\overline{H_1}$ l'image réciproque dans $\overline{Sp}(W)$. Alors l'extension $\overline{H_1}$ de $H_1$ est scindée, sauf si $W$ est de la forme  $W_1\otimes_{F'} W_2$  et $H_1=Sp(W_1)$, où $W_1$ est un espace symplectique sur une extension $F'$ de $F$ et où $W_2$ est de dimension impaire.
\end{thm}
\begin{proof}
C'est une variante du théorème de Vigneras dans \cite{MVW}, on peut voir \cite{thesiswang} par les details.
\end{proof}
\section{La représentation de Weil}
Notre reférence dans cette sous-section est \cite{MVW}. Dans cette sous-section, on désigne par $F$ un corps non archimédien de caractéristique impaire.\\
Soit $W$ un espace vectoriel de dimension $2n$ sur $F$, muni d'une forme symplectique non dégénérée $\langle, \rangle$.
Le groupe d'Heisenberg associé à $(W, \langle, \rangle)$, est l'ensemble $W\times F$, muni de la topologie produit et de la loi de groupe
$$(w, t) (w',t'):= (w+w', t+ t' + \langle w, w' \rangle/2).$$
Notons $\xi: F \longrightarrow H$ le monomorphisme $t \longmapsto (0, t)$. Nous  fixons un caractère non trivial $\psi: F \longrightarrow \C^{\times}$.
\begin{thm}[ Stone, Von Neumann]\label{stoneVonNeumann}
A isomorphisme près, il existe une et une seule représentation lisse irréductible de $H$, notée par $\rho_{\psi}$, telle que
$$\rho_{\psi}\circ \xi(t)=\psi(t)  \textrm{ pour tout } t \in F.$$
\end{thm}
\begin{proof}
Ceci découle de [\cite{MVW}, Pages 28-29].
\end{proof}
Soit $A$ un sous-groupe fermé de $W$, posons
$$A^{\perp}= \{ w\in W, \textrm{ pour tout } a \in A, \psi(\langle w, a\rangle)=1 \}.$$
On note
$$\mathcal{A}=\{ A | A  \textrm{ un sous-groupe fermé de } W  \textrm{ tel que } A=A^{\perp} \}.$$
\begin{thm}[\cite{MVW}]
Pour chaque $A \in \mathcal{A}$, la représentation $c-\Ind_{A \times F^{\times}}^{H} 1 \cdot \psi$ réalise la représentation unique $\rho_{\psi}$ de $H$, associée à $\psi$, dans le Théorème \ref{stoneVonNeumann}.
\end{thm}
\begin{proof}
Voir [\cite{MVW}, Pages 28-29].
\end{proof}
Soit $(\rho_{\psi}, S)$  un modèle par $\rho_{\psi}$.  On pose
$$\widetilde{Sp_{\psi}}(W)=\{(g,M_g)| g\in Sp(W); M_g\rho_{\psi}(h)M_g^{-1} =\rho_{\psi}(gh) \textrm{ pour tout }h\in H\},$$
qui est un sous groupe topologique de $Sp(W) \times GL(S)$. On a aussi une suite exacte
$$ \quad 1\longrightarrow \C^{\times} \longrightarrow \widetilde{Sp}_{\psi}(W) \stackrel{p_{W}}{\longrightarrow} Sp(W) \longrightarrow 1$$
et il existe un unique isomorphisme
$$I: \widetilde{Sp}(W) \longrightarrow  \widetilde{Sp}_{\psi}(W)$$
tel que  le diagramme
\[\begin{CD}
1 @>>> \C^{\times} @>>> \widetilde{Sp}(W) @>p>> Sp(W) @>>> 1\\
  @.  @|  @V\sim VIV @| \\
1 @>>> \C^{\times} @>>> \widetilde{Sp}_{\psi}(W) @>p_W>> Sp(W)@>>> 1
\end{CD}\]
est commutatif.
Le composé de l'isomorphisme $I$ avec la projection $\widetilde{Sp}_{\psi}(W) \longrightarrow GL(S)$ donne une représentation du groupe
métaplectique, qu'on appelle souvent \textbf{la représentation de Weil}, liée à $\psi$, qu'on note $\omega_{\psi}$.
Pour utiliser la théorie des représentations des groupes localement compacts totalement discontinus, nous considérons par la suite la représentation de Weil du groupe métaplectique $Mp(W)$ qui est égal à $\overline{Sp}(W)$  sauf mention explicite.\\
La représentation de Weil vérifie les propriétés élémentaires ci-dessous[cf. \cite{MVW}, Pages 35-36]:\\
(1) $\omega_{\psi}$ est une représentation lisse admissible.\\
(2) $\omega_{\psi}= \omega_{\psi}^+ \oplus \omega_{\psi}^{-}$, où $\omega_{\psi}^{\mp} \in \Irr(Mp(W))$.\\
(3) La contragrédiente de $\omega_{\psi}$ est $\omega_{\psi^-}$.\\
\ \\
On définit un groupe $H \rtimes \overline{Sp}(W)$ sous la loi suivant:
$$((w_1,t_1), (g_1, \epsilon)) \cdot ((w_2,t_2), (g_2, \epsilon)):=((w_1, t_1) + (g_1 w_2, t_2), (g_1, \epsilon_1)\cdot (g_2, \epsilon_2)),$$
où $(w_i,t_i) \in H$ et $(g_i, \epsilon_i)\in \overline{Sp}(W)$.\\

$(\star \star)$ En effet, la représentation unique $(\rho_{\psi}, S)$ du groupe $H$ peut se prolonger dans $H\rtimes \overline{Sp}(W)$, notée par $\Omega_{\psi}$, alors $\Omega_{\psi}|_{\overline{Sp}(W)} \simeq \omega_{\psi}$ est la représentation de Weil associée à  $\psi$.\\

En utilisant le point $(\star \star)$ ci-dessus, pour chaque $A \in \mathcal{A}$, si on note $Stab_{A}(Sp(W))= G_A$ et $\overline{G_A}$ l'image  réciproque de $G_A$ dans $\overline{Sp}(W)$, alors la représentation  $\Omega_{\psi, H\rtimes \overline{G_A}}=c-\Ind_{A\rtimes \overline{G_A}}^{H\rtimes \overline{G_A}} 1 \cdot \psi \cdot 1$ est une représentation du groupe $H \rtimes \overline{G_A}$   satisfaisant  $\Omega_{\psi, H\rtimes \overline{G_A}}|_H \simeq \rho_{\psi}$. Donc la représentation $\Omega_{\psi,H \rtimes \overline{G_A}}|_{\overline{G_A}}$ peut prolonger en la représentation de Weil du groupe $\overline{Sp}(W)$.\\
Si on prend $A=$ un lagrangien maximal, alors on obtient un modèle de Sch\"odinger réalisé la représentation $\omega_{\psi}$.\\
Si on prend $A=$ un réseau tel que $A=A^{\perp}$, alors on  obtient un modèle latticiel réalisé $\omega_{\psi}$.\\

Soient $(H_1, H_2)$ une paire réductive  duale de $Sp(W)$, et  $\overline{H_1}, \overline{H_2}$ leurs images réciproques dans $\overline{Sp}(W)$, on voit que $\overline{H_1}$,$\overline{H_2}$  sont commutant  l'un de l'autre. On note $\mathcal{R}_{\omega_{\psi}}(\overline{H})=\{ \overline{\pi} \in \Irr(\overline{H}) | \textrm{ tel que } \Hom_{\overline{H}}(\omega_{\psi},  \overline{ \pi} ) \neq 0\}$  pour $H$ un sous-groupe fermé de $Sp(W)$.
\begin{thm}[Howe, Waldspurger]\label{howecorrespondanceclassique}
Si $F$ est local non archimédien de caractéristique résiduelle $\neq 2$, alors $\mathcal{R}_{\omega_{\psi}}(\overline{H_1\times H_2})$  est le bigraphe forte d'une bijection entre $\mathcal{R}_{\omega_{\psi}}(\overline{H_1})$  et $\mathcal{R}_{\omega_{\psi}}(\overline{H_2})$.
\end{thm}
\begin{rem}
Si les groupes $ \overline{H_i}$  satisfont aux conditions: $ \overline{H_i} \simeq H_i \times \mu_8$, par la définition de la représentation de Weil du groupe $\overline{Sp}(W)$,  on sait que
si $\overline{\pi_i}\in \mathcal{R}_{\omega_{\psi}}(\overline{H_i})$, il faut que $\overline{\pi_1} \simeq \pi_i \otimes \id_{\mu_8}$ pour telle représentation $\pi_i \in \Irr(H_i)$ et \underline{l'unique} représentation $\id_{\mu_8}$ du groupe $\mu_8$, donc $\omega_{\psi}|_{H_1 \times H_2}$ est aussi une représentation de bigraphe forte.
\end{rem}
\section{Un sous-groupe intermédiaire $\Gamma$ de $Mp(W)$}
Dans cette sous-section, on désigne par $F$ un corps non archimédien de caractéristique impaire. \\

Soient $D$ un corps, muni d'une involution $\tau$, $F$ le corps commutatif formé des points fixés de $\tau$, $W$ un espace vectoriel de dimension finie sur $D$, muni d'un produit $\epsilon$-hermitien $\langle, \rangle$( où $\epsilon=\pm 1$). On désigne par $U(W)$ le groupe unitaire de $(W, \langle, \rangle)$ et par $GU(W)$ le groupe de similitudes unitaires de $(W, \langle, \rangle)$. On note $\mathbb{G}_m$ le groupe algébrique multiplicatif.

On peut regarder $U(W)$ ou $GU(W)$ comme schémas en groupes de la façon suivante: on désigne par $\underline{Alg_F}$ la catégorie des $F$-algèbres commutatives unitaires, par $\underline{Gr}$ la catégorie des groupes. Soit $A$ un élément de $\underline{Alg_F}$, on note $D_A=D\otimes_FA$, $W_A=W\otimes_FA$. $D_A$ est une $F$-algère munie d'une  $\tau$-involution: $\tau(d\otimes a)=\tau(d)\otimes a$, pour $d\in D, a\in A$ et le $A_D$-module $W_A$ est muni d'une structure  $\epsilon$-hermitienne: $\langle w_i\otimes a_i, w_j \otimes a_j\rangle_A= \langle w_i, w_j \rangle \otimes a_ia_j$.  Alors $\textbf{GU}(W)$(resp.  $\textbf{U}(W)$) est le foncteur de $\underline{Alg_F}$ dans $\underline{Gr}$ défini par
$$\textbf{GU}(W) (A)= \{ g\in GL(W_A)| \langle gw_a, gw_a'\rangle_A = \lambda(g) \langle w_a, w_a'\rangle_A \textrm{ pour tous } w_a, w_a' \in W_A \textrm{ et un élément } \lambda(g) \textrm{ de } D_A^{\times}\}$$
(resp.  $\textbf{U}(W)(A)=\{ g\in GL(W_A)| \langle g w_a, gw_a' \rangle_A=\langle w_a, w_a'\rangle_A \textrm{ pour tous } w_a, w_a' \in W_A \}$).

On sait que $\textbf{U}(W)$, $\textbf{GU}(W)$ sont des $F$-groupes algébriques. Par définition, on a $U(W)=\textbf{U}(W)(F)$ et $GU(W)=\textbf{GU}(W)(F)$.
\begin{prop}
On a une suite exacte des schémas en groupes
$$1 \longrightarrow \textbf{U}(W) \longrightarrow \textbf{GU}(W) \longrightarrow \mathbb{G}_m \longrightarrow 1.$$
\end{prop}
\begin{proof}
Ceci découle de [ \cite{KMRT}, Page 346, Sequence (23.4)].
\end{proof}
Soient $\overline{F}$ une cl\^oture séparable de $F$, $\Gal=\Gal(\overline{F}/F)$ le groupe de Galois de $\overline{F}/F$. On a une suite exacte
$$1 \longrightarrow U(W) \longrightarrow GU(W) \stackrel{\lambda}{\longrightarrow } F^{\times} \longrightarrow H^1(\Gal, U(W)) \longrightarrow H^1(\Gal, GU(W)) \longrightarrow H^1(\Gal, \mathbb{G}_m) =1.$$

\begin{lemme}\label{wittlambda}
Si on a la décomposition de Witt $W=V_H \oplus W^0$ avec $V_H \simeq mH$ où $H$  est  le plan hyperbolique $\epsilon$-hermitien sur $D$ et $W^0$ est un espace anisotrope, alors
on a la suite exacte suivante:
$$1 \longrightarrow U(W) \longrightarrow GU(W) \longrightarrow \lambda(GU(W^0)) \longrightarrow 0.$$
\end{lemme}
\begin{proof}
Soit $g\in GU(W)$, $W=g(V_H) \oplus g(W^0)$ est aussi la décomposition de Witt de $W$. D'après le Théorème de Witt [cf. \cite{MVW}, Page 6], il existe un élément $g_0 \in U(W)$, tel que $g(W^0)=g_0(W^0)$, il en résulte que $g_0^{-1}g\in GU(W^0)$ et $\lambda(g_0^{-1} g)=\lambda(g)$. Donc on a montré que $\lambda(GU(W)) \subset \lambda(GU(W^0))$. Il reste à montrer que c'est tout.
Rappelons que le plan hyperbolique $\epsilon$-hermitien $H$  égal au $D$-espace vectoriel à droite $D\times D$ muni du produit $\epsilon$-hermitien $\langle (d_1,d_2), (d_1',d_2') \rangle = \tau(d_1)d_2'+ \epsilon \tau(d_2)d_1'$. Ainsi, $\lambda(GU(H))=F^{\times}$,  Par la suite exacte ci-dessus, on sait que $\lambda(GU(W^0))$ est un sous-groupe de $F^{\times}$, il en résulte que si on prend un élément $\lambda\in \lambda(GU(W^0))$, alors il existe des éléments $g_0\in GU(W^0)$ et $g_H\in GU(H)$ tels que $\lambda(g_0)=\lambda(g_H)=\lambda$. Donc on peut regarder $g=g_0 \times \underbrace{g_H \times \cdots \times g_H  }_{m}  $ comme un élément du groupe $GU(W)$, de plus, on voit aisément que $\lambda(g)=\lambda$, ceci termine la démonstration.
\end{proof}

\begin{prop}\label{lamedavalues}
(1) Si $D=F$, $\epsilon=-1, U(W)=Sp(W)$, $GU(W)=GSp(W)$, alors on a une suite exacte $$1 \longrightarrow Sp(W)  \longrightarrow GSp(W) \stackrel{\lambda}{\longrightarrow} F^{\times} \longrightarrow 1.$$
(2) Si $D=F$, $\epsilon=1$, $U(W)=O(W)$, $GU(W)=GO(W)$, supposons que on a la décomposition de Witt $W=W^0\oplus mH$; alors
\begin{itemize}
 \item[(i)] si $\dim_F W^0= 0,  3, 4$, alors on a une suite exacte:
 $$1 \longrightarrow O(W) \longrightarrow GO(W) \stackrel{\lambda}{\longrightarrow} F^{\times} \longrightarrow 1.$$
 \item[(ii)] si $\dim_F W^0=1$, alors on une suite exacte:
 $$1 \longrightarrow O(W) \longrightarrow GO(W) \stackrel{\lambda}{\longrightarrow} F^{\times 2} \longrightarrow 1. $$
 \item[(iii)] si $\dim_F W^0=2$, donc $W^0=E(f)$ où $E$ est une extension quadratique de $F$ et $f\in F^{\times} \textrm{ modulo } \nnn_{E/F}(E^{\times})$,  alors on a une suite exacte:
 $$1 \longrightarrow O(W) \longrightarrow GO(W) \stackrel{\lambda}{\longrightarrow} \nnn_{E/F}(E^{\times}) \longrightarrow 1. $$
 \end{itemize}
(3) Si $D=E$ est une extension quadratique de $F$, $\epsilon= \pm 1$, alors on a des suites exactes
$$1 \longrightarrow U(W) \longrightarrow GU(W) \stackrel{\lambda}{\longrightarrow} F^{\times} \longrightarrow 1 \textrm{ si } n \textrm{ pair },$$
et
$$1 \longrightarrow U(W) \longrightarrow GU(W) \stackrel{\lambda}{\longrightarrow} \nnn_{E/F}(E^{\times}) \longrightarrow 1  \textrm{ si } n \textrm{ impair }.$$
(4) Si $D$ est l'unique corps de quaternions sur $F$, $\epsilon=1$, alors on a une suite exacte
$$1 \longrightarrow U(W) \longrightarrow GU(W) \stackrel{\lambda}{\longrightarrow} F^{\times} \longrightarrow 1.$$
(5) Si $D$ est  l'unique corps de quaternions sur $F$, $\epsilon=-1$, alors on a une suite exacte
$$1 \longrightarrow U(W) \longrightarrow GU(W) \stackrel{\lambda}{\longrightarrow} F^{\times} \longrightarrow 1.$$
\end{prop}
\begin{proof}
Par le Lemme \ref{wittlambda},  en fait, le morphisme $\lambda$ est défini sur le groupe de Witt-Grothendieck.\\
(1) C'est bien connu.\\
(2) Par  [\cite{MVW}, Page 7], on connaît  les espaces quadratiques anisotropes à isomorphisme près:
(i) $F(a)$ pour $a \in F^{\times}$ modulo $F^{\times 2}$, de dimension 1, en ce cas $\lambda(GU(W))=F^{\times 2}$; (ii)  $E$, $E(f)$ pour chaque extension quadratique $E/F$, munie de la norme sur $F$, et $f\in F^{\times}$ n'est pas norme d'un élément de $E$, donc en ce cas $\lambda(GU(W))=\nnn_{E/F}(E^{\times})$; (iii) $D^0(a)$ ou $D$  où $D$ est l'unique corps de quaternions sur $F$, muni de la norme réduite, $D^0$ étant le sous-espace des éléments de trace nulle de $D$, et $a\in F^{\times}/{F^{\times 2}}$. Donc en ce cas, $\lambda(GU(W))=\nnn_{D/F}(D^{\times})=F^{\times}$.\\
(3) Dans ce cas, $\tau$ est  de seconde espèce, on a une bijection entre  les espaces $\epsilon$-hermitiens et les espaces  hermitiens. Par les énoncés de [\cite{MVW}, Page 7], les  espaces hermitiens anisotropes  sur $E$ sont classifiées de la façon suivante (i) $E$ , $E(f)$ où $f\in F^{\times}$ n'est pas norme d'un élément de $E$. (ii) $D$. Donc en le cas (i), $\lambda(GU(W))=\nnn_{E/F}(E^{\times})$; en le cas (ii) $\lambda(GU(W))=F^{\times}$, on a trouvé le résultat.\\
(4) Par les énoncés de [\cite{MVW}, Page 7], il existe un seul espace anisotrope sur $D$, et $\lambda(GU(D))=\nnn_{D/F}(D^{\times})=F^{\times}$. \\
(5) D'après le Théorème d'orthogonalisation, on a $W\simeq \oplus_{i=1}^n D(a_i)$. Pour chaque élément $a\in F^{\times}$,  le résultat dans [\cite{Scha}, Page 364] a montré qu'il existe un élément $d_a^i\in D^{\times}$ tel que $d_a^{i \star} a_i d_a^i=aa_i$. L'élément $\delta_a=\underbrace{d_a^1 \times \cdots \times d_a^n}_n$ appartient à $GU(W)$ et satisfait à $\lambda(\delta_a)=a$, donc $\lambda(GU(W))=F^{\times}$.
\end{proof}
Soient $W$ un espace symplectique sur $F$, $W=W_1\otimes_D W_2$ une décomposition en produit tensoriel, donnant $(U(W_1), U(W_2))$ une paire réductive duale irréductible de $Sp(W)$. On définit un sous-groupe intermédiaire $\Gamma$ de $Sp(W)$ de la façon suivante:
$$\Gamma= \{ (g_1,g_2)| g_1 \in GU(W_1), g_2 \in GU(W_2) \textrm{ tels que } \lambda(g_1g_2)=1\}.$$
On a l'application canonique
$$\Gamma \stackrel{i}{\longrightarrow} Sp(W); i(g_1, g_2) \longmapsto g_1 \otimes g_2.$$
On note $\overline{\Gamma}$ (resp.  $\widetilde{\Gamma}$) l'image réciproque de $i(\Gamma)$ dans $\overline{Sp}(W)$ (resp. $\widetilde{Sp}(W)$).
\begin{thm}\label{scindagedugroupeR}
Soient $(W, \langle,\rangle)$ un espace symplectique sur le corps non archimédien $F$  avec une décomposition en produit tensoriel $W=W_1\otimes_D W_2$, $\Gamma$ le sous-groupe intérmédiaire, $\overline{\Gamma}$ (resp. $\widetilde{\Gamma}$) leur image réciproque dans $\overline{Sp}(W)$(resp. $\widetilde{Sp}(W)$); alors $\overline{\Gamma}$, $\widetilde{\Gamma}$ sont scindés au-dessus de  $\Gamma$ sauf si $W\simeq W_1\otimes_{F'} W_2$ où $W_1$ est un espace symplectique sur une extension $F'$ de $F$,  et $\dim_{F'} W_2$ impaire.
\end{thm}
D'après  [\cite{MVW}, Page 52, Lemme], il suffit de considérer les paires duales de type 1 dans le Théorème \ref{MVWclassificationdepaires}. Nous posons
$$\Lambda_{\Gamma}=\{ \lambda(g_1)=\lambda(g_2)^{-1}| (g_1, g_2) \in \Gamma \},$$
donc il existe une suite exacte
\[
\begin{array}[c]{cccccccccc}
0 &\longrightarrow & U(W_1) \times U(W_2)&  \longrightarrow &   \Gamma      &\stackrel{\lambda}{\longrightarrow}  & \Lambda_{\Gamma}&  \longrightarrow & 1 \\
  &                &                     &                  &  (g_1,g_2) &   \longmapsto                     &  \lambda(g_1) &    &
\end{array}
\]
En utilisant la suite exacte d'Hochschild-Serre:
$$ \cdots  \longrightarrow H^2(\Lambda_{\Gamma}, \mu_8) \longrightarrow H^2(\Gamma, \mu_8) \longrightarrow H^2(U(W_1) \times U(W_2), \mu_8) \longrightarrow \cdots $$
D'après le Théorème \ref{scindagesurlepaireduale}, la restriction à $U(W_1)\times U(W_2)$  de la classe $[c_{Rao}]$ est triviale; ainsi  on peut  choisir un  2-cocycle $c$ de $[c_{Rao}]$ tel que
$$c(r,r')=c_{\Lambda_{\Gamma}}(\lambda(r), \lambda(r')) \cdots (1)$$
pour un $2$-cocycle $c_{\Lambda_{\Gamma}}$ dans $H^2(\Lambda_{\Gamma}, \mu_8)$. \\
\ \\
($\star$) Si on peut trouver un sous-groupe $\Gamma_1$ de $\Gamma$ tel que $\lambda(\Gamma_1)=\Lambda_{\Gamma}$ et que la restriction de $[c]$ à $\Gamma_1$ est triviale, alors $[c_{\Lambda_{\Gamma}}]$ est aussi triviale, donc $[c]$ l'est aussi.

\begin{lemme}\label{scindagehyperbolique}
Si $W_1$ ou $W_2$ est un espace hyperbolique,  alors les extensions $\overline{\Gamma}$, $\widetilde{\Gamma}$ sont scindées au-dessus de $\Gamma$.
\end{lemme}
\begin{proof}
On peut supposer que $W_2=nH$ où $H$ est le plan hyperbolique sur $D$, on a alors $W_{1_H}=W_1 \otimes_D H \hookrightarrow \oplus_1^n W_{1_H} \simeq W_1\otimes_D W_2; w \longmapsto (w)_1^n$. Soit $\Gamma_{W_{1_H}}$ le sous-groupe intermédiaire attaché à $W_{1_H}$. On a un morphisme de groupes
 $$\Gamma_{W_{1_H}} \longrightarrow \Gamma; \gamma \longmapsto (\gamma)_1^n.$$
 L'image est un sous-groupe de $\Gamma$ satisfaisant aux conditions du point ($\star$) ci-dessus, notée par $\Gamma_1$. Notant $\overline{\Gamma_{W_{1_H}}}$ (resp. $\overline{\Gamma_1}$) l'image réciproque de $\Gamma_{W_{1_H}}$ (resp. $\Gamma_1$) dans $\overline{Sp}(W_{1_H})$ (resp. $\overline{Sp}(W)$). D'après le Lemme \ref{morphismedegroupes}, on a un morphisme $$\overline{\Gamma_{W_{1_H}}} \longrightarrow \overline{\Gamma_1}.$$ Si $\overline{\Gamma_{W_{1_H}}}$ est scindé au-dessus de $\Gamma_{W_{1_H}}$, on trouve une section-morphisme $\Gamma_1 \longrightarrow \overline{\Gamma_1}$, d'où le résultat.  Ainsi, il suffit de supposer $W=W_1 \otimes_D H$. D'abord, supposons que $H$ est le plan hyperbolique hermitien à gauche sur $D$, muni de la base hyperbolique $\{ e_1, e_1^{\star}\}$; si $w_2=a_1e_1 + a_2 e_1^{\star}\in H$, $w_2'=b_1e_1 + b_2 e_1^{\star} \in H$, on a
$$ \langle w_1, w_1' \rangle_2= a_1 \tau(b_2) + a_2 \tau(b_1)= (a_1, a_2)
 \begin{pmatrix}
0& 1\\
 1&0
\end{pmatrix}  \tau\Bigg(\left( \begin{array}{l}
b_1 \\
b_2
\end{array}\right)\Bigg).$$
Notant que $g= \begin{pmatrix}
a& b\\
c&d
\end{pmatrix} \in GL_2(D)$ agit sur $w_2=xe_1 + y e_1^{\star} \in H $ par $g\cdot w_2= (x, y)  \begin{pmatrix}
a& b\\
c&d
\end{pmatrix}  \left(\begin{array}{l}
e_1 \\
e_1^{\star}
\end{array}\right)$,  alors $g\in GU( H)$ si et seulement si
$$\begin{pmatrix}
a & b\\
c & d
\end{pmatrix} \begin{pmatrix}
0& 1\\
1&0
\end{pmatrix} \begin{pmatrix}
\tau(a)& \tau(c)\\
\tau(b)& \tau(d)
\end{pmatrix}= \lambda(g)\begin{pmatrix}
0& 1\\
1&0
\end{pmatrix}.$$
Il en résulte que $\lambda(g) \in F^{\times}$ et $\begin{pmatrix}
1& 0\\
0&\lambda(g)
\end{pmatrix} \in GU(H)$. Donc on a  la suite exacte
$$1 \longrightarrow U(H) \longrightarrow GU(H) \stackrel{\lambda}{\longrightarrow} F^{\times} \longrightarrow 1.$$
On a m\^eme une section, i.e. un homomorphisme $s: F^{\times} \longrightarrow GU(H); a \longmapsto   \begin{pmatrix}
1& 0\\
0&a
\end{pmatrix}$ qui satisfait à $\lambda\circ s= id_{F}$. On voit aisément que le m\^eme argument   est aussi vrai si $H$ est le plan hyperbolique anti-hermitien sur $D$. Ainsi, on peut définir un sous-groupe $\Gamma_1$ de $\Gamma$ où
$$\Gamma_1=\{ (g_1,g_2) \in GU(W_1) \times s(F^{\times}) | \lambda(g_1)=\lambda(g_2)^{-1} \}.$$
L'image de $\Gamma_1$ dans le groupe $Sp(W)$ est liée à  un sous-groupe de Borel. Cela implique que la restriction de $[c_{Rao}]$ à $\Gamma_1$ est triviale. Par le point $(\star)$ ci-dessus, on obtient le résultat.
\end{proof}
\begin{lemme}\label{scindagesursymplectique}
Si $W_1$ est un espace symplectique de dimension $2n$ et que $W_2$ est un espace orthogonal de dimension $2m$ sur $F$, alors les groupes $\overline{\Gamma}, \widetilde{\Gamma}$ sont scindés au-dessus de  $\Gamma$.
\end{lemme}
\begin{proof}
C'est une conséquence du Lemme \ref{scindagehyperbolique}.
\end{proof}
\begin{lemme}\label{scindagequadratique}
Soient $D=E$ une extension quadratique de $F$, $W_1$( resp. $W_2$) un espace hermitien (resp. anti-hermitien) anisotrope sur $E$; alors les groupes $\overline{\Gamma}$, $\widetilde{\Gamma}$ sont scindés au-dessus de  $\Gamma$.
\end{lemme}
\begin{proof}
(i) Si $\dim_E(W_1)=1$, $W_1=E(f)$ où $f=1$ ou $f\in F^{\times} \smallsetminus \nnn_{E/F}(E^{\times})$. La forme associée est définie par $\langle e_1, e_2 \rangle_1=f\overline{e_1} e_2$. On a une application bijective $E$-linéaire $W= E(f) \otimes_E W_2 \stackrel{b}{\longrightarrow} W_2; e\otimes w_2 \longmapsto ew_2$. Si $e_1 \otimes w_2, e_1'\otimes w_2' \in W$, on a
$$\langle e_1\otimes w_2, e_1'\otimes w_2' \rangle_W= \tr_{E/F}( \langle e_1, e_1'\rangle_1 \overline{\langle w_2, w_2' \rangle_2}) = \tr_{E/F}(f\overline{e_1} e_1' \overline{\langle w_2, w_2' \rangle_2})$$
$$=\tr_{E/F} (f\overline{e_1 \langle w_2, w_2' \rangle_2 \overline{e_1}'}) \stackrel{ \textrm{ par définition} }{=} \tr_{E/F}(f \overline{\langle e_1w_2, e_1' w_2' \rangle_2}).$$
Donc si on munit l'espace $W_2$ de la forme symplectique  $\langle, \rangle_2'=\tr_{E/F}(f\overline{\langle, \rangle_2})$, alors l'application $b$ ci-dessus est une isométrie. Rappelons:
$$\Gamma=\{ (g_1,g_2) | g_1 \in E^{\times}, g_2 \in GU(W_2), \textrm{ tels que } \lambda(g_1)\lambda(g_2)=1\}$$
et $$\Gamma \hookrightarrow Sp(W, \langle, \rangle) \simeq Sp(W_2', \langle, \rangle_2').$$
On note par $i$ le morphisme composé. Ainsi, on peut voir aisément que $i(\Gamma)= i(U(W_1) \times U(W_2))$, donc $\overline{\Gamma}$, $\widetilde{\Gamma}$ sont scindés au-dessus de  $\Gamma$ en ce cas.\\
(ii) Si $\dim_E(W_2)=1$, $W_2 \simeq E(f)$ où $f=1$ ou $f\in F^{\times} \smallsetminus \nnn_{E/F}(E^{\times})$. On peut prendre un tel élément $u\in E^{\times}$ satisfaisant à $\overline{u}/u=-1$, tel que la forme anti-hermitienne soit donnée par la formule suivante:
$$\langle e_2, e_2' \rangle_2= u f e_2 \overline{e_2'} \textrm{ pour } e_2, e_2' \in E.$$
On a $$W =W_1 \otimes_E E(f) \simeq W_1.$$
Rappelons la définition de la forme $\langle, \rangle_W$:
$$\langle w_1 \otimes e_2, w_1'\otimes e_2' \rangle_W = \tr_{E/F} \big( \langle w_1, w_1' \rangle_1 (-uf \overline{e_2} e_2')\big)= \tr_{E/F}(-uf\langle e_2 w_1, e_2' w_1' \rangle_1 ).$$
Nous pouvons identifier $W$ à $W_1$ par  munir $W_1$ de la forme $\tr_{E/F}(-uf\langle, \rangle_1)$. Il en résulte que
$$i(\Gamma)=i(U(W_1) \times U(W_2))$$
où $i: \Gamma \longrightarrow Sp(W_1, \tr_{E/F}(-uf\langle, \rangle_1))$, donc $\overline{\Gamma}, \widetilde{\Gamma}$ sont scindés au-dessus de  $\Gamma$.\\
(iii)  Si $\dim_{E} W_1 = \dim_{E} W_2 =2$, donc $W_1 \simeq D$ et $W_2 \simeq D$ où $D$ est le corps des quaternions sur $F$. Soit $E=F(i)$ avec $i^2=-\alpha$, choisissons un élément $j$ de $D$ tels que $ j^2=-\beta, ij=-ji= k$, de sorte que  $\{ 1, i, j, k\}$ forme une base de $D$. On note $-$ la conjugaison canonique de $D$. Par hypothèse, $W_1 \simeq D= E + jE$ muni de la forme
$$\langle e_1+je_2, e_1' + je_2' \rangle_1=\tr_{D/E}\Big( \overline{(e_1+ j e_2)}(e_1' + je_2')\Big)= \overline{e_1} e_1' + \beta \overline{e_2} e_2'$$
où $\tr_{D/E}(e_1+je_2):=e_1$. D'autre part, $W_2 \simeq D = E + Ej$, et il existe   un élément $u$ de $E$ satisfait à $u^{\sigma}/u=-1$  où $\Gal(E/F) = \langle \sigma \rangle$,  tels que la forme anti-hermitienne sur  $D$ soit définie  de la façon suivante:
$$\langle e_1+e_2j, e_1' + e_2' j\rangle_2= \tr_{D/E}\Big( u (e_1+e_2 j) (\overline{e_1'} + \overline{e_2'j})\Big)= u(e_1\overline{e_1'}+ \beta e_2 \overline{e_2'}).$$
Rappelons que la $F$-forme associée à $W=W_1 \otimes W_2$ soit définie de la façon suivante:
$$\langle w_1\otimes w_2, w_1' \otimes w_2' \rangle= \tr_{E/F}\Big( \langle w_1, w_1'\rangle \overline{ \langle w_2, w_2' \rangle_2} \Big)$$
$$=\tr_{E/F}\Big( - u \big(\overline{a_1}a_1' + \beta \overline{a_2}a_2' \big)\big( b_2 \overline{b_1} + \beta b_2'\overline{b_1'}\big)\Big)$$
$$=\tr_{E/F}\Big( -u\big( \overline{a_1}\overline{b_1} a_1'b_2 + \beta\overline{a_1}\overline{b_1'}a_1'b_2'\big)\Big)+ 
\tr_{E/F}\Big( -u\beta \big( \overline{a_2}\overline{b_1} a_2' b_2 + \beta \overline{a_2} \overline{b_1'} a_2' b_2'\big)\Big)$$
où $w_1=a_1+ j a_2, w_1'=a_1' + ja_2' \in W_1; w_2=b_1 + b_1' j, w_2'= b_2 + b_2' j\in W_2$. \\
On sait que $W= W_1\otimes_E W_2 \simeq D\otimes_E D \simeq (E + jE)\otimes_E D \simeq D \oplus D$ comme $F$-espace. Pour montrer que ceci sont aussi isométriques, il suffit de munir l'espace $D \oplus D$ de la forme symplectique ci-dessous:
$$\langle, \rangle_{D \oplus D}=\tr_{E/F}\Big( \overline{\langle, \rangle_2}\Big) + \tr_{E/F} \Big( \beta \overline{\langle, \rangle_2}\Big).$$
Soit $\Gamma_1$ le sous-groupe de $D^{\times}$ engendré par $E^{\times}$ et $\langle j\rangle$. Par définition, on peut identifier $\Gamma_1$ à un  sous-groupe de $\Gamma$. De plus $\lambda(\Gamma_1)=\langle \lambda (E^{\times}), \lambda(j)\rangle= \langle \nnn_{E/F}(E^{\times}), \beta \rangle =F^{\times}$.  Rappelons l'application $ \Gamma_1  \hookrightarrow \Gamma \stackrel{i}{ \longrightarrow} Sp(D\oplus D)$. Si $e\in E^{\times}$, $i\big((e, e^{-1})\big)=1$; Si on prend $j \in \Gamma_1$, on peut voir aisément que $i\big((j, j^{-1})\big)$ est dans un sous-groupe parabolique minimal de $Sp(D\oplus D)$. Comme $\Gamma_1$ est engendré par $E^{\times}$ et $f$, de plus  $i$ est un homomorphisme de groupes, il en résulte que $i(\Gamma_1)$ est dans un sous-groupe parabolique minimal de $Sp(D\oplus D)$. Donc la restriction de $[c_{Rao}]$ à $\Gamma_1$ est triviale;  utilisant le point $(\star)$, nous obtenons le résultat.
\end{proof}

\begin{lemme}\label{scindagequaternions}
Si $D$ est l'unique corps de quaternions sur $F$ et $W_1 = D$, alors les extensions $\overline{\Gamma}, \widetilde{\Gamma}$ sont scindées sur $\Gamma$.
\end{lemme}
\begin{proof}
Rappelons que  la forme hermitienne sur $D$  est donnée par $\langle d, d' \rangle_1= \tau(d) d'$  où  $\tau$ est  l'involution canonique  de $D$. On note
$$ \Gamma_1=\{ (g,g^{-1})| g\in D^{\times}\}.$$
$\Gamma_1$ est un sous-groupe de $\Gamma$ satisfaisant au point $(\star)$. D'après le Théorème d'orthogonalisation, on a $W_2 \simeq \oplus_{i=1}^n (a_i)D$. D'après le Lemme \ref{morphismedegroupes}, on se ramère au cas $\dim_D W_2=1$. Dans ce cas,  l'espace $W$ est isométrique à $W_2$  en munissant $W_2$  de la forme $\langle, \rangle_2'=\tr_{D/F}(\tau(\langle, \rangle_2))$.  L'action de $\Gamma_1$ sur $W_2$ est donnée par la formule suivante:
$$d \longmapsto g dg^{-1} \textrm{ pour } (g, g^{-1}) \in \Gamma_1, d\in W_2 \simeq D.$$
On prend une base $\{ 1, i, j , k=ij\}$ de $D$. Par la théorie des nombres, $F^{\times}$ est engendré par $F^{\times 2}, \nnn_{D/F}(i), \nnn_{D/F}(j)$. Notant $t: \Gamma \longrightarrow Sp(W_2, \langle, \rangle_2')$. On pose
$$\Gamma_{1,0}=\{ \gamma=(g, g^{-1}) \in \Gamma_1 | \nnn_{D/F}(g) \in F^{\times 2}\}.$$
L'image de $\Gamma_{1,0}$ dans $Sp(W_2, \langle, \rangle_2')$ est dans celle de $U(W_1) \times U(W_2)$. Soit $\Gamma_{1,1}$ le groupe engendré par $(i, i^{-1})$ et $(j, j^{-1})$. L'image de $\Gamma_{1,1}$ est dans un sous-groupe de Borel de $Sp(W_2, \langle, \rangle_2')$. Considérons le groupe intermédiaire $\Gamma_1'$ engendré par $\Gamma_{1,0}$, $i$, $j$. On a une suite exacte
$$1 \longrightarrow \Gamma_{1,0} \longrightarrow \Gamma_1' \longrightarrow \Gamma_1'/{\Gamma_{1,0}} \simeq F^{\times }/{F^{\times 2}} \longrightarrow 1.$$
Par le m\^eme argument du point $(\star)$, on voit que $\overline{\Gamma_1'}$ est scindé au-dessus de  $\Gamma_1'$, d'où le résultat.
\end{proof}
\underline{La preuve du Théorème \ref{scindagedugroupeR}}\\
D'abord, on a les décompositions de Witt, $W_1 = m_1H \oplus W_1^0$ et $W_2=m_2 H \oplus W_2^0$ où $W_1^0$ et $W_2^0$ sont anisotropes. Si $W_1^0=0$ ou $W_2^0=0$, on  déduit le résultat du Lemme
\ref{scindagehyperbolique}.  Si $W_1^0 \neq 0$ et $W_2^0 \neq 0$, on a une application  $W^0 = W_1^0 \otimes_D W_2^0 \hookrightarrow  W= W_1^0 \otimes_D W_2^0 \oplus \big( W_1^0 \otimes_D m_2H \oplus m_1 H \otimes_D W_2^0 \oplus m_1H \otimes m_2 H\big)$. On peut définir un sous-groupe
$$\Gamma_1=\{ (g_1,g_2) | g_1 \in GU(W_1^0), g_2 \in GU(W_2^0) \textrm{ tels que }  \lambda(g_1)\lambda(g_2)=1\}.$$
D'après la démonstration du Lemme \ref{wittlambda}, on sait que $\Gamma$ est un sous-groupe de $U(W_1) \times U(W_2)\Gamma_1$. Il suffit de considèrer le groupe intermédiaire $U(W_1) \times U(W_2) \cdot \Gamma_1$. Par le m\^eme argument du point $(\star)$, on est réduit à traiter le groupe $\Gamma_1$.  D'après le Lemme \ref{morphismedegroupes},   on est donc  ramené au cas $W=W_1^0 \otimes_D W_2^0$, et par les lemmes \ref{scindagesursymplectique} à \ref{scindagequaternions}, on trouve le résultat dans ces cas. Ceci termine la démonstration.
\begin{rem}
Dans le Théorème \ref{scindagedugroupeR},  nous excluons le cas:\\
$W=W_1\otimes_{F'} W_2$ où $W_1$ (resp.  $W_2$ ) est un espace symplectique (resp.  orthogonal) sur une extension $F'$ de $F$, de dimension $2n$ (resp.  $2m-1$) avec $n \geq 1, m\geq 1$.\\
Dans ce cas, on a vu que l'image réciproque de $Sp(W_1)$ dans $\overline{Sp}(W)$ (resp.  $\widetilde{Sp}(W)$) est égale à $\overline{Sp}(W)$ (resp.   $\widetilde{Sp}(W)$); a fortiori, le groupe intermédiaire $\overline{\Gamma}$(resp.   $\widetilde{\Gamma}$) n'est pas scindé au-dessus de  $\Gamma$.
\end{rem}
\begin{prop}\label{saufcasscindage}
Soit $W=W_1 \otimes_{F'} W_2$ un espace symplectique sur $F'$, où $W_1$ (resp.  $W_2$) est un espace symplectique (resp.  orthogonal) de dimension $2n$ (resp.  $2m-1$) avec $n \geq 1$ et $m\geq 1$. On note $\widetilde{GSp}(W_1)$  chaque extension du groupe  de similitudes par $\C^{\times}$ tel que $\widetilde{GSp}(W_1)$ contient $\widetilde{Sp}(W_1)$ comme un sous-groupe.\\
On définit
$$\widetilde{\Gamma}^{1/2}= \{ (\widetilde{g}, h)\in \widetilde{GSp}(W_1) \times GO(W_2) | \widetilde{\lambda}(\widetilde{g}) \lambda(h)=1\}$$
où $\widetilde{\lambda}$ est le morphisme $\widetilde{GSp}(W_1) \longrightarrow F^{'\times}$. Alors on a un morphisme de groupes $\widetilde{\Gamma}^{\frac{1}{2}} \stackrel{i_{\frac{1}{2}}}{\longrightarrow} \widetilde{Sp}(W); \big( (g, \epsilon), h\big) \longmapsto (g \otimes h, \epsilon)$ pour $g\in GSp(W_1), \epsilon \in \C^{\times}, h\in GO(W_2)$,  et $\lambda(g)\lambda(h)=1$.
\end{prop}
\begin{proof}
Soient $\{ e_1, \cdots, e_n; e_1^{\star}, \cdots, e_n^{\star}\}$ une base symplectique de $W_1$, $X$ (resp.  $X^{\star}$) le lagrangien engendré par les $e_i$ (resp.  $e_i^{\star}$), $\{ f_1, \cdots, f_{2m-1}\}$ une base orthogonale de $W_2$ sur $F'$. On sait que $W=X\otimes W_2 \oplus X^{\star} \otimes W_2$ (resp.  $W_1=X\oplus X^{\star}$) est une polarisation complète. Sous les polarisations ci-dessus, on sait que le groupe $\widetilde{GSp}(W_1)$ (resp.  $\widetilde{GSp}(W)$) est engendré par $\widetilde{Sp}(W_1)$ (resp.  $\widetilde{Sp}(W)$) avec $F^{'\times}$.  On désigne par  $c_{W}$ le $2$-cocycle associé au groupe $\widetilde{Sp}(W)$ lié à $\psi$ et  $X^{\star} \otimes W_2$ [cf. \cite{MVW}, chapitre 1 ou \cite{Kud2}].\\
(i) $\widetilde{Sp}(W_1) \times O(W_2)$ est un sous-groupe de $\widetilde{\Gamma}^{\frac{1}{2}}$. D'après le Théorème \ref{scindagesurlepaireduale}, on prend le $2$-cocycle $c_{W_1}$ de $H^2(Sp(W_1), \C^{\times})$, tel que
$$\widetilde{Sp}(W_1) \stackrel{s_1}{\longrightarrow} \widetilde{Sp}(W);$$
$$ \widetilde{g}=(g, \epsilon) \longmapsto (g \otimes 1, \epsilon),$$
définit un homomorphisme de groupes, i.e. $c_{W_1}(g_1,g_2)=c_W(g_1\otimes 1, g_2\otimes 1)$ pour $g_1,g_2 \in Sp(W_1)$.
 Comme $O(W_2)$ est un sous-groupe du groupe parabolique $P(X^{\star} \otimes W_2)$, on a un morphisme
$$O(W_2) \stackrel{s_2}{\longrightarrow} \widetilde{Sp}(W);$$
$$h \longmapsto (1 \otimes h, 1).$$
Il s'ensuit que
$$s_1((g,\epsilon)) s_2(h)= (g\otimes 1, \epsilon) \cdot (1 \otimes h, 1)= (g\otimes h, c_{W}(g\otimes 1, 1 \otimes h) \epsilon)= (g\otimes h, \epsilon)=i_{\frac{1}{2}}(\widetilde{g}, h).$$
 Par le résultat de Waldspurger, on sait que $s_1(\widetilde{g})$ commute avec $s_2(h)$, donc  $i_{\frac{1}{2}}|_{\widetilde{Sp}(W_1) \times O(W_2)}$ est un homomorphisme de groupes.\\
De plus, soient $\widetilde{g_1}=(g_1, \epsilon_1), \widetilde{g_2}=(g_2, \epsilon_2) \in \widetilde{Sp}(W_1)$ et $h_1, h_2 \in O(W_2)$. On a
$$i_{\frac{1}{2}} \Big( (\widetilde{g_1} \widetilde{g_2}, h_1h_2)\Big)= i_{\frac{1}{2}} \Big( (\widetilde{g_1}, h_1)\cdot (\widetilde{g_2}, h_2)\Big)=
(g_1 \otimes h_1, \epsilon_1) \cdot (g_2 \otimes h_2, \epsilon_2).$$
Il en résulte que
$$c_{W_1}(g_1, g_2)=c_W(g_1 \otimes h_1, g_2 \otimes h_2).$$
(ii) On définit
$\Lambda(\widetilde{\Gamma}^{\frac{1}{2}})=\{ (\widetilde{g_t}, h) \in \widetilde{\Gamma}^{\frac{1}{2}} | \widetilde{g_t}=(g_t, \epsilon) \textrm{ avec } g_t = \begin{pmatrix}
1& 0\\
0&t
\end{pmatrix}, \textrm{ pour } t \in F^{\times}, \epsilon \in \C^{\times}, h\in GO(W_2) \}$.\\
Si on prend deux élements $(\widetilde{g_{t_i}}, h_i)= \big( (g_{t_i}, \epsilon_i), h_i\big)$ de $\Lambda( \widetilde{\Gamma}^{\frac{1}{2}})$ pour $i=1, 2$, alors
$$i_{\frac{1}{2}}\big( (\widetilde{g_{t_i}}, h_i)\big)= (g_{t_i}\otimes h_i, \epsilon_i)= \bigg( \begin{pmatrix}
h_i& 0\\
0&t_ih_i
\end{pmatrix}, \epsilon_i\bigg).$$
$$c_{W}(g_{t_1}\otimes h_1, g_{t_2}\otimes h_2)= c_{W} \bigg( \begin{pmatrix}
h_1& 0\\
0&t_1h_1
\end{pmatrix}, \begin{pmatrix}
h_2& 0\\
0&t_2h_2
\end{pmatrix} \bigg)=1.$$
Donc
$$i_{\frac{1}{2}}\big( (\widetilde{g_{t_1}}, h_1)\big) i_{\frac{1}{2}}\big( (\widetilde{g_{t_2}}, h_2) \big)= (g_{t_1t_2} \otimes h_1h_2, \epsilon_1\epsilon_2).$$
De plus
$$( \widetilde{g_{t_1}}, h_1) (\widetilde{g_{t_2}}, h_2) = (\widetilde{g_{t_1}} \widetilde{g_{t_2}}, h_1h_2) \stackrel{c_{W_1}(g_{t_1}, g_{t_2})=c_{W}(g_{t_1}\otimes 1, g_{t_2}\otimes 1)=1}{=} \bigg( (g_{t_1t_2} , \epsilon_1\epsilon_2), h_1h_2\bigg).$$
$$i_{\frac{1}{2}}\bigg( \big( ( g_{t_1t_2}, \epsilon_1 \epsilon_2), h_1h_2\big) \bigg) = \big( g_{t_1t_2} \otimes h_1 h_2, \epsilon_1 \epsilon_2\big).$$
Il en résulte que
$$i_{\frac{1}{2}}\bigg( (\widetilde{g_{t_1}}, h_1) (\widetilde{g_{t_2}}, h_2) \bigg)= i_{\frac{1}{2}}\bigg( (\widetilde{g_{t_1}}, h_1)\bigg) i_{\frac{1}{2}}\bigg( (\widetilde{g_{t_2}}, h_2)\bigg),$$
i.e. $i_{\frac{1}{2}}|_{\Lambda(\widetilde{\Gamma}^{\frac{1}{2}})}$ est un homomorphisme de groupes.\\
(iii) Soient $(\widetilde{g}, h)= \Big( (g, \epsilon), h \Big) \in \widetilde{\Gamma}^{\frac{1}{2}}$ et $(\widetilde{g}, h)= (\widetilde{g_0}, h_0) \cdot (\widetilde{g_t}, h_t)$ avec $(\widetilde{g_0}, h_0)=\Big( (g_0,\epsilon), h_0\Big) \in \widetilde{Sp}(W_1) \times O(W_2)$, $(\widetilde{g_t}, h_t)=\Big( (g_t, 1), h_t \Big) \in \Lambda(\widetilde{\Gamma}^{\frac{1}{2}})$; alors
$$i_{\frac{1}{2}}\big( (\widetilde{g}, h)\big)= (g\otimes h, \epsilon)= (g_0 \otimes h_0, \epsilon) (g_t\otimes h_t, 1)= i_{\frac{1}{2}}\Big( (\widetilde{g_0}, h_0) \Big) i_{\frac{1}{2}}\Big((\widetilde{g_t}, h_t)\Big).$$
(iv) En général, soient $(\widetilde{g}^{(i)}, h^{(i)})=\bigg( (g^{(i)}, \epsilon^{(i)}), h^{(i)}\bigg) \in \widetilde{\Gamma}^{\frac{1}{2}}$ pour $i=1, 2$, où
$(\widetilde{g}^{(i)}, h^{(i)})= (\widetilde{g_0}^{(i)}, h_0^{(i)}) (\widetilde{g_t}^{(i)}, h_t^{(i)})$ avec $(\widetilde{g_0}^{(i)}, h_0^{(i)})= \Big( (g_0^{(i)}, \epsilon^{(i)}), h_0^{(i)}\Big) \in \widetilde{Sp}(W_1) \times O(W_2)$ et $ (\widetilde{g_t}^{(i)}, h_t^{(i)})=\Big( (g_t^{(i)}, 1), h_t^{(i)} \Big) \in \Lambda(\widetilde{\Gamma}^{\frac{1}{2}})$; alors
$$(\widetilde{g}^{(1)}, h^{(1)}) (\widetilde{g}^{(2)}, h^{(2)})= \bigg( (g_0^{(1)}, \epsilon^{(1)}), h_0^{(1)}\bigg) \bigg( (g_t^{(1)}, 1), h_t^{(1)}\bigg) \bigg((g_0^{(2)}, \epsilon^{(2)}), h_0^{(2)}\bigg) \bigg( (g_t^{(2)}, 1), h_t^{(2)}\bigg)$$
$$\stackrel{\textrm{ La restriction du cocycle $c_W$ au groupe parabolique minimal est triviale}}{=} \bigg( (g_0^{(1)}, \epsilon^{(1)}), h_0^{(1)}\bigg) \bigg( \big( g_t^{(1)} g_0^{(2)} (g_t^{(1)})^{-1}, \epsilon^{(2)}\big), h_t^{(1)}h_0^{(2)} (h_t^{(1)})^{-1}\bigg)$$ $$ \bigg( \big(g_t^{(1)}, 1\big), h_t^{(1)}\bigg) \bigg( \big( g_t^{(2)},1\big), h_t^{(2)}\bigg)$$
$$= \bigg( \Big( g_0^{(1)} g_t^{(1)} g_0^{(2)} (g_t^{(1)})^{-1}, c_{W_1} \big( g_0^{(1)}, g_t^{(1)} g_0^{(2)} (g_t^{(1)})^{-1}\big) \epsilon_1^{(1)} \epsilon^{(2)}\Big), h_0^{(1)} h_t^{(1)} h_0^{(2)} (h_t^{(1)})^{-1}\bigg) \bigg(\big( g_t^{(1)}g_t^{(2)}, 1\big), h_t^{(1)} h_t^{(2)}\bigg).$$
Donc
$$i_{\frac{1}{2}}\bigg( (\widetilde{g}^{(i)}, h^{(i)})\bigg)= \bigg( g_0^{(i)}\otimes h_0^{(i)}, \epsilon^{(i)}\bigg) \bigg( g_t^{(i)}\otimes h_t^{(i)}, 1\bigg).$$
$$i_{\frac{1}{2}}\bigg( (\widetilde{g}^{(1)}, h^{(1)})\bigg) i_{\frac{1}{2}}\bigg( (\widetilde{g}^{(2)}, h^{(2)})\bigg) = \bigg( g_0^{(1)} \otimes h_0^{(1)}, \epsilon^{(1)}\bigg) \bigg( g_t^{(1)}\otimes h_t^{(1)}, 1 \bigg) \bigg( g_0^{(2)} \otimes h_0^{(2)}, \epsilon^{(2)}\bigg) \bigg( g_t^{(2)} \otimes h_t^{(2)}, 1 \bigg),$$
et
$$i_{\frac{1}{2}}\bigg( \Big(\widetilde{g}^{(1)}, h^{(1)}\Big) \Big( \widetilde{g}^{(2)}, h^{(2)}\Big)\bigg) = \bigg( g_0^{(1)} g_t^{(1)} g_0^{(2)} (g_t^{(1)})^{-1} \otimes h_0^{(1)} h_t^{(1)} h_0^{(2)} (h_t^{(1)})^{-1}, \epsilon^{(1)} \epsilon^{(2)} c_{W_1} (g_0^{(1)}, g_t^{(1)} g_0^{(2)} (g_t^{(1)})^{-1})\bigg) $$ $$ \bigg( g_t^{(1)} g_t^{(2)} \otimes h_t^{(1)} h_t^{(2)}, 1\bigg)$$
$$\stackrel{c_{W_1}\big(g_0^{(1)}, g_t^{(1)} g_0^{(2)} (g_t^{(1)})^{-1}\big)=c_{W} (g_0^{(1)} \otimes h_0^{(1)}, g_t^{(1)} g_0^{(2)} (g_t^{(1)})^{-1} \otimes h_t^{(1)} \otimes h_0^{(2)} (h_t^{(1)})^{-1})}{=} \bigg( g_0^{(1)} \otimes h_0^{(1)}, \epsilon^{(1)}\bigg)$$ $$ \bigg( g_t^{(1)} g_0^{(2)} (g_t^{(1)})^{-1} \otimes h_t^{(1)} h_0^{(2)} (h_t^{(1)})^{-1}, \epsilon^{(2)}\bigg) \bigg( g_t^{(1)} g_t^{(2)} \otimes h_t^{(1)} h_t^{(2)}, 1\bigg)$$
$$\stackrel{c_W|_{P\times G}=1}{=} \bigg( g_0^{(1)} \otimes h_0^{(1)}, \epsilon^{(1)}\bigg) \bigg(g_t^{(1)} \otimes h_t^{(1)}, 1\bigg) \bigg( g_0^{(2)} \otimes h_0^{(2)}, \epsilon^{(2)}\bigg) \bigg( g_t^{(2)} \otimes h_2^{(2)}, 1 \bigg)= i_{\frac{1}{2}}\bigg((\widetilde{g}^{(1)}, h^{(1)})\bigg)i_{\frac{1}{2}}\bigg( (\widetilde{g}^{(2)}, h^{(2)})\bigg).$$
Ceci termine  la démonstration!
\end{proof}
\section{Représentations de bigraphe forte pour les groupes des similitudes I}\label{bigraphefortedesimilitudes1}
Dans cette sous-section, nous expliquont  comment on obtient une  représentation de bigraphe forte pour les groupes des similitudes.

Supposons que $F$ est un corps local non archimédien de caractéristique  résiduelle impaire $p$. Fixons un caractère non trivial $\psi: F\longrightarrow \C$. Soient  $W$  un espace symplectique de dimension $2n$ sur $F$, $\omega_{\psi}$ la représentation de Weil du groupe $\overline{Sp}(W)$ . Si $H$ est un sous-groupe réductif de $Sp(W)$, nous notons $\overline{H}$  son image réciproque dans $\overline{Sp}(W)$ .

 On  désignera par  $A$ un groupe abélien fini d'ordre $n$ avec $2|n$, et $(p,n)=1$.

Soient $D$ un corps de centre $F'$, $F'$ une extension finie de $F$,  $W=W_1\otimes_D W_2$ une décomposition en produit tensoriel. Ainsi $(U(W_1), U(W_2))$ est une paire réductive duale irréductible de $Sp(W)$; nous notons  $\Gamma$ le groupe intermédiaire associé. On a une suite exacte
$$1 \longrightarrow U(W_1) \times U(W_2) \longrightarrow \Gamma \stackrel{\lambda}{ \longrightarrow} \Lambda_{\Gamma} \longrightarrow 1$$
où $\Lambda_{\Gamma}=\{ \lambda(g_1)=\lambda(g_2)^{-1} | (g_1, g_2) \in \Gamma\}$ est un sous-groupe de $F^{'\times }$. On a aussi les suites exactes
$$ 1 \longrightarrow U(W_i) \longrightarrow GU(W_i) \stackrel{\lambda}{\longrightarrow} \Lambda_{GU(W_i)} \longrightarrow 1,$$
où $ \Lambda_{GU(W_i)}$ est un sous-groupe de $F^{' \times}$ contenant $\Lambda_{\Gamma}$.On notera
$$G^{\Gamma}U(W_i)= \textrm{ l'image réciproque de } \Lambda_{\Gamma} \textrm{ dans } GU(W_i). $$

\ \\
(i) Tout d'abord, supposons que $W= W_1\otimes_D W_2$ n'est pas le cas dans la Proposition \ref{saufcasscindage}. D'après le Théorème \ref{scindagedugroupeR}
, l'extension $\overline{\Gamma}$ est scindée au-dessus de $\Gamma$, on peut obtenir un homomorphisme qui \underline{n'est pas unique} en général
$$\Gamma \stackrel{i_{\Gamma}}{\longrightarrow} \overline{Sp}(W).$$
Donc on obtient une représentation lisse du groupe $\Gamma$ par  restriction de la représentation de Weil à $\Gamma$, on note  $\rho_{\psi}$ cette représentation de $\Gamma$.
\ \\
Pour les groupes, on a les relations suivantes
$$G^{\Gamma}U(W_i) / U(W_i) \simeq \Lambda_{\Gamma}$$
et
$$ \Gamma/{U(W_1) \times U(W_2)} \simeq \Lambda_{\Gamma}.$$
D'après la Proposition \ref{lamedavalues}, $\Lambda_{\Gamma} = F^{' \times}$ ou $F^{' 2}$ ou bien $\nnn_{E'/F'}(E^{'\times})$ avec $E' /F'$ une extension quadratique, il en résulte que $F^{' \times}/{\Lambda_{\Gamma}}$ est un groupe abélien de cardinal fini.\\
De plus, par défintion, $G^{\Gamma}U(W_i) \supset F^{'\times } U(W_i)$ et $G^{\Gamma}U(W_i)/{ F^{'\times } U(W_i)}$ est un groupe abélien de cardinal  fini.
\begin{thm}\label{scindéecas1}
Si on définit $$\pi=c-\Ind_{\Gamma}^{G^{\Gamma}U(W_1) \times G^{\Gamma}U(W_2)} \rho_{\psi},$$
alors $\pi$ est une représentation de bigraphe forte.
\end{thm}
\begin{proof}
Ceci résulte du Théorème \ref{abelienbigraphedefini}.
\end{proof}
\  \\
(ii) Soit $W \simeq W_1 \otimes_{F'} W_2$ un espace symplectique sur $F'$ où $W_1$( resp.  $W_2$) est un espace symplectique (resp.  orthogonal) de dimension $2n$ ( resp.  $2m-1$) avec $n\geq 1$ et $m\geq 1$. On a une suite exacte
$$1  \longrightarrow Sp(W_1) \longrightarrow GSp(W_1) \longrightarrow F^{'\times} \longrightarrow 1.$$
On déduit de la suite exacte de Hochschild-Serre
$$\cdots \longrightarrow H^2(F^{'\times}, A) \longrightarrow H^2( GSp(W_1), A) \stackrel{d}{\longrightarrow} H^2(Sp(W_1) , A)  \longrightarrow \cdots $$
Prenons un élément $[c]$ de $H^2(GSp(W_1), A)$ tel que $d([c])=[c_{Rao}]$. La classe $[c]$ détermine un groupe $\widetilde{GSp}^A(W_1)$. On a une suite exacte
$$1 \longrightarrow \widetilde{Sp}^A(W_1) \longrightarrow \widetilde{GSp}^A(W_1)  \stackrel{\widetilde{\lambda}}{\longrightarrow} \Lambda_{\widetilde{GSp}^A(W_1)}=F^{'\times} \longrightarrow 1.$$
Nous pouvons définir un sous-groupe intermédiaire
$$\widetilde{\Gamma}^{A \frac{1}{2}}= \{ (\widetilde{g}, h) \in \widetilde{GSp}^A(W_1) \times GO(W_2) | \widetilde{\lambda}(\widetilde{g})\lambda(h)=1\}.$$
Considérons l'application $H^2(GSp(W_1), A) \longrightarrow H^2(GSp(W_1), \C^{\times})$, l'image de $[c]$  détermine le groupe $\widetilde{GSp}(W_1)$. Donc  on peut regarder $\widetilde{GSp}^A(W_1)$ comme un sous-groupe de $\widetilde{GSp}(W_1)$. Il en résulte que $\widetilde{\Gamma}^{A\frac{1}{2}}$ est un sous-groupe de $\widetilde{\Gamma}^{\frac{1}{2}}$ défini dans la Proposition  \ref{saufcasscindage}.
Donc d'après la Proposition \ref{saufcasscindage}, on peut obtenir un morphisme de groupes
$$ \widetilde{\Gamma}^{A \frac{1}{2}} \stackrel{i_{A\frac{1}{2}}}{\longrightarrow} \widetilde{Sp}^A(W).$$
Gr\^ace à ce morphisme,  on obtient une représentation lisse du groupe $\widetilde{\Gamma}^{A \frac{1}{2}}$ par restriction de la  représentation de Weil, notée  $\rho_{\psi}$. Par suite, on a des suites exactes
$$1 \longrightarrow \widetilde{Sp}^A(W_1) \times O(W_2) \longrightarrow \widetilde{\Gamma}^{A \frac{1}{2}} \stackrel{\widetilde{\lambda}}{\longrightarrow} \Lambda_{\widetilde{\Gamma}^{A \frac{1}{2}}} \longrightarrow 1$$
et
$$ 1 \longrightarrow O(W_2) \longrightarrow GO(W_2) \stackrel{\lambda}{\longrightarrow} \Lambda_{GO(W_2)} \longrightarrow 1.$$
Par définition, $\Lambda_{\widetilde{\Gamma}^{A \frac{1}{2}}} \subset \Lambda_{GO(W_2)} \subset \Lambda_{\widetilde{GSp}^A(W_1)}=F^{'\times}$, nous noterons  $G^{\widetilde{\Gamma}^{A \frac{1}{2}}}O(W_2)$ (resp.  $G^{\widetilde{\Gamma}^{A \frac{1}{2}}}\widetilde{Sp}^A(W_1)$) l'image réciproque de $\Lambda_{\widetilde{\Gamma}^{A \frac{1}{2}}}$ dans $GO(W_2)$ (resp.  $\widetilde{GSp}^A(W_1)$). Pour les groupes, on a les relations
$$G^{\widetilde{\Gamma}^{A \frac{1}{2}}}\widetilde{Sp}^A(W_1) / \widetilde{Sp}^A(W_1) \simeq G^{\widetilde{\Gamma}^{A \frac{1}{2}}}O(W_2) /O(W_2) \simeq \widetilde{\Gamma}^{A \frac{1}{2}}/{\widetilde{Sp}^A(W_1) \times O(W_2)} \simeq \Lambda_{\widetilde{\Gamma}^{A \frac{1}{2}}}.$$
D'après la Proposition \ref{lamedavalues}, les groupes $F^{'\times}/{\Lambda_{\widetilde{\Gamma}^{A \frac{1}{2}}}}$, ou $G^{\widetilde{\Gamma}^{A \frac{1}{2}}}O(W_1)/{F^{'\times}O(W_1)}$ sont abéliens d'ordre fini.

\begin{thm}
Si on définit $$\pi=c-\Ind_{\widetilde{\Gamma}^{A \frac{1}{2}}}^{G^{\widetilde{\Gamma}^{A \frac{1}{2}}}\widetilde{Sp}^A(W_1) \times G^{\widetilde{\Gamma}^{A \frac{1}{2}}}O(W_2)} \rho_{\psi},$$
alors $\pi$ est une représentation de bigraphe forte.
\end{thm}
\begin{proof}
Ceci découle du Théorème \ref{abelienbigraphedefini}.
\end{proof}
\begin{rem}
Pour présenter simplement le résultat, nous  notons aussi $G^{\Gamma}U(W_1) \times G^{\Gamma}U(W_1)$ (resp.  $\Gamma$) représenté $G^{\widetilde{\Gamma}^{A \frac{1}{2}}}\widetilde{Sp}^A(W_1) \times G^{\widetilde{\Gamma}^{A \frac{1}{2}}}O(W_2)$ ( resp.  $\widetilde{\Gamma}^{A \frac{1}{2}}$ ) en cas (ii) ci-dessus.
\end{rem}
\begin{rem}
(1) La définition de $\pi$ dépend du choix de l'homomorphisme $\Gamma\stackrel{i_{\Gamma}}{\longrightarrow}Mp(W)$, pour obtenir une correspondance intéressante, il faut choisir le morphisme \textbf{approprié} $\Gamma \stackrel{i_{\Gamma}}{\longrightarrow} Mp(W)$.\\
(2) Pour le modèle realisé la représentation de bigraphe $\pi$ du groupe $G^{\Gamma}U(W_1) \times G^{\Gamma}U(W_2)$, nous  indiquons une voie possible: le premier étape, on peut essayer de écrire le modèle realisé la représentation $\rho_{\psi}$ du groupe intermédiaire $\Gamma$ qui est une extension de $U(W_1) \times U(W_2)$ par un groupe abélien $\Lambda_{\Gamma}$ en utilisant les démonstrations des Lemme \ref{scindagehyperbolique}--- Lemme \ref{scindagequaternions} et de la  Proposition \ref{saufcasscindage}; le deuxiéme étape, utilisons la technique canonique en représentation locale pour trouver finalement un modèle de $c-\Ind_{\Gamma}^{G^{\Gamma}U(W_1) \times G^{\Gamma}U(W_2)} \rho_{\psi}$.
\end{rem}
\begin{rem}
Si $W= W^{(1)} \oplus W^{(2)}$ avec $W^{(i)}=W_1^{(i)} \otimes_{D^{(i)}}W_2^{(i)}$,   on sait que $\Big(U(W_1^{(1)}) \times U(W_1^{(2)}), U(W_2^{(1)})\times U(W_2^{(2)})\Big)$ est une  paire réductive duale de $Sp(W)$, en ce cas,  on prend les groupes intermédiaires $\Gamma^{(i)}$ et construit les représentations $\pi^{(i)}$ de bigraphe forte des groupes $G^{\Gamma^{(i)}}U(W_1^{(i)}) \times G^{\Gamma^{(i)}}U(W_2^{(i)})$ pour $i=1,2$. D'après la Proposition \ref{bigraphesproduits}, $\pi_1 \times \pi_2$   est aussi une représentation du bigraphe forte du groupe $\Big(G^{\Gamma^{(1)}}U(W_1^{(1)}) \times G^{\Gamma^{(2)}}U(W_1^{(2)})\Big) \times \Big(G^{\Gamma^{(1)}}U(W_2^{(1)}) \times G^{\Gamma^{(2)}}U(W_2^{(2)})\Big)$.
\end{rem}
\subsection{Exemples I}: Reprenons les notations au début de cette sous-section. Supposons qu'on a des décompositions de Witt $W_i= W_i^0 \oplus m_i H_i$ avec $W_i^0$ un espace anisotrope et $H_i$ le plan hyperbolique.

\subsubsection{Cas (1)}
Soient $D=F'$, $\epsilon_1=-1$, $\epsilon_2=1$, $U(W_1)=Sp(W_1)$, $U(W_2)=O(W_2)$; $GU(W_1)=GSp(W_1)$, $GU(W_2)=GO(W_2)$.\\

\paragraph{(i)} $\dim_{F'} W_2^0=0,4$, $\Gamma=\{(g,h) \in GSp(W_1) \times GO(W_2)| \lambda(g) \lambda(h)=1\}$, $\Lambda_{\Gamma}=F^{'\times}$, $G^{\Gamma}Sp(W_1)=GSp(W_1)$, $G^{\Gamma}O(W_2)=GO(W_2)$.  Donc $\pi=c-\Ind_{\Gamma}^{GSp(W_1) \times GO(W_2)} \rho_{\psi}$ est une représentation de bigraphe forte.\\

\paragraph{(ii)} $\dim_{F'} W_2^0=1$, $\Gamma^A:=\widetilde{\Gamma}^{A \frac{1}{2}}=\{(\widetilde{g},h) \in \widetilde{GSp}^A(W_1)\times GO(W_2)| \widetilde{\lambda}(\widetilde{g})\lambda(h)=1\}$, $\Lambda_{\widetilde{\Gamma}^{A \frac{1}{2}}}=F^{\times 2}$, $\widetilde{GSp}^A_+(W_1):= G^{\widetilde{\Gamma}^{A \frac{1}{2}}}\widetilde{Sp}^A(W_1)=\{\widetilde{g} \in \widetilde{GSp}^A(W_1) | \widetilde{\lambda}(\widetilde{g}) \in F^{\times 2}\}$, $G^{\widetilde{\Gamma}^{A \frac{1}{2}}}O(W_2)=GO(W_2)$. Donc $\pi=c-\Ind_{\Gamma^A}^{\widetilde{GSp}^A_+(W_1)\times GO(W_2)} \rho_{\psi}$ est une représentation de bigraphe forte.\\

\paragraph{(iii)} $\dim_{F'} W_2^0=2$, $W_2^0=E'(f)$ où $E'$ est une extension quadratique de $F'$, $f=1$ ou $f\in F' \backslash \nnn_{E'/F'}(E^{'\times})$. $\Gamma=\{ (g, h) \in GSp(W_1) \times GO(W_2) | \lambda(g) \lambda(h)=1\}$, $\Lambda_{\Gamma}=\nnn_{E'/F'}(E^{'\times})$, $GSp_+(W_1):=G^{\Gamma}Sp(W_1)=\{ g\in GSp(W_1)| \lambda(g)\in \nnn_{E'/F'}(E^{'\times})\}$, $G^{\Gamma}O(W_2)=GO(W_2)$. Donc $\pi=c-\Ind_{\Gamma}^{GSp_+(W_1) \times GO(W_2)} \rho_{\psi}$ est une représentation de bigraphe forte.\\

\paragraph{(iv)} $\dim_{F'} W_2^0=3$, $\Gamma^A:= \widetilde{\Gamma}^{A \frac{1}{2}}=\{ (\widetilde{g}, h) \in \widetilde{GSp}^A(W_1) \times GO(W_2)| \widetilde{\lambda}(\widetilde{g}) \lambda(h)=1\}$, $\Lambda_{\Gamma}=F^{'\times}$, $G^{\Gamma^A}\widetilde{Sp}^A(W_1)=\widetilde{GSp}^A(W_1)$, $G^{\Gamma^A}O(W_2)=GO(W_2)$. Donc $\pi=c-\Ind_{\Gamma^A}^{\widetilde{GSp}^A(W_1) \times GO(W_2)} \rho_{\psi}$ est une représentation de bigraphe forte.

\subsubsection{Cas (2)}
Soient $D=E'$ une extension quadratique de $F'$,
$\Gamma=\{ (g, h) \in GU(W_1) \times GU(W_2) | \lambda(g) \lambda(h)=1\}$.\\

\paragraph{(i)} Si $\dim_{E'} W_1$ et $\dim_{E'} W_2$ sont paires. En ce cas, $\Lambda_{\Gamma}=F^{'\times}$, $G^{\Gamma}U(W_i)=GU(W_i)$, $\pi=c-\Ind_{\Gamma}^{GU(W_1) \times GU(W_2)}\rho_{\psi}$ est une représentation de bigraphe forte.\\

\paragraph{(ii)} Si $\dim_{E'} W_1$, $\dim_{E'}W_2$ sont impaires, $\Lambda_{\Gamma}=\nnn_{E'/F'}(E^{'\times})$, $G^{\Gamma}U(W_i)=GU(W_i)$, $\pi=c-\Ind_{\Gamma}^{GU(W_1) \times GU(W_2)} \rho_{\psi}$ est une représentation de bigraphe forte.\\

\paragraph{(iii)} L'autre cas, supposons que $\dim_{E'} W_1$ est paire et $\dim_{E'} W_2$ est impaire. $\Lambda_{\Gamma}=\nnn_{E'/F'}(E^{'\times})$, $GU_+(W_1):= G^{\Gamma}U(W_1)=\{ g\in GU(W_1)| \lambda(g)\in \nnn_{E'/F'}(E^{'\times})\}$. $G^{\Gamma}U(W_2)=GU(W_2)$. Donc $\pi=c-\Ind_{\Gamma}^{GU_+(W_1) \times GU(W_2)} \rho_{\psi}$ est une représentation de bigraphe forte.

\subsubsection{Cas (3)}
Soient  $D$ l'unique corps de quaternions sur $F'$, $G^{\Gamma}U(W_i)=GU(W_i)$.  Alors $\pi=c-\Ind_{\Gamma}^{GU(W_1) \times GU(W_2)} \rho_{\psi}$ est une représentation de bigraphe forte.
\subsubsection{Cas (1)'}
Soient $D=F'$, $\epsilon_1=-1$, $\epsilon_2=1$, $U(W_1)=Sp(W_1)$, $U(W_2)=O(W_2)$; $GU(W_1)=GSp(W_1)$, $GU(W_2)=GO(W_2)$.\\

\paragraph{(i)'} Si $\dim_{F'}W_2$ est paire. \\
Pour chaque une extension quadratique $E'$ de $F'$, on définit: $G^{E'}Sp(W_1)=\{ g\in GSp(W_1)| \lambda(g)\in  \nnn_{E'/F'}(E^{'\times})\}$, $G^{E'}O(W_2)=\{ h\in GO(W_2)| \lambda(h)\in \nnn_{E'/F'}(E^{'\times})\}$, et un autre sous-groupe intermédiaire  $\Gamma^{E'}=\{ (g,h) \in G^{E'}Sp(W_1) \times G^{E'}O(W_2)| \lambda(g)\lambda(h)=1\}$ de $\Gamma$. Alors $\pi^{E'}=c-\Ind_{\Gamma^{E'}}^{G^{E'}Sp(W_1) \times G^{E'}O(W_2)} \rho_{\psi}|_{\Gamma^{E'}}$ est une représentation de bigraphe forte.\\

\paragraph{(ii)'} Si $\dim_{F'}W_2$ est impaire.\\
Dans ce cas, on définit: $\widetilde{GSp}^A_+(W_1)=\{ \widetilde{g} \in \widetilde{GSp}^A(W_1)| \widetilde{\lambda}(\widetilde{g})\in F^{'\times 2}\}$, $GO_+(W_2)=\{ h\in GO(W_2)| h\in GO(W_2) \lambda(h) \in F^{\times 2}\}$,  un autre sous-groupe intermédiaire  $\Gamma^A_+=\{ (g,h) \in \widetilde{GSp}^A_+(W_1) \times GO_+(W_2)| \widetilde{\lambda}(\widetilde{g})\lambda(h)=1\}$ de $\widetilde{\Gamma}^{A \frac{1}{2}}$. Alors $\pi_+=c-\Ind_{\Gamma^A_+}^{\widetilde{GSp}^A_+(W_1) \times GO_+(W_2)} \rho_{\psi}|_{\Gamma^A_+}$ est aussi une représentation de bigraphe forte.

\subsubsection{Cas (2)'}
Soit  $D=E'$ une extension quadratique de $F'$. Si on définit $GU^{E'}(W_i)= \{ g\in GU(W_i)| \lambda(g) \in \nnn_{E'/F'}(E^{'\times})\}$, $\Gamma^{E'}=\{ (g,h)\in GU^{E'}(W_i)| \lambda(g) \lambda(h)=1\}$, alors $\pi^{E'}=c-\Ind_{\Gamma^{E'}}^{GU^{E'}(W_1) \times GU^{E'}(W_2)} \rho_{\psi}|_{\Gamma^{E'}} $ est aussi une représentation de bigraphe forte.

\section{Représentations de bigraphe forte pour les groupes des similitudes II}
Dans cette sous-section, en suite à considérer comment on obtient une  représentation du bigraphe forte pour les groupes des similitudes, on va traiter les cas non scindés.\\

Reprenons les notations au début de la sous-section \ref{bigraphefortedesimilitudes1}.\\

(i)  Supposons que $W= W_1\otimes_D W_2$ n'est pas le cas dans la Proposition \ref{saufcasscindage}. Rappelons  la définition du groupe intermédiaire
$$\Gamma=\{ (g,h)\in GU(W_1) \times GU(W_2) | \lambda(g)\lambda(h)=1\}.$$
D'après le Théorème \ref{scindagedugroupeR}, l'extension $\overline{\Gamma}$ est scindée au-dessus de $\Gamma$, et  on obtient une représentation lisse du groupe $\Gamma$ par restriction de la représentation de Weil à $\Gamma$, on la note $\rho_{\psi}$. De plus, pour chaque groupe $\widetilde{GU}^A(W_i)$,  on a les suites exactes
$$1 \longrightarrow \widetilde{U}^A(W_i) \longrightarrow \widetilde{GU}^A(W_i) \stackrel{\lambda}{\longrightarrow} \Lambda_{\widetilde{GU}^A(W_i)}=\Lambda_{GU(W_i)} \longrightarrow 1.$$
On notera
$$\widetilde{\Gamma}^{A}= \{ (\widetilde{g}, \widetilde{h})\in \widetilde{GU}^A(W_1) \times \widetilde{GU}^A(W_2)| \widetilde{\lambda}(\widetilde{g})\widetilde{\lambda}(\widetilde{h})=1\},$$
donc on a la suite exacte
$$1 \longrightarrow \widetilde{U}^A(W_1) \times \widetilde{U}^A(W_2) \longrightarrow \widetilde{\Gamma}^A \stackrel{\lambda}{\longrightarrow} \Lambda_{\widetilde{\Gamma}^A}=\Lambda_{\Gamma} \longrightarrow 1.$$
On notera
$$ G^{\widetilde{\Gamma}^A}\widetilde{U}^A(W_i)= \textrm{ l'image réciproque de } \Lambda_{\widetilde{\Gamma}^A} \textrm{ par l'application  de }\lambda.$$
\'Evidemment, on a un homomorphisme
$$\widetilde{\Gamma}^A \stackrel{p}{\longrightarrow} GU(W_1) \times GU(W_2);$$
son image est $\Gamma$.  Donc on obtient une représentaion $\rho_{\psi}|_{\widetilde{\Gamma}^A}$ de $\widetilde{\Gamma}^A$, notée $\widetilde{\rho}_{\psi}$.
\begin{thm}
Si on définit
$$\pi= c-\Ind_{\widetilde{\Gamma}^A}^{G^{\widetilde{\Gamma}^A}\widetilde{U}^A(W_1) \times G^{\widetilde{\Gamma}^A}\widetilde{U}^A(W_2)} \widetilde{\rho}_{\psi},$$
alors $\pi$ est une représentation de bigraphe forte.
\end{thm}
\begin{proof}
Ceci découle du Théorème \ref{abelienbigraphedefini} et du Théorème \ref{restrictionadmissible}.
\end{proof}
\  \\
(ii) Soit $W \simeq W_1 \otimes_{F'} W_2$ un espace symplectique sur $F'$ où $W_1$( resp.  $W_2$) est un espace symplectique (resp.  orthogonal ) de dimension $2n$ ( resp.  $2m-1$) avec $n\geq 1$ et $m\geq 1$. On a deux  suites exactes
$$1  \longrightarrow Sp(W_1) \stackrel{i}{ \longrightarrow}  GSp(W_1) \stackrel{\lambda}{\longrightarrow} F^{'\times} \longrightarrow 1,$$
et
$$1 \longrightarrow  O(W_2)  \stackrel{i}{ \longrightarrow} GO(W_2) \stackrel{\lambda}{\longrightarrow} \Lambda_{GO(W_2)} \longrightarrow 1.$$
On déduit des  suites exactes de Hochschild-Serre
$$\cdots \longrightarrow H^2(F^{'\times}, A) \longrightarrow H^2( GSp(W_1), A) \stackrel{i^2}{\longrightarrow} H^2(Sp(W_1) , A)  \longrightarrow \cdots$$
et
$$\cdots \longrightarrow H^2(\Lambda_{GO(W_2)}, A) \longrightarrow H^2( GO(W_2), A) \stackrel{i^2}{\longrightarrow} H^2(O(W_2) , A)  \longrightarrow \cdots.$$
Prenons un élément $[c_{GSp}]$ (resp.  chaque $[c_{GO}]$)  de $H^2(GSp(W_1), A)$ ( resp.  $H^2(GO(W_2), A)$) tel que $i^2([c_{GSp}])=[c_{Rao}]$. La classe $[c_{GSp}]$ (resp. $[c_{GO}]$) détermine un groupe $\widetilde{GSp}^A(W_1)$( resp.  $\widetilde{GO}^A(W_2)$).\\
Rappelons que nous avons déjà défini un sous-groupe intermédiaire dans le cas scindé (ii)
$$\widetilde{\Gamma}^{A\frac{1}{2}}=\{ (\widetilde{g}, h)\in \widetilde{GSp}^A(W_1) \times GO(W_2) | \widetilde{\lambda}(\widetilde{g}) \lambda(h)=1\}.$$
Gr\^ace au morphisme
$$\widetilde{\Gamma}^{A\frac{1}{2}} \stackrel{\widetilde{i}_{\frac{1}{2}}}{\longrightarrow} \widetilde{Sp}(W),$$
on obtient une représentation lisse du groupe $\widetilde{\Gamma}^{A\frac{1}{2}}$, i.e. $\rho_{\psi}=\omega_{\psi}|_{\widetilde{\Gamma}^{A\frac{1}{2}}}$, où $\omega_{\psi}$ est la représentation de Weil du groupe $\widetilde{Sp}(W)$.\\
\underline{Dans le cas non scindé}, on peut définir un autre sous-groupe intermédiaire
$$\widetilde{\Gamma}^A=\{ (\widetilde{g}, \widetilde{h}) \in \widetilde{GSp}^A(W_1)  \times \widetilde{GO}^A(W_2) | \widetilde{\lambda}(\widetilde{g}) \widetilde{\lambda}(\widetilde{h})=1\}.$$
On a les suites exactes canoniques
$$1 \longrightarrow \widetilde{Sp}^A(W_1) \longrightarrow \widetilde{GSp}^A(W_1) \stackrel{\lambda}{\longrightarrow} \Lambda_{\widetilde{GSp}^A(W_1)}=F^{' \times} \longrightarrow 1 \cdots (1)$$
$$1 \longrightarrow \widetilde{O}^A(W_2) \longrightarrow \widetilde{GO}^A(W_2) \stackrel{\lambda}{\longrightarrow} \Lambda_{\widetilde{GO}^A(W_2)}= \Lambda_{GO(W_2)} \longrightarrow 1 \cdots (2)$$
et
$$1 \longrightarrow \widetilde{Sp}^A(W_1) \times \widetilde{O}^A (W_2) \longrightarrow \widetilde{\Gamma}^A \longrightarrow \Lambda_{\widetilde{\Gamma}^A} \longrightarrow  1.$$
On notera
$$G^{\widetilde{\Gamma}^A}\widetilde{Sp}^A (W_1)\textrm{  l'image réciproque de } \Lambda_{\widetilde{\Gamma}^A} \textrm{ par l'application } \lambda \textrm{ dans la suite } (1),$$
et
$$G^{\widetilde{\Gamma}^A}\widetilde{O}^A(W_2) \textrm{  l'image réciproque de } \Lambda_{\widetilde{\Gamma}^A} \textrm{ par l'application } \lambda \textrm{ dans la suite } (2).$$
Comme on a l'homorphisme canonique
$$1 \longrightarrow A \longrightarrow  \widetilde{GO}^A(W_2) \longrightarrow GO(W_2) \longrightarrow 1,$$
on obtient un homomorphisme
$$\widetilde{\Gamma}^A \longrightarrow \widetilde{GSp}^A(W_1) \times  GO(W_2).$$
Son image est $\widetilde{\Gamma}^{A\frac{1}{2}}$, ainsi on a une représentation lisse $\rho_{\psi}|_{\widetilde{\Gamma}^A}$ du groupe $\widetilde{\Gamma}^A$, notée  $\widetilde{\rho}_{\psi}$.\\
Rappels: pour la représentation de Weil $\omega_{\psi}$, on sait que $\omega_{\psi}|_{\widetilde{Sp}^A(W_1) \times O(W_2)}$ est une représentation de bigraphe forte.
On  définit
$$\widetilde{O}^A(W_2)=\{ \widetilde{g} \in \widetilde{GO}^A(W_2) | \widetilde{\lambda}( \widetilde{g})=1 \}$$
d'où une suite exacte
$$1 \longrightarrow A \longrightarrow \widetilde{O}^A(W_2) \stackrel{ p^A}{\longrightarrow } O(W_2) \longrightarrow 1.$$
Gr\^ace au morphisme $p^A$, on obtient une représentation $\omega_{\psi}|_{\widetilde{Sp}^A(W_1) \times \widetilde{O}^A(W_2)}$. Par la démonstration de la Remarque \ref{ouvertmorphisme},
$\omega_{\psi}|_{\widetilde{Sp}^A(W_1) \times \widetilde{O}^A(W_2)}$ est une représentation de bigraphe forte. Par définition
$$\widetilde{\rho}_{\psi}|_{\widetilde{Sp}^A(W_1) \times \widetilde{O}^A(W_2)}=\omega_{\psi}|_{\widetilde{Sp}^A(W_1) \times \widetilde{O}^A(W_2)}.$$
\begin{thm}
Si on définit
$$\pi=c-\Ind_{\widetilde{\Gamma}^A}^{G^{\widetilde{\Gamma}^A}\widetilde{Sp}^A(W_1) \times G^{\widetilde{\Gamma}^A}\widetilde{O}^A(W_2)} \widetilde{\rho}_{\psi},$$
alors $\pi$ est une représentation de bigraphe forte.
\end{thm}
\begin{proof}
Ceci découle du Théorème \ref{abelienbigraphedefini}.
\end{proof}
\subsection{Exemples II}
Soient $F$ un corps local non archimédien  de caractéristique résiduelle impaire $p$, $F'$ une extension finie de $F$, et $D$ un corps de quaternions sur $F'$  muni de l'involution canonique $\tau$.\\

$(W_1, \langle, \rangle_1)$( resp.  $(W_2, \langle, \rangle_2)$) un espace $\epsilon_1$( resp.  $\epsilon_2$)-hermitien sur $D$ à droite( resp.  à gauche) tel que $\epsilon_1 \epsilon_2=-1$, $W=W_1 \otimes_D W_2$ un espace symplectique sur $F$, muni de la forme $\tr_{F'/F}\big(\langle, \rangle_1 \otimes \tau(\langle, \rangle_2)\big)$.\\

$\omega_{\psi}$ la représentation de Weil liée au caractère non trivial $\psi$, $U(W_i)$ (resp.  $GU(W_i)$)  le groupe unitaire (resp.   le groupe unitaire de similitudes) de $(W_i, \langle, \rangle_i)$, $A$ un groupe abélien fini fixé d'ordre $n$ et $2| n, (n,p)=1$; $\widetilde{GU}(W_i)$ le revêtement central de $GU(W_i)$ par $A$  défini à la sous-section \ref{appendice2}, $\widetilde{\lambda}$ le morphisme canonique de $\widetilde{GU}(W)$ dans $F^{'\times }$, $\widetilde{U}(W_i)$ son noyau et $\Lambda_{\widetilde{GU}}(W_i)$ son image.\\

Notons $\widetilde{\Gamma}$ le groupe intermédiaire constitué des éléments $(\widetilde{g}, \widetilde{h})$ de $\widetilde{GU}(W_1) \times \widetilde{GU}(W_2)$ tels que $\widetilde{\lambda}(\widetilde{g})\widetilde{\lambda}(\widetilde{h})=1$; $\widetilde{\Gamma} \stackrel{i_{\widetilde{\Gamma}}}{\longrightarrow} \widetilde{Sp}(W)$ le homomorphisme naturel et   $\widetilde{\rho}_{\psi}=\omega_{\psi}|_{\widetilde{\Gamma}}$.\\

Choisissons des décompositions de Witt $W_i= W_i^0 \oplus m_i H_i$ avec $W_i^0$ un espace anisotrope et $H_i$ le plan hyperbolique.
\subsubsection{Cas (1)}
Soient  $D=F'$, $\epsilon_1=-1$, $\epsilon_2=1$,
$U(W_1)=Sp(W_1)$, $U(W_2)=O(W_2)$; $GU(W_1)=GSp(W_1)$, $GU(W_2)=GO(W_2)$.

Discutons suivant la dimension de la partie anisotrope de $W_2$.\\

\paragraph{(i)} $\dim_{F'} W_2^0=0,4$, $\Lambda_{\widetilde{\Gamma}}=F^{'\times}$, $\widetilde{GSp}^{\widetilde{\Gamma}}(W_1)=\widetilde{GSp}(W_1)$, $\widetilde{GO}^{\widetilde{\Gamma}}(W_2)=\widetilde{GO}(W_2)$.  Donc $\pi=c-\Ind_{\widetilde{\Gamma}}^{\widetilde{GSp}(W_1) \times \widetilde{GO}(W_2)} \widetilde{\rho}_{\psi}$ est une représentation de bigraphe forte.\\

\paragraph{(ii)} $\dim_{F'} W_2^0=1$,  $\Lambda_{\widetilde{\Gamma}}=F^{'\times 2}$, $\widetilde{GSp}^{+}(W_1):= \widetilde{GSp}^{\widetilde{\Gamma}}(W_1)=\{\widetilde{g} \in \widetilde{GSp}(W_1) | \widetilde{\lambda}(\widetilde{g}) \in F^{'\times 2}\}$, $\widetilde{GO}^{\widetilde{\Gamma}}(W_2)=\widetilde{GO}(W_2)$. Donc $\pi=c-\Ind_{\widetilde{\Gamma}}^{\widetilde{GSp}^+(W_1)\times \widetilde{ GO}(W_2)} \widetilde{\rho}_{\psi}$  est une représentation de bigraphe forte.\\

\paragraph{(iii)} $\dim_{F'} W_2^0=2$, $W_2^0=E'(f)$ où $E'$ est une extension quadratique $F'$, $f=1$ ou $f\in F' \backslash \nnn_{E'/F'}(E^{'\times})$. $\Lambda_{\widetilde{\Gamma}}=\nnn_{E'/F'}(E^{\times})$, $\widetilde{GSp}^+(W_1):=\widetilde{GSp}^{\widetilde{\Gamma}}(W_1)=\{ \widetilde{g}\in \widetilde{GSp}(W_1)| \widetilde{\lambda}(\widetilde{g})\in \nnn_{E/F}(E^{'\times})\}$, $\widetilde{GO}^{\widetilde{\Gamma}}(W_2)=\widetilde{GO}(W_2)$. Donc $\pi=c-\Ind_{\widetilde{\Gamma}}^{\widetilde{GSp}^+(W_1) \times \widetilde{GO}(W_2)} \widetilde{\rho}_{\psi}$ est une représentation de bigraphe forte.\\

\paragraph{(iv)} $\dim_{F'} W_2^0=3$,  $\Lambda_{\widetilde{\Gamma}}=F^{'\times}$, $\widetilde{GSp}^{\widetilde{\Gamma}}(W_1)=\widetilde{GSp}(W_1)$, $\widetilde{GO}^{\widetilde{\Gamma}}(W_2)=\widetilde{GO}(W_2)$. Donc $\pi=c-\Ind_{\widetilde{\Gamma}}^{\widetilde{GSp}(W_1) \times \widetilde{GO}(W_2)} \widetilde{\rho}_{\psi}$  est une représentation de bigraphe forte.
\subsubsection{Cas (2)}
Soit  $D=E'$ une extension quadratique de $F'$.\\

\paragraph{(i)} Si $\dim_{E'} W_1$ et $\dim_{E'} W_2$ sont paires. En ce cas, $\Lambda_{\widetilde{\Gamma}}=F^{'\times}$, $\widetilde{GU}^{\widetilde{\Gamma}}(W_i)=\widetilde{GU}(W_i)$, $\pi=c-\Ind_{\widetilde{\Gamma}}^{\widetilde{GU}(W_1) \times \widetilde{GU}(W_2)} \widetilde{\rho}_{\psi}$  est une représentation de bigraphe forte.\\

\paragraph{(ii)} Si $\dim_{E'} W_1$, $\dim_{E'} W_2$ sont impaires, $\Lambda_{\widetilde{\Gamma}}=\nnn_{E'/F'}(E^{'\times})$, $\widetilde{GU}^{\widetilde{\Gamma}}(W_i)=\widetilde{GU}(W_i)$, $\pi=c-\Ind_{\widetilde{\Gamma}}^{\widetilde{GU}(W_1) \times \widetilde{GU}(W_2)} \widetilde{\rho}_{\psi}$ est une représentation de bigraphe forte.\\

\paragraph{(iii)} Supposons que $\dim_{E'} W_1$ est paire et $\dim_{E'} W_2$ est impaire. Alors $\Lambda_{\widetilde{\Gamma}}=\nnn_{E'/F'}(E^{'\times})$, $\widetilde{GU}_+(W_1):= \widetilde{GU}^{\widetilde{\Gamma}}(W_1)=\{ \widetilde{g}\in \widetilde{GU}(W_1)| \widetilde{\lambda}(\widetilde{g})\in \nnn_{E'/F'}(E^{'\times})\}$. $\widetilde{GU}^{\widetilde{\Gamma}}(W_2)=\widetilde{GU}(W_2)$. Donc $\pi=c-\Ind_{\widetilde{\Gamma}}^{\widetilde{GU}_+(W_1) \times \widetilde{GU}(W_2)} \widetilde{\rho}_{\psi}$ est une représentation de bigraphe forte.
\subsubsection{Cas (3)}
Soient  $D$ l'unique corps de quaternions sur $F'$, $\widetilde{GU}^{\widetilde{\Gamma}}(W_i)=\widetilde{GU}(W_i)$. Alors  $\pi=c-\Ind_{\widetilde{\Gamma}}^{\widetilde{GU}(W_1) \times \widetilde{GU}(W_2)} \rho_{\psi}$ est une représentation de bigraphe forte.
\subsubsection{Cas (1)'}
Soient $D=F'$, $\epsilon_1=-1$, $\epsilon_2=1$,
$U(W_1)=Sp(W_1)$, $U(W_2)=O(W_2)$; $GU(W_1)=GSp(W_1)$, $GU(W_2)=GO(W_2)$.\\

\paragraph{(i)'} Si $\dim_{F'}W_2$ est paire. \\
Pour chaque une extension quadratique $E'$ de $F'$, On définit: $\widetilde{GSp}^{E'}(W_1)=\{ \widetilde{g}\in \widetilde{GSp}(W_1)| \widetilde{\lambda}(\widetilde{g})\in  \nnn_{E'/F'}(E^{'\times})\}$, $\widetilde{GO}^{E'}(W_2)=\{ \widetilde{h}\in \widetilde{GO}(W_2)| \widetilde{\lambda}(\widetilde{h})\in \nnn_{E'/F'}(E^{'\times})\}$,  et un autre sous-groupe intermédiaire  $\widetilde{\Gamma}^{E'}=\{ (\widetilde{g},\widetilde{h}) \in \widetilde{GSp}^{E'}(W_1) \times \widetilde{GO}^{E'}(W_2)| \widetilde{\lambda}(\widetilde{g})\widetilde{\lambda}(\widetilde{h})=1\}$ de $\widetilde{\Gamma}$. Alors $\pi^{E'}=c-\Ind_{\widetilde{\Gamma}^{E'}}^{\widetilde{GSp}^{E'}(W_1) \times \widetilde{GO}^{E'}(W_2)} \rho_{\psi}|_{\widetilde{\Gamma}^{E'}}$ est une représentation de bigraphe forte.\\

\paragraph{(ii)'} Si $\dim_{F'}W_2$ est impaire.\\
Dans ce cas, on définit: $\widetilde{GSp}_+(W_1)=\{ \widetilde{g} \in \widetilde{GSp}(W_1)| \widetilde{\lambda}(\widetilde{g})\in F^{'\times 2}\}$, $\widetilde{GO}_+(W_2)=\{ \widetilde{h}\in \widetilde{GO}(W_2)|  \widetilde{\lambda}(\widetilde{h}) \in F^{'\times 2}\}$, et  un autre sous-groupe intermédiaire  $\widetilde{\Gamma}_+=\{ (\widetilde{g},\widetilde{h}) \in \widetilde{GSp}_+(W_1) \times \widetilde{GO}_+(W_2)| \widetilde{\lambda}(\widetilde{g})\widetilde{\lambda}(\widetilde{h})=1\}$ de $\widetilde{\Gamma}$, alors $\pi_+=c-\Ind_{\widetilde{\Gamma}_+}^{\widetilde{GSp}_+(W_1) \times \widetilde{GO}_+(W_2)} \rho_{\psi}|_{\widetilde{\Gamma}_+}$ est une représentation de bigraphe forte.
\subsubsection{Cas (2)'}
Soit $D=E'$ une extension quadratique de $F'$.
Si on définit $\widetilde{GU}^{E'}(W_i)= \{ \widetilde{g}\in \widetilde{GU}(W_i)| \widetilde{\lambda}(\widetilde{g}) \in \nnn_{E'/F'}(E^{'\times})\}$, $\widetilde{\Gamma}^{E'}=\{ (\widetilde{g},\widetilde{h})\in \widetilde{GU}^{E'}(W_i)| \widetilde{\lambda}(\widetilde{g}) \widetilde{\lambda}(\widetilde{h})=1\}$, alors $\pi^{E'}=c-\Ind_{\widetilde{\Gamma}^{E'}}^{\widetilde{GU}^{E'}(W_1) \times \widetilde{GU}^{E'}(W_2)} \rho_{\psi}|_{\widetilde{\Gamma}^{E'}} $ est aussi une représentation de bigraphe forte.
\section{Appendice I. Des lemmes en représentation locale}
En vue d'application, nous citons et montrons des résultats en représentation locale. Soient $G$ un groupe localement compact totalement discontinu, $H$ un sous-groupe fermé de $G$, $\Delta_G$ (resp. $\Delta_H$) les fonctions unimodulaires de $G$ (resp. $H$)[cf. \cite{BernZ}, Page 10]; $(\pi, V)$ (resp. $(\rho, W)$) une représentation lisse de $G$ (resp. $H$), et $(\check{\pi}, \check{V})$ (resp. $(\check{\sigma}, \check{W})$) la représentation contragriénte de $(\pi, V)$ (resp. $(\sigma, W)$).
\begin{lemme}\label{contradigiente}
$$(c-\Ind_H^G\sigma)^{\vee} \simeq  \Ind_H^G \check{\sigma} \Delta_G/{\Delta_H}.$$
\end{lemme}
\begin{proof}
Ceci découle de [\cite{BernZ}, Page 23]
\end{proof}
\begin{lemme}
Supposons $\Delta_G/{\Delta_H}=1$. Si $(\sigma, W)$ est une représentation lisse admissible de $H$, alors
$$\Ind_H^G \sigma \simeq (c-\Ind_H^G \check{\sigma})^{\vee}.$$
En particulier, si $\Ind_H^G \sigma$ est admissible, alors $(\Ind_H^G \sigma)^{\vee} \simeq c-\Ind_H^G \check{\sigma}$.
\end{lemme}
\begin{proof}
C'est une consequence du lemme \ref{contradigiente}.
\end{proof}
\begin{lemme}\label{deltaequality}
\begin{itemize}
 \item[(i)] Si $H$ est un sous-groupe ouvert de $G$, alors $\Delta_H=\Delta_G|_H$.
 \item[(ii)] Si $H$ est un sous-groupe distingué de $G$, et que $G/H$ est abélien, alors $\Delta_H=\Delta_G|_H$.
 \end{itemize}
\end{lemme}
\begin{proof}
(1) Par hypothèse, on a une suite exacte
$$1 \longrightarrow S^{\star} (H\setminus G) \longrightarrow S^{\star}(G) \stackrel{p_{H\setminus G}^{\star}}{\longrightarrow} S^{\star}(H) \longrightarrow 1.$$
Si $\mu_G$ est une mesure duale de Haar à gauche de $G$, alors $p_{H\setminus G}^{\star}(\mu_G)$ l'est aussi de $H$, d'où le résultat.\\
(2) Si $\mu_H$ est une mesure duale de Haar à gauche de $H$, et que $\mu_{G/H}$ est une mesure duale de Haar de $G/H$. Rappelons qu'on a un homomorphisme
$$ S(G) \longrightarrow S(G/H);$$
 $$ f \longmapsto \overline{f},$$
  où $\overline{f}(gH):=\int_{H} f(gh) \mu_H(h)$. Comme $\bigg(g^{-1}\supp(f)\bigg) \cap H$ et $\supp(f)H/H$ sont compacts, on voit que $\overline{f}$ est bien définie. On définit un élément $\mu_G \in S^{\star}(G)$ de la façon suivante:\\
  $$\mu_G(f):=\int_{G/H}\overline{f}(\overline{g}) d\mu_{G/H}(\overline{g}).$$
  Si on définit une application $\rho(g_0):G \longrightarrow G; g \longmapsto g_0 g$, alors $\mu_G(\rho(g_0) f)=\int_{G/H} \overline{\rho(g_0) f} (\overline{g}) d\mu_{G/H}(\overline{g})$ $=\int_{G/H} \overline{f}(\overline{g}_0^{-1} \overline{g}) d\mu_{G/H}(\overline{g})=\mu_G(f)$, i.e. $\mu_G$ est une mesure duale de Haar à gauche du groupe $G$. Si $\delta(g): G \longrightarrow G; g' \longmapsto g' g^{-1}$, alors par définition[cf. \cite{BernZ}], $\delta(g)\mu_G=\Delta(g) \mu_G$. En particulier, pour $h_0\in H$, $\overline{\delta(h_0^{-1} )f} (\overline{g})=\int_Hf(ghh_0) d\mu_H(h)=\Delta_H(h_0) \overline{f}(\overline{g})$. Il en résulte que $\Big( \delta(h_0) \mu_G\Big) (f) = \mu_G\Big( \delta(h_0^{-1})f\Big)$ $=\int_{G/H} \overline{\delta(h_0^{-1}) f}(\overline{g}) d\mu_{G/H}(\overline{g})$
  $= \int_{G/H} \Delta_H(h_0) \overline{f}(\overline{g})d\mu_{G/H}(\overline{g})=\Delta_H(h_0) \mu_G(f)$, d'où le résultat.
\end{proof}

\begin{lemme}\label{admissibleduale}
Si $\pi|_H$ est une représentation lisse admissible du groupe $H$, alors $(\pi|_H)^{\vee} \simeq \check{\pi}|_H$.
\end{lemme}
\begin{proof}
Par hypothèse, $\pi$ est aussi une représentation lisse admissible de $G$, et on a des homomorphismes  canoniques
$$\check{\pi}|_H \hookrightarrow (\pi|_H)^{\vee} \qquad \cdots (1)$$
et
$$ \pi|_H  \hookrightarrow (\check{\pi}|_H)^{\vee} \qquad \cdots (2).$$
D'après (1), on sait que $\check{\pi}|_H$ est aussi une représentation lisse admissible de $H$.\\
 D'après (2), on a
$$\check{\pi}|_H \simeq (\check{\pi}|_H)^{\vee \vee} \twoheadrightarrow (\pi|_H)^{\vee},$$
donc on trouve le résultat.
\end{proof}
\begin{thm}[BERNSTEIN, ZELEVINSKY]\label{frobeniusreciproquecind}
$$\Hom_G(c-\Ind_H^G \rho, \check{\pi}) \simeq \Hom_H(\frac{\Delta_G}{\Delta_H} \rho, (\pi|_H)^{\vee}).$$
\end{thm}
\begin{proof}
Ceci découle de [\cite{BernZ}, Pages 23-24, Proposition].
\end{proof}
\begin{prop}\label{frobeniusreciproquecindvariante}
Conservons les hypothèses dans le Lemme \ref{deltaequality} et le Lemme \ref{admissibleduale}. On a
$$\Hom_G(c-\Ind_H^G \rho, \pi) \simeq \Hom_H( \rho, \pi|_H).$$
\end{prop}
\begin{proof}
Ceci est une conséquence du Théorème \ref{frobeniusreciproquecind}.
\end{proof}
\begin{prop}\label{resctrictiondeind}
Si $G_1$ est un sous-groupe de $G$ tel que
$$H\cap G_1 \backslash G_1 \stackrel{e}{\longrightarrow} H\backslash G$$
est une bijection, alors

$$\Res_{G_1}^G\big( c-\Ind_H^G\rho\big) \simeq c-\Ind_{H\cap G_1}^{G_1} \big( \Res_{H\cap G_1}^H \rho\big).$$
\end{prop}
\begin{proof}
(0) \underline{Bijection=Homéomorphisme}:  $e$ est continue et bijective, il reste à vérifier ce qui est fermée. Soit $\overline{P}$ un ensemble fermé de $H\cap G_1 \setminus G_1$. On note $P$ l'image inverse de $\overline{P}$ dans $G_1$ par $G_1 \longrightarrow H_1\cap G_1 \setminus G_1$. Ce qui résulte que $H\cdot P$ est un ensemble fermé de $G$, et $\Big( H\setminus H\cdot P\Big)^c=H \setminus \big(H\cdot P\big)^c$. Donc $e(\overline{P})=H \setminus H\cdot P$ est aussi fermé dans $H \setminus G$, d'où le résultat. \\
(1)\underline{ Rappelons la définition}: \\
$c-\Ind_H^G(W)=\{ f: G \longrightarrow W| $
\begin{itemize}
\item[(a)] $f(hg)=\rho(h) f(g)$ pour $g\in G, h\in H$,
\item[(b)] il existe un sous-groupe ouvert compact $K$ de $G$( dépendant de $f$) tel que $f(gk)=f(g)$ pour $g\in G$, $k\in K $,
\item[(c)] il existe une partie compact $K_f$ de $G$ tel que $supp(f) \subset HK_f\}$.
\end{itemize}
(2)\underline{ Vectors $K_1$-invariants (i)}. Soit $K_1$ un sous-groupe ouvert compact de $G_1$ et soit $\Omega$ un système de représentants des doubles classes $H\cap G_1 \backslash G_1 /K_1$. Posons $K_{1_{g_1}}= H\cap G_1 \cap g_1K_1 g_1^{-1}$ pour $g_1 \in \Omega$. On sait que [cf. \cite{Rena}, Page 83, lemme]
 $$(c-\Ind_{H\cap G_1}^{G_1} \rho)^{K_1} \stackrel{i}{\simeq} \{ f: \Omega \longrightarrow W| f(g_1) \in W^{K_{1_{g_1}}} \textrm{ pour tout } g_1 \in \Omega \textrm{ et } f \textrm{ ait un support fini } \}, $$
 où $i$ est  la restriction des fonctions de $c-\Ind_{H\cap G_1}^{G_1} W$ à $\Omega$.\\
 (2)\underline{ Vectors $K_1$-invariants (ii)}. Maintenant, on considère $\Big( \Res_{G_1}^G c-\Ind_H^G W\Big)^{K_1}$. Par hypothèse, on a
 $$H_1 \setminus G_1/ K_1 \simeq H\setminus G / K_1.$$
 Reprenons $\Omega$ comme le système  de représentants des doubles classes $H \setminus G /K_1$ et $K_{1_{g_1}}=H \cap g_1K_1 g_1^{-1}= \big( H\cap G_1 \big) \cap g_1 K_1 g_1^{-1}$ pour $g_1 \in \Omega$. On considère la restriction des fonctions de $\big( c-\Ind_H^G W\big)^{K_1}$ à $\Omega$.\\
 (a) Toute fonction $f$ dans $\big(c-\Ind_H^G W\big)^{K_1}$, vérifie:
 $$ f(hg_1k_1) = \rho(h) f(g_1) \textrm{ pour } h\in H, g_1\in \Omega, k_1 \in K_1.$$
 Il en résulte que la restriction de $f$ à $\Omega$ détermine bien $f$.\\
 (b) Soient $g_1 \in \Omega$, $h\in K_{1_{g_1}}=H\cap g_1 K_1 g_1^{-1}$, $f \in \big(c-\Ind_H^G W\big)^{K_1}$,  on a
 $$\rho(h) f(g_1)=f(hg_1)=f(g_1g_1^{-1} h g_1)=f(g_1),$$
 donc $f(g_1)$ appartient à $W^{K_{1_{g_1}}}$.\\
 (c) Nous affirmons que toute fonction $f$ sur $\Omega$ de support fini à valeurs dans $W$ vérifiant $f(g) \in W^{K_{1_{g}}}$ pour $g\in \Omega$, peut se relever en une fonction de $\big(c-\Ind_H^G W\big)^{K_1}$. Il reste à montrer que $f\in c-\Ind_H^G W$.
 \begin{itemize}
\item[(i)] D'abord, le composé des applications   $G_1 \longrightarrow G_1/{H_1} \simeq G/H$ est continu et ouvert, donc pour un sous-groupe ouvert compact $C$ (resp. $C_1$) de $G$ (resp. $G_1$), et $g\in G_1$, on a les égalités
    $$HgC=HgC^{(1)},$$
    et
    $$HgC^{(0)}=HgC_1.$$
    pour un voisinage ouvert compact $C^{(1)}$ (resp. $C^{(0)}$) de l'élément $1_{G_1}$ (resp. $1_{G}$) dans $G_1$ (resp. $G$).
\item[(ii)] Soit $g\in \Omega$. Si $HgK_1=HgE_g$ pour un voisinage ouvert compact $E_g$ de l'élément $1_G$ dans $G$. On suppose que $K_1 \subset E\cap G_1$ pour un sous-groupe ouvert compact $E$ de $G$. Comme $HgE\supset HgE_g$, on suppose que $E_g \subset E$. On note
    $$E_{\Omega}:=\cap_{g\in \Omega} E_g,$$
    qui est un ensemble compact de $G$. On a l'égalité
    $$E\setminus E_{\Omega}= \cup_{g\in \Omega} (E \setminus E_g).$$
    Comme $E\setminus E_{\Omega}$ est compact, il existe un ensemble $\{ g^{(1)}, \cdots, g^{(n)}\}$ dans $G$, tels que
    $E \setminus E_{\Omega}=\cup_{i=1}^n E\setminus E_{g^{(i)}}$, d'où un ensemble ouvert compact $E_{\Omega}=\cap_{i=1}^n E_{g^{(i)}}$. On suppose que $E_{\Omega}$ contient un sous-groupe ouvert compact $E_0$ de $G$.\\
    Supposons que $\supp(f) \cap \Omega=\{ g_1, \cdots, g_n\}$. On prend un sous-groupe ouvert compact $F_0$ de $G$ tel que
    $$F_0 \subset E_0$$
    et
    $$F_0 \cap G_1= F_0^{(1)} \subseteq K_1.$$
\item[(iii)] Par définition, $f(g_i)=v_{g_i} \in W^{K_{1_{g_i}}}$. On prend des sous-groupes ouverts compacts $F_i$ de $G$ tels que
\begin{equation}\label{equationstar}
v_{g_i}\in W^{F_{i_{g_i}}}
\end{equation}
où  $F_{i_{g_i}}= g_i F_i g_i^{-1} \cap H$,
et  $$F_i \subseteq F_0.$$
Supposons que
$$Hg_i \big(F_i \cap G_1\big)\supseteq Hg_i L_i, $$
pour des sous-groupes ouverts compacts $L_i$ de $G$ satisfaisant à $L_i \subseteq F_i$ pour $1 \leq i \leq n$. \\
On définit un sous-groupe ouvert compact de $G$ de la façon suivante
$$K=\cap_{i=1}^n L_i.$$
Donc on a
$$Hg_i K=Hg_i K_i^{(1)} \subseteq Hg_i L_i \subseteq Hg_i \big( F_i \cap G_1\big).$$
On suppose que $K_i^{(1)} \subseteq F_i \cap G_1$. Si $k\in K$, on a
$$k=g_i^{-1} h_i g_i l_i$$
pour $h_i \in H \cap g_i F_i g_i^{-1}=F_{i_{g_i}}$, et $l_i \in K_i^{(1)} \subseteq F_i \cap G_1 \subseteq F_0 \cap G_1$.\\
Donc
$$f(g_i k)=f(g_ig_i^{-1}h_ig_il_i)=f(h_ig_il_i)=f(h_ig_i)$$
$$=\rho(h_i) f(g_i)=\rho(h_i)v_{g_i}\stackrel{ \textrm{ l'égalité (\ref{equationstar})}}{ =}v_{g_i}=f(g_i).$$
\item[(iv)] On prend
$$K_0= \cap_{k_1\in K_1} k_1^{-1} K k_1.$$
Comme $K_1K \subseteq E$, par la démonstration du point $(c)(ii)$, on voit que $K_0$ est un sous-groupe ouvert compact de $G$ et $K_0K_1=K_1K_0$. \\
D'abord, pour $k_1 \in K_1$, $k_0\in K_0$, on a $k_1k_0=k'_1k_0'$ pour $k_0' \in K_0, k_1' \in K_1$. On trouve
$$f(g_ik_1k_0)=f(g_ik_0'k_1')=f(g_ik_0')=f(g_i).$$
Si $g\in \Omega \setminus \{ g_1, \cdots, g_n\}$, on a
$$HgK_1 \subseteq HgK_1K_0 \subseteq HgK_0K_1 \subseteq HgE_g = HgK_1.$$
Donc, on a aussi
$$0=f(g)=f(gk_0) \textrm{ pour } g\in HgK_1 \textrm{ et } k_0\in K_0.$$
\end{itemize}
Finalement, on a affirmé l'assertion!\\
(3) \underline{L'isomorphisme.} On définit
 $$ \Res_{G_1}^G\big(c-\Ind_H^G W\big) \stackrel{\gamma}{\longrightarrow} c-\Ind_{H\cap G_1}^{G_1} W,$$
 $$f \longmapsto f|_{G_1}.$$
  Si $\supp(f) \subseteq HK$ pour un ensemble compact $K$ de $G$; comme $H_1\backslash G_1 \simeq H\backslash G$, on vérifie que $\supp(f|_{G_1}) \subseteq H_1K^{1}$ pour un ensemble compact $K^1$ de $G_1$. Donc $\gamma$ est bien définie, et on voit asiément ce qui est aussi injective.\\
 D'autre part, par le point (2) ci-dessus, on a vu que, pour chaque sous-groupe ouvert compact $K_1$ de $G_1$, $\gamma$ induit une bijection entre $\big( \Res_{G_1}^G(c-\Ind_H^G W)\big)^{K_1}$ et  $\big(c-\Ind_{H\cap G_1}^{G_1} W\big)^{K_1}$, d'où le résultat.
 \end{proof}
 \section{Appendice II. Les représentations admissibles des groupes des similitudes}\label{appendice2}
 Dans cette sous-section, nous voulons discuter la condition admissible pour les représentations des groupes de similitudes.\\

 On désigera par $F$ un corps local non-archimédien de caractéristique réduelle impaire $p$, et $A$ un groupe abélien d'ordre fini $n$, satisfaisant à $2|n$, et $(p,n)=1$.\\

 D'abords, nous rappelons des lemmes de Moore en homologie et montrons quelque résultats en scindage.
\\
Pour le corps $F^{\times}$, on désigne par  $\mathfrak{O}$ l'anneau d'entiers, $\mathfrak{P}$ l'idéal maximal de $\mathfrak{O}$, $U_n=\{ u \in F^{\times} | u\equiv 1 \mod \mathfrak{P}^{n} \}$, $U=U_0$,  $k_{F}$ sons corps résiduelle d'ordre $q=p^l$ avec $p\neq 2$,  $U_0/{U_1} \simeq k^{\times}$ un groupe cyclique d'ordre $q-1$. De plus, on peut prendre un sous-groupe de $S$ de $U$ tel que $U \simeq U_1 \times S$ [cf. \cite{Mo}, Page 20].
\begin{lemme}\label{tidtytyt}
On a l'isomorphisme naturel
 $$H^2(F^{\times}, A) \simeq \Hom(S, A).$$
\end{lemme}
\begin{proof}
Ceci découle de [\cite{Mo}, Page 20]. D'abord, $F^{\times}\simeq U \times \Z$, par la suite spectrale de Leray, on obtient
$$H^2(F^{\times}, A) \simeq H^2(\Z, A) \oplus H^1 (\Z, H^1 (U, A)) \oplus H^2(U, A).$$
Le premier terme est zéro et $H^1(\Z, H^1(U, A)) \simeq \Hom(U, A)$. Comme $U \simeq S \times U_1$, on a $\Hom(U,A) \simeq \Hom(S,A)$. En utilisant la suite de spectrale de Leray encore une fois, on obtient
$$H^2(U, A) \simeq H^2(S, A) \oplus H^1(S, H^1(U_1, A)) \oplus H^2(U_1, A).$$
Comme $S$ est cyclique, on a $H^2(S, A)=0$ et $ H^1(S, H^1(U_1, A)) \simeq \Hom(S, \Hom(U_1, A)) =0$. De plus $H^2(U_1,A)$ est un groupe de  $p$-torsion et de $n$-torsion, donc zéro. Finalement, on trouve le résultat.
\end{proof}
\begin{lemme}\label{scindesurF8}
Pour le sous-groupe $F^{\times n}$ de $F^{\times}$, l'application $H^2(F^{\times}, A) \longrightarrow H^2(F^{\times n}, A)$ est nulle.
\end{lemme}
\begin{proof}
C'est une consequence du Lemme \ref{tidtytyt}.
\end{proof}

Rappels:
Soient $D$ un corps, muni d'une involution $\tau$, $F$ le corps commutatif formé des points fixés de $\tau$, $W$ un espace vectoriel de dimension finie sur $D$, muni d'un produit $\epsilon$-hermitien $\langle, \rangle$( où $\epsilon=\pm 1$). On désigne par $U(W)$ le groupe unitaire de $(W, \langle, \rangle)$ et par $GU(W)$ le groupe de similitudes unitaire de $(W, \langle, \rangle)$.
 \  \\
On a une suite exacte
$$1 \longrightarrow U(W) \longrightarrow GU(W) \stackrel{\lambda}{\longrightarrow} \Lambda_{GU(W)} \longrightarrow 1.$$
Il  en résulte qu'on a des longues suites exactes
$$ \cdots \longrightarrow H^2(\Lambda_{GU(W)}, A) \longrightarrow H^2(GU(W), A) \stackrel{\lambda^2}{\longrightarrow} H^2(U(W), A) \longrightarrow  \cdots$$
Prenons un élément $[c]$ de $H^2(GU(W), A)$. Associons $[c]$( resp.  $\lambda^2([c])$) le groupe $\widetilde{GU}^A(W)$ (resp.  $\widetilde{U}^A(W)$). On a le diagramme commutatif suivant
\[
\begin{array}{ccccccccccc}
1 & \longrightarrow  &  A            & \longrightarrow     & \widetilde{U}^A(W)  & \longrightarrow & U(W)       & \longrightarrow  & 1       \\
  &                  &  \parallel    &                     &  \downarrow      &                 & \downarrow &                  &          \\
1 & \longrightarrow  &      A        & \longrightarrow     & \widetilde{GU}^A(W) & \longrightarrow & GU(W)      & \longrightarrow  & 1
\end{array}
\]
Par le lemme du serpent, on obtient le diagramme commutatif:
\[
\begin{array}{ccccccccccc}
  &                  &    1          &                     &      1           &                 &   1        &                  &    \\
  &                  &   \downarrow  &                     & \downarrow       &                 &  \downarrow&                  &     \\
1 & \longrightarrow  &  A        & \longrightarrow     & \widetilde{U}^A(W)  & \longrightarrow & U(W)       & \longrightarrow  & 1       \\
  &                  &  \parallel    &                     &  \downarrow      &                 & \downarrow &                  &          \\
1 & \longrightarrow  & A         & \longrightarrow     & \widetilde{GU}^A(W) & \longrightarrow & GU(W)      & \longrightarrow  & 1     \\
  &                  &  \downarrow   &                     & \downarrow       &                 & \downarrow &                  &       \\
1 & \longrightarrow  &     1       &\longrightarrow &\Lambda_{\widetilde{GU}^A(W)}    &  =         &\Lambda_{GU(W)}& \longrightarrow & 1  \\
  &                  &   \downarrow  &                     & \downarrow       &                 &  \downarrow&                  &     \\
 &                  &    1          &                     &      1           &                 &   1        &                  &
\end{array}
\]

\begin{lemme}\label{scindagedecorpsF}
Gr\^ace au monomorphisme $F^{\times} \hookrightarrow GU(W)$, on peut regarder $F^{\times}$ comme un sous-groupe de $GU(W)$, alors la suite exacte
$$1 \longrightarrow A \longrightarrow \widetilde{GU}^A(W) \longrightarrow GU(W) \longrightarrow 1 $$ est scindée au-dessus de $F^{\times n}$.
\end{lemme}
\begin{proof}
On note $\widetilde{F^{\times} U}^A(W)$ l'image réciproque de $F^{\times}U(W)$ dans $\widetilde{GU}^A(W)$.
Il suffit de montrer que chaque suite exacte
$$1 \longrightarrow A \longrightarrow  \widetilde{F^{\times} U}^A(W) \longrightarrow F^{\times}U(W) \longrightarrow 1$$
est scindée au-dessus de $F^{\times n}$. Comme $F^{\times}U(W) \simeq F^{\times } \times PU(W)$, où $PU(W)= U(W)/\{\pm 1\}$. Par la suite spectrale de Leray, on obtient
$$H^2\Big(F^{\times}PU(W), A\Big) \simeq H^2(PU(W), A) \oplus H^1(PU(W), H^1(F^{\times}, A)) \oplus H^2(F^{\times}, A)$$
$$ \simeq H^2(PU(W), A) \oplus \Big( \Hom(PU(W),  \Hom(F^{\times}, A)) \oplus \Hom(S, A)\Big),$$
où $S$ est le sous-groupe de $F^{\times}$ défini en Lemme \ref{tidtytyt}. Donc  l'application $H^2(F^{\times}U(W), A)\longrightarrow  H^2(F^{\times n}, A)$  est nulle.
\end{proof}
\begin{lemme}\label{FncommutateaUW}
D'après le Lemme \ref{scindagedecorpsF} ci-dessus, considérons $F^{\times n}$ comme un sous-groupe de $\widetilde{GU}^A(W)$, alors $ F^{\times n}$ commute à $\widetilde{U}^A(W)$.
\end{lemme}
\begin{proof}
Par la suite spectrale de Leray ci-dessus, on a vu que la restriction  du cocycle de $[c]$ à $F^{\times n} \times PU(W)$ est triviale, ceci implique le résultat.
\end{proof}
\begin{thm}
Soient $\widetilde{\pi}\in \Irr(\widetilde{GU}^A(W))$, $\widetilde{\sigma}\in \Irr(\widetilde{U}^A(W))$. Alors $\widetilde{\pi}, \widetilde{\sigma}$ sont admissibles.
\end{thm}
\begin{proof}
Ceci découle de [\cite{BernD}, Page 17, et Pages 25-32]
\end{proof}

\begin{thm}\label{restrictionadmissible}
Si $\widetilde{\pi}\in \Irr(\widetilde{GU}(W))$, alors $\widetilde{\pi}|_{\widetilde{U}(W)}$ est admissible.
\end{thm}
\begin{proof}
  D'abord $\widetilde{\pi}$ est admissible.  Par la théorie de nombre[cf. \cite{Neu}, Page 142, Corollary], $F^{\times}/{F^{\times 2n}}$ est un groupe abélien d'ordre fini. Comme $\widetilde{GU}^A(W)/{F^{\times n} \widetilde{U}^A(W)} \simeq F^{\times}/{F^{\times 2n}}$, on sait que  $\widetilde{\pi}|_{F^{\times n}\widetilde{U}^A(W)}$ est admissible, d'où le résultat.
 \end{proof}


\begin{thebibliography}{99}
\labelwidth=4em
\addtolength\leftskip{25pt}
\setlength\labelsep{0pt}
\addtolength\parskip{\smallskipamount}


\bibitem[Bar]{Bar2} L.BARTHEL,
 {\it Correspondance de Howe entre groupes de similitudes, thesis, University of Paris 7
(1989).}


\bibitem[B1]{BernD} J. BERNSTEIN,
 {\it  rédigé par P. Deligne, Le centre de Bernstein}, in ”Représentations des groupe
réductifs sur un corps local, Travaux en cours,” Hermann, Paris, 1984.

\bibitem[B2]{Bern} J. BERNSTEIN,
 {\it P-invariant distributions on GL(n) and the classification of unitary representations of GL(n) (non-Archimedean case)}, Lecture Notes in Math 943,  Springer-Verlag, 50-102 (1984).

\bibitem[B3]{Bern3} I.N. BERNSHTEIN,
 {\it All reductive $p$-adic groups are of type I}, Functional Anal. Appli $8$ (1974) 91-93.

\bibitem[BZ]{BernZ} I.N. BERNSTEIN, A.V.ZELEVINSKY,
 {\it Representations of the group $GL(n,F)$ where $F$ is a non-archimedean local field}, Russ.Math.Surv.31(3). 1-68(1976).

\bibitem[BH]{BushH} C.J.BUSHNELL, G.HENNIART,
 {\it The local langlands conjecture for $GL(2)$},
 Grundlehren der Mathematischen Wissenschaften, 335. Springer-Verlag, Berlin, 2006.

\bibitem[GT]{GT} W.T.GAN and W.TANTONO,
 {\it The local Langlands conjecture for GSp(4) II: the case of inner forms, preprint.} 



\bibitem[HM]{HanM} M. HANZER, G. MUI\`C,
{\it Parabolic induction and Jacquet functors for metaplectic groups},
  J.Algebra 323 (2010), 241–260.

\bibitem[HK]{HK} M.HARRIS, S.S.KUDLA,
{\it Arithmetic automorphic forms for the non-holomorphic discrete series of GSp(2)},
 Duke Math. J. 66 (1992), 59-121.


\bibitem[H]{Ho} R. HOWE,
{\it $\theta$-series and invariant theory, in automorphic forms, representations and $L$-functions},
  Proc. Symp. in Pure Math. XXXIII,AMS 1979, 275-286.


 \bibitem[KMRT]{KMRT} M-A. KNUS, A. A. MERKURJEV, M. ROST, and J-P. TIGNOL,
 {\it    The Book of Involutions}, Number 44 in American Mathematical Society Colloquium Publications. AmericanMathematical Society,
Providence, R.I., 1998. With a preface in French by J. Tits.


\bibitem[Kud]{Kud2} S. KUDLA,
 {\it  Notes on the local theta correspondence( lectures at the European School in Group Theory)}, preprint, available at http://www.math.utotonto.ca/~skudla/castle.pdf, 1996.


\bibitem[Mi]{Mi} A. MINGUEZ,
 {\it Correspondance de Howe explicite: paires duales de type II},
Ann. Scient. Ec. Norm. Sup.41, f.5, (2008), 715-739.

\bibitem[MVW]{MVW} C.MOEGLIN, M.F.VIGNERAS, J-L. WALDSPURGER,
{\it Correspondances de Howe sur un corps p-adique},
Lecture Notes in Math. Vol 1921, Springer-Verlag, New York, 1987.

\bibitem[Mo]{Mo} C.C. MOORE,
 {\it Group extensions of $p$-adic and adelic linear groups},
Publ. IHES 35 (1968), 5670.


\bibitem[N]{Neu} J.NEUKIRCH,
{\it Algebraic number theory},
Grundlehren der Mathematischen Wissenschaften, vol.322. Springer Berlin (1999).



\bibitem[PS]{PS} PIATESHI-SHAPIRO,I.I,
{\it Special automorphic forms on $PGSp_4$},
Arithmetic and Geometry, Vol.I, 309-325, Progr. Math,35, Birh\"auser Boston, Bostonn, MA, 1983.



\bibitem[Rena]{Rena} D.RENARD,
 {\it Représentations des groupes réductifs p-adiques}, Cours
Spécialisés 17, Société Math. de France 2010.

\bibitem[R]{Rob1} B.ROBERTS,
{\it The theta correspondence for similitudes},
Israel J. of Math. 94 (1996), 285-317.


\bibitem[Sc]{Scha} W.SCHARLAU,
{\it Quadratic and Hermitiens forms},
Springer-Verlag Grundlehren der mathematischen Wissenschaften 270 [1985].


\bibitem[Sh]{Shim} H.SHIMIZU,
{\it Theta-series and automorphic forms on $GL_2$},
J. Math.Soc.Japan 24 (1972), 638-683.


\bibitem[Wa1]{Wald2} J-L. WALDSPURGER,
 {\it Correspondance de Shimura},
 J. Math. Pures Appl.(9) 59 (1980), no.1, 1-132.


\bibitem[Wa2]{Wald} J-L. WALDSPURGER,
 {\it  Démonstration d'une conjecture de dualité de Howe dans le cas p-adique  $p \neq 2$},
 in Israel Math. Conf. Proc. vol.2 (1990), 267-324.

 \bibitem[Wang]{thesiswang} C-H. WANG,
 {\it Thèse, Université Paris-sud 11}.

\bibitem[W]{Weil} A.WEIL,
 {\it  Sur certains groupes d'opérateurs unitaires},
 Acta Mathematica 111 (1964), 143-211.

\end{thebibliography}
\end{document}